\documentclass[oneside]{amsart}

\usepackage{todonotes}
\usepackage[english]{babel}
\usepackage{amssymb}
\usepackage{amsmath}
\usepackage{amsfonts}
\usepackage{geometry}
\usepackage{bbm}
\usepackage{ntrees}
\usepackage[colorlinks=true]{hyperref}
\usepackage{tikz}
\usepackage{mathrsfs}
\usepackage[foot]{amsaddr}
\usepackage{multirow}
\usetikzlibrary{matrix,arrows,decorations.pathmorphing}
\usepackage{verbatim }
\usepackage{mathtools}
\usepackage{enumitem}
\usepackage[author={Sebastian}]{pdfcomment}
\newcommand{\myquad}[1][1]{\hspace*{#1em}\ignorespaces}
\newcounter{cprop}[section]

\newtheorem{theorem}[cprop]{Theorem}
\theoremstyle{plain}

\newtheorem{corollary}[cprop]{Corollary}

\newtheorem{lemma}[cprop]{Lemma}
\newtheorem{proposition}[cprop]{Proposition}

\newtheorem{assumption}[cprop]{Assumption}

\numberwithin{equation}{section}

\theoremstyle{definition}
\newtheorem{definition}[cprop]{Definition}

\theoremstyle{remark}
\newtheorem{remark}[cprop]{Remark}

\usepackage{graphicx}
\newcommand{\bigsubseteq}{\mathrel{\scalebox{1.2}{\(\subseteq\)}}}
\newcommand{\bigsetminus}{\mathrel{\scalebox{1.5}{\(\setminus\)}}}
\makeatletter
\def\namedlabel#1#2{\begingroup
	\def\@currentlabel{#2}%
	\label{#1}\endgroup
}
\makeatother

\def\R{\mathbb{R}}
\def\C{\mathbb{C}}
\def\N{\mathbb{N}}

\def\cD{\mathcal{D}}

\def\cL{\mathcal{L}}

\def\txtd{{\textnormal{d}}}

\def\Id{{\textnormal{Id}}}

\newcommand{\vertiii}[1]{{\left\vert\kern-0.25ex\left\vert\kern-0.25ex\left\vert #1 
		\right\vert\kern-0.25ex\right\vert\kern-0.25ex\right\vert}}
\numberwithin{equation}{section}

\begin{document}
	\title[An integrable bound for semilinear rough PDEs]{An integrable bound for semilinear rough partial differential equations  with unbounded diffusion coefficients}%\todo{(Maybe better to change the title?)}

	\author{Alexandra Blessing Neam\c{t}u}
	\address{Alexandra Blessing Neam\c{t}u \\
		Fachbereich Mathematik und Statistik, Universität Konstanz, Konstanz, Germany}
	\email{alexandra.blessing@uni-konstanz.de}

	\author{Mazyar Ghani Varzaneh}
	\address{Mazyar Ghani Varzaneh\\
		Fachbereich Mathematik und Statistik, Universität Konstanz, Konstanz, Germany}
	\email{mazyar.ghani-varzaneh@uni-konstanz.de}
	\begin{abstract}
		This work develops moment bounds for the controlled rough path norm of the solution of semilinear rough partial differential equations.~The novel aspects are two-fold: first we consider rough paths of low time regularity $\gamma\in(1/4,1/2)$ and second treat unbounded diffusion coefficients. To this aim we introduce a suitable notion of a controlled rough path according to a monotone scale of Banach spaces and innovative control functions. 
	\end{abstract}

	\maketitle
	
	\section{Introduction}
  %  \textcolor{red}{In two references, I changed Blessing to Blessing Neam\c{t}u as written}\textcolor{blue}{Thank you!}
	We investigate non-autonomous rough partial differential equations given by 
	\begin{align} \label{Main_Equation}
		\begin{cases}\txtd y_t = [A(t)y_t +F(t,y_t) ]~\txtd t + G(t,y_t)\circ\mathrm{d}\mathbf{X}_t\\
			y_0\in E_\alpha,
		\end{cases}
	\end{align}
	where $(E_\alpha)_{\alpha\in\R}$ is a family of Banach spaces, $\mathbf{X}$ is a rough path of regularity $\gamma\in(1/4,1/2)$ and the diffusion coefficient $G$ is linear and thus unbounded.~The precise assumptions on the coefficients $A$, $F$ and $G$ will be stated in Sections~\ref{sec:crp} and~\ref{sec:sol}.\\
	
	The main goal of this work is to obtain a-priori bounds for the moments of the control rough path norm of the mild solution of~\eqref{Main_Equation}. For similar results for rough paths of regularity $\gamma\in(1/3,1/2)$ and smooth, i.e.~three-times Fr\'echet differentiable with bounded derivatives and bounded diffusion coefficients \( G \), we refer to the recent works~\cite{GVR25, BNS24, BGS25}.~Here we go beyond the setting of these works decreasing the regularity of the driving rough path and relaxing the boundedness assumption on $G$. The techniques developed in the previous works~\cite{GVR25} and~\cite{BGS25} in order to obtain such integrable bounds are limited to bounded diffusion coefficients $G$. The treatment of the unbounded case was left open in~\cite{BGS25}. Another development in comparison to the results in~\cite{GVR25,BGS25} is given by the spatial loss of regularity of $G$. For rough paths of regularity $\gamma\in(1/3,1/2)$ it is well-known according to~\cite{GH19,GHT21} that $G$ is allowed to use spatial regularity $\sigma<\gamma$. 
	This is due to the fact that the stochastic convolution increases the spatial regularity of the corresponding rough path by a parameter $\sigma$ which is strictly less than the time regularity $\gamma$ of the rough path.~However, in order to ensure the integrability of the bound for the controlled rough path norm of the solution,~\cite{GVR25,BGS25} imposed that $\sigma\in[0,\frac{1-\gamma}{2})$. In this work we are able to overcome this limitation and treat the case $\sigma\in[0,\gamma)$ as well. \\
	
	We mention that the techniques developed in this work are also applicable to nonlinear diffusion coefficients $G$ that are further assumed to be ${N+1}$-times Fr\'echet differentiable with bounded derivatives. Here $N = \left\lfloor \frac{1}{\gamma} \right\rfloor$ represents the number of iterated integrals that have to be taken into account. For a better comprehension we only focus here on the linear case and work with the mild formulation of~\eqref{Main_Equation}. For well-posedness results of semilinear rough stochastic partial differential equations with linear diffusion coefficients based on a variational approach we refer to~\cite{Friz}. \\
	
	This manuscript is structured as follows. In Section~\ref{sec:crp} we  introduce the main concepts of controlled rough paths and control functions required in this setting. We specify that we consider \( (p, \gamma) \)-rough paths $\mathbf{X}$ for $p<\gamma$ and define suitable controlled rough paths according to $\mathbf{X}$.~Our definition is a natural generalization of the controlled rough paths according to a monotone scale of interpolation spaces introduced in~\cite{GHT21}. Here we incorporate in Definition~\ref{controlled} more iterated integrals and decrease the spatial regularity of the Gubinelli derivatives and remainders according to the time regularity $\gamma$ of the rough path $\mathbf{X}$.~The main innovative aspect in comparison to~\cite{GHT21} is the replacement of the corresponding H\"older norms of the rough input $\mathbf{X}$ and of the controlled rough path with suitable control functions.~These are necessary in order to establish the desired integrable bounds and represent a main technical ingredient of our work.~To our best knowledge, this is the first work that introduces controlled rough paths measuring their time regularity in suitable Banach spaces in terms of control functions instead of H\"older continuity.~Afterwards we construct a rough integral for such controlled rough paths against $\mathbf{X}$.~This is naturally achieved by means of a modified version of the Sewing lemma stated in Theorem~\ref{shdccsa}, keeping track of several parameters reflecting an interplay between the time regularity of $\mathbf{X}$, the loss of spatial regularity $\sigma<\gamma$, the number of iterated integrals and the parameter $p<\gamma$ arising in the definition of the $(p,\gamma)$-rough path chosen such that $\sigma<p<\gamma$ and of the control function.~As already mentioned, the main technical challenge is to incorporate the control functions in our arguments.   \\
	
	Section~\ref{sec:sol} establishes the well-posedness of~\eqref{Main_Equation} based on the properties of the rough integral introduced in Section~\ref{sec:crp}.~This can be achieved by a standard fixed-point argument in the space of controlled rough paths introduced in Section~\ref{sec:crp}. We provide the necessary steps required in order to set up this argument relying on the assumption that the diffusion coefficient $G$ is linear. As already mentioned, one could incorporate a nonlinear term $G$ which is $N+1$ Fr\'echet differentiable with bounded derivatives. In this case, our techniques combined with a standard fixed-point argument as in~\cite{GH19,GHT21} entail the local well-posedness of~\eqref{Main_Equation}. An additional boundedness assumption on the Fr\'echet derivatives up to order $N+1$ as in~\cite{HN22} will imply the global-in-time existence.~Such boundedness assumptions have been relaxed for rough differential equations in~\cite{Li}.~We refrain from providing the details of the fixed-point argument and focus instead on establishing moment bounds for~\eqref{Main_Equation}.\\% in this setting.  \\
	
	Section~\ref{sec:ibound} is inspired by the seminal paper~\cite{CLL13} that establishes the existence of moments of all order for the norm of  differential equations driven by Gaussian rough paths. We extend this result in Theorem~\ref{thm:ibound} to the in the infinite-dimensional case based on Borell's inequality and a detailed analysis of the iterated integrals arising in our setting. Moreover we provide suitable estimates for the translated path in terms of the control function introduced in Section~\ref{sec:crp}.~We emphasize that Theorem~\ref{thm:ibound} is the first result on integrable bounds for the norm of the mild solution of~\eqref{Main_Equation} in the infinite dimensional setting and in the range of regularity $\gamma\in(1/4,1/2)$. As stated in~\cite{GH19}, moment bounds for the norm of the solution of a rough partial differential equation, are challenging to obtain.~This aspect was addressed in~\cite{GVR25} for $\gamma\in(1/3,1/2)$ and under the limitations on $G$ and $\sigma$ mentioned above. \\
	
	We mention that such integrable bounds  turned out to be very useful for the analysis of the long-time behavior of~\eqref{Main_Equation} driven by rough paths of regularity $\gamma\in(1/3,1/2)$ by means of the random dynamical systems approach.~These were employed in~\cite{BNS24} to study attractors for similar problems, including semilinear partial differential equations with rough boundary noise.~Recently,~\cite{BGS25} additionally established the existence of moments of all order for the controlled rough path norm of the  Jacobian of~\eqref{Main_Equation}. This turned out to be a crucial step in the application of the multiplicative ergodic theorem that ensures the existence of Lyapunov exponents for~\eqref{Main_Equation}.~In conclusion, we believe that our approach can be applied for related dynamical aspects for~\eqref{Main_Equation} as considered in~\cite{GVR25, BNS24,BGS25}.

	%	We mention that our definitions on controlled rough paths~\ref{controlled} and~\ref{controlled2} can also be understood in more general settings, such as regularity structures or nilpotent \todo{add ref}structures. However,  
	
	\subsection*{Acknowledgements}
	\label{sec:acknowledgements}
	Alexandra Blessing and Mazyar Ghani Varzaneh acknowledge support from DFG CRC/TRR 388 {\em Rough Analysis, Stochastic Dynamics and Related Fields}, Project A06.

	%	\section{Preliminaries}

	% {\bf Coefficients}
	% \begin{itemize}
		% 	\item $A$ generates a  parabolic evolution family $(U(t,s))_{t\geq s}$ on a separable Banach space $\cB$. One can work with the monotone family of interpolation spaces $(\cB_\alpha)_{\alpha\in\mathbb{R}}$ satisfying for $t>s$ similar estimates as in the autonomous case, i.e. 
		% 	\begin{align}
			% 		\begin{split}
				% 			|(U(t,s)-\text{Id}) x|_{\alpha}&\lesssim |t-s|^\sigma |x|_{\alpha+\sigma}\\
				% 			|U(t,s)x|_{\alpha+\sigma}&\lesssim |t-s|^{-\sigma}|x|_\alpha.
				% 		\end{split}
			% 	\end{align}
		
		% 	%\item $F$ is time periodic
		% 	\item $\mathbf{W}=(W,\mathbb{W})$ is a $\gamma$-H\"older rough path (for the integrable bounds on the solution the assumption was verified for the fbm)
		% 	\item $G\in C^3_b$ is bounded or satisfies the assumption from~\cite{HN22}, i.e.  for $\theta\in\lbrace 0,\gamma,2\gamma\rbrace$ , $G \colon \mathcal{B}_{\alpha-\theta}\rightarrow\mathcal{B}_{\alpha-\theta-\eta}^{n}$ is  Fr\'echet differentiable function up to three times with bounded derivatives and $D_{.}G[G(.)]:\mathcal{B}_{\alpha-\theta}\rightarrow\mathcal{B}_{\alpha-\theta-\eta}$ is bounded.
		% 	\item $F:[0,T]\times \cB_\alpha \to \cB_{\alpha-\delta}$ for $\delta\in[0,1)$ is locally Lipschitz with linear growth in the second argument and periodic in time
		% 	%\item $f_2$ is time-periodic, or drop it. This could be another forcing 
		% \end{itemize}
	\section{Controlled rough paths and rough integrals}\label{sec:crp}
	\subsection{Parabolic evolution families}
	We consider a family $(E_\alpha)_{\alpha\in \R}$ of interpolation spaces, such that $E_\beta \hookrightarrow E_\alpha$ for $\alpha<\beta$ and the following interpolation inequality holds
	\begin{align}\label{interpolation:ineq}
		|x|^{\alpha_3-\alpha_1}_{E_{\alpha_2}} \lesssim |x|^{\alpha_3-\alpha_2}_{E_{\alpha_1}} |x|^{\alpha_2-\alpha_1}_{E_{\alpha_3}},
	\end{align}
	for $\alpha_1\leq \alpha_2\leq \alpha_3$, $x\in  E_{\alpha_3}$.  Furthermore, we make some standard assumptions about the linear part of the equation~\eqref{Main_Equation}, namely we assume that the non-autonomous operators $(A(t))_{t\in[0,T]}$ are sectorial and satisfy a suitable H\"older-continuity in time. These conditions are known as the Kato-Tanabe assumptions. We refer to~\cite{BGS25} for more details and further references about this topic.

	\begin{itemize}\namedlabel{ass:A}{\textbf{(A)}}
		\item[\textbf{(A1)}\namedlabel{ass:A1}{\textbf{(A1)}}] The family $(A(t))_{t\in [0,T]}$ consists of closed and densely defined operators $A(t):E_1\to E_0$ on a time independent domain $D(A)=E_1$. Furthermore, they have bounded imaginary powers, i.e. there exists $C>0$ such that
		\begin{align*}
			\sup_{|s|\leq 1} \|(-A(t))^{is}\|_{\mathcal{L}(D(A))}\leq C
		\end{align*}
		for every $t,s\in \R$, where $i$ denotes the imaginary unit. 
		\item[\textbf{(A2)}\namedlabel{ass:A2}{\textbf{(A2)}}] There exists $\vartheta\in (\pi,\frac{\pi}{2})$ and a constant $M$ such that $\Sigma_{\vartheta}:=\{z\in \C~:~|\arg(z)|<\vartheta\}\subset R(A(t))$ where $R(A(t))$ denotes the resolvent set of $A(t)$ and
		\begin{align*}
			\lVert (z-A(t))^{-1}\rVert_{\cL(E_0)}\leq \frac{M}{1+|z|},
		\end{align*}
		for all $z\in \Sigma_\vartheta$ and $t\in [0,T]$.
		\item[\textbf{(A3)}\namedlabel{ass:A3}{\textbf{(A3)}}] There exists a $\varrho\in(0,1]$ such that %\textcolor{blue}{changed $\nu$ in $\rho$, since $\nu$ appears everywhere in section 5. }
		\begin{align*}
			\lVert A(t)-A(s)\rVert_{\cL(E_1,E_0)}\lesssim |t-s|^\varrho,
		\end{align*}
		for all $s,t\in[0,T]$.
	\end{itemize}

	Under these assumptions, we obtain an evolution family which is the non-autonomous equivalent of a $C_0$-semigroup.
	\begin{theorem}
		Let $(A(t))_{t\in [0,T]}$ satisfy the Assumptions~\ref{ass:A1}--\ref{ass:A3}. Then there exists a unique parabolic evolution family $(U_{s,t})_{0\leq s\leq t\leq T}$ of linear operators $U_{t,s}:E_0\to E_0$ such that the following properties holds:

		\begin{itemize}
			\item[i)] For all $0\leq r\leq s\leq t\leq T$ we have
			\begin{align*}
				U_{t,s}U_{s,r}=U_{t,r}
			\end{align*}
			as well as $U_{t,t}=\Id_{E_0}$.
			\item[ii)] The mapping $(s,t)\mapsto U_{t,s}$ is strongly continuous. 
			\item[iii)] For $s\leq t$ we have the identity
			\begin{align*}
				\frac{\txtd}{\txtd t} U_{t,s}=A(t)U_{t,s}.
			\end{align*}
		\end{itemize}
	\end{theorem}
	From now on, we denote $E_\alpha:=D((-A(t))^\alpha)$ endowed with the norm $|\cdot|_{E_\alpha}:=\|(-A(t))^\alpha \cdot\|_{E_0}$ and assume $(A(t))_{t\in [0,T]}$ satisfies Assumption \ref{ass:A1}--\ref{ass:A3}  on $(E_\alpha,E_{\alpha+1})$ for every $\alpha>0$. Then the resulting evolution family satisfies for $t>s$ similar estimates as in the autonomous case, i.e. there exists constants $C_1, C_2>0$ such that 
	\begin{align}\label{regularity}
		\begin{split}
			|(U_{t,s}-\text{I}) x|_{E_\alpha}& \leq C_1 |t-s|^{\sigma_1} |x|_{E_{\alpha+\sigma_1}},\\
			|U_{t,s}x|_{E_{\alpha+\sigma_2}}& \leq C_2 |t-s|^{-\sigma_2}|x|_{E_\alpha},
		\end{split}
	\end{align}
	for $\sigma_2\in [k_-,k_+]$ and $\sigma_1\in [0,1]$, where $k_-<k_+$ are fixed natural numbers and the constant $C_1=C_1(\alpha,\sigma_1), C_2=C_2(\alpha,\sigma_2)$ may also depend on $k_-,k_+$. We mention that Assumption~\ref{ass:A1} is necessary in order to ensure that the spaces $E_\alpha$ are time-independent, see~\cite[Remark 2.5 i)]{BGS25}.
	
	\subsection{Notations}\label{NOTa}
	In this subsection we collect some notations and conventions which will be used throughout the manuscript. 
	\begin{enumerate}
		\item The symbol $\circ$ is commonly used to denote the composition. The notation $a \lesssim b$ signifies that there exists a constant $C$, depending only on parameters of no particular significance, such that $a \leq Cb$.
		%	\item The constants need not be the same, even if we use the same notation, unless explicitly specified otherwise. 
		\item We denote the norm of an arbitrary Banach space \( E \) by \( \vert \cdot \vert_E \). {If \( E \) is finite dimensional, we write \( \vert \cdot \vert \).}
		\item For three Banach spaces \( U \), \( V \) and \( W \), {where \( U \) and \( V \) are finite-dimensional}, we identify the space of linear operators from \( U \otimes V \) to \( W \),  denoted by \( L(U \otimes V, W) \), with \( L(U, L(V, W)) \).%\todo{does this identification work for arbitrary Banach spaces or we need them to be a Banach algebra? in our case, it works since we have finite-dimensional spaces $\R^d$ etc and $E_\alpha$ is infinite dim but i don't think that this is true for arbitrary Banach spaces, \textcolor{red}{ If we take the projective tensor product(which is the most natural one), it is  always true. I will make it more precise}}
        %\todo{my suggestion would be to write for e.g. something like that $L(\R^{d_1}, L(R^{d_2};E_\alpha)=L(R^{d_1}\otimes \R^{d_2};E_\alpha)$,\textcolor{red}{I simply assume that $U$ and $V$ are finite. even in the infinite case it is true by taking the Projective tensor algebra (It depends which norm we impose on the tensor product)}}
		\item For arbitrary \( N \geq 1 \) and \( d\geq 1\) we denote by \( T^{N}(\mathbb{R}^d) \)  the truncated tensor algebra of level \( N \), i.e.  
		\[
		T^{N}(\mathbb{R}^d) = \bigoplus_{k=0}^{N} (\mathbb{R}^d)^{\otimes k},
		\]
		where \( (\mathbb{R}^d)^{\otimes 0} = \mathbb{R} \) and \( (\mathbb{R}^d)^{\otimes 1} = \mathbb{R}^d \).  
		Furthermore, for \( 0 \leq k \leq N \), the map  
		\begin{align*}
			\Pi^{k} \colon T^{N}(\mathbb{R}^d) \to (\mathbb{R}^d)^{\otimes k}
		\end{align*}
		refers to the usual projection map.
		
		\item  By \( P[s,t] \) we mean the set of finite partitions of an arbitrary interval \([s,t]\) and
		\begin{align}
			\sup_{\pi \in P[s,t]} \left\{ \sum_k B_{\tau_k,\tau_{k+1}} \right\},
		\end{align}
		denotes 
		\begin{align}
			\sup_{\pi = \{s = \tau_0 < \tau_1 < \cdots < \tau_m = t\}} \left\{ \sum_{0\leq k<m} B_{\tau_k,\tau_{k+1}} \right\},
		\end{align}
		where \((B_{\tau_k,\tau_{k+1}})_{0 \leq k < m}\) is a sequence of positive values, which will be explicitly defined in the context. These values depend on \(\tau_k\) and \(\tau_{k+1}\), i.e.~two consecutive points of the partition \(\pi\in P[s,t]\).
		\item Recalling that $\gamma$ denotes the regularity of $\mathbf{X}$ and assuming that $\frac{1}{\gamma}\notin \N$, we set $N := \left\lfloor \frac{1}{\gamma} \right\rfloor$ and consider $1\leq j \leq N$. For a path \( \xi^{j} : [0,T] \rightarrow L\left((\mathbb{R}^d)^{\otimes j}, E_{\alpha-j\gamma-\sigma}\right) \) we write when it is clear from the context
		\begin{align*}
			\sup_{\tau \in [s,t]} \left| \xi^{j}_{\tau} \right|_{\alpha-j\gamma-\sigma} := \sup_{\tau \in [s,t]} \left| \xi^{j}_{\tau} \right|_{L\left((\mathbb{R}^d)^{\otimes j}, E_{\alpha-j\gamma-\sigma}\right)}.
		\end{align*}
	\end{enumerate} 
	The following basic lemma will frequently be used frequently throughout this manuscript.
	%\todo{I just added another inequality that is sometimes used and incorporated both into the same lemma.}
	
		\begin{lemma} \label{INM}
			Assume that \( p_1 > p_2 > 0 \). Then, for every (finite or infinite) sequence of real values \( (x_k)_{k \geq 0} \), the following inequality holds true due to the embedding of the spaces of sequences $l^{p_1}\hookrightarrow l^{p_2}$ for $p_1>p_2$ %\todo{  changed to $p_1$ and $p_2$, $p$ is the parameter in the rough path and $q$ the integrability one, will change also where this is applied in the computations}
			\begin{align}\label{85sd}
				\left(\sum_{k}\vert x_{k}\vert^{p_1}\right)^{\frac{1}{p_1}}\leq \left(\sum_{k}\vert x_{k}\vert^{p_2}\right)^{\frac{1}{p_2}}.
			\end{align}
			Moreover, for every \( M \in \mathbb{N} \), we have the inequality
			\begin{align}\label{785a}
				\left( \sum_{0 \leq k \leq M} |x_{k}| \right)^{p_1} \leq C(M,p_1) \sum_{0 \leq k \leq M} |x_{k}|^{p_1},
			\end{align}
			for a constant \( C(M,p_1) \) which depends only on \( M \) and \( p_1 \).
		\end{lemma}
	
	%	This is due to the embedding of the spaces of sequences $l^p\hookrightarrow l^q$ for $p>q$.
	\subsection{Rough path theory}
	In this subsection, we provide an overview of rough path theory. For technical reasons, our presentation slightly differs from the literature. % particularly in the definition of the underlying metrics.
	In order to fix the ideas we first recall the setting of~\cite{GHT21} of controlled rough paths tailored to parabolic rough PDEs. For $d\geq 1$ we consider a $d$-dimensional $\gamma$-H\"older rough path $\textbf{X}=(X,\mathbb{X})$, for $\gamma\in(1/3,1/2)$ with $X_0=0$. 
	More precisely, we have for $T>0$
	\begin{align*}
		X\in C^{\gamma}([0,T];\mathbb{R}^d) ~~\mbox{ and } ~~ \mathbb{X}\in C_2^{2\gamma}(\Delta_{[0,T]};\mathbb{R}^d\otimes\mathbb{R}^d),
	\end{align*}
	where $\Delta_{J}\coloneqq\{(s,t)\in J\times J~:~s\leq t\}$ for $J\subset \R$ and the connection between $X$ and $\mathbb{X}$ is given by Chen's relation
	\begin{align*}
		\mathbb{X}_{s,t}- \mathbb{X}_{s,u}-\mathbb{X}_{u,t}=(\delta X)_{s,u}\otimes (\delta X)_{u,t},
	\end{align*}
	for $s\leq u\leq t$, where we write $(\delta X)_{s,u}:=X_u-X_s$ for an arbitrary path. Here, we denote by $C^\gamma([0,T])$ the space of $\gamma$-H\"older continuous paths, by $C^{2\gamma}_2(\Delta_T)$ the space of $2\gamma$-Hölder continuous two-parameter functions and by $\mathscr{C}^\gamma([0,T])$ the space of $\gamma$-H\"older rough paths $\mathbf{X}=(X,\mathbb{X})$. In this setting, one can define the notion of a controlled rough path according to such a family of function spaces, as introduced in \cite{GHT21}. We omit the time dependence if it is clear from the context, meaning that we write $C^\gamma(E_\alpha)=C^\gamma([0,T];E_\alpha)$.
	\begin{definition}\label{def:crp}
		Let $\alpha\in\R$. We call a pair $(y,y')$ a controlled rough path if $(y,y')\in C(E_\alpha) \times (C(E_{\alpha-\gamma} ) \cap C^{\gamma}(E_{\alpha-2\gamma}))^d$ and the remainder  
		\begin{align}\label{remainder}
			(s,t)\in \Delta_{[0,T]}\mapsto R^y_{s,t}:= (\delta y)_{s,t} -y'_s \circ (\delta X)_{s,t}
		\end{align}%\todo{Throughout the text, the composition is mostly denoted by \( \circ \). We should consistently choose either \( \circ \) or \( \cdot \), as this is largely a matter of taste.}\todo{agree}
		belongs to $ C^{\gamma}(E_{\alpha-\gamma})\cap C^{2\gamma}(E_{\alpha-2\gamma})$, where $y'_s \circ (\delta X)_{s,t}=\sum_{i=1}^d y_s^{i,\prime}(\delta X^{i})_{s,t}$. The component $y'$ is referred to as Gubinelli derivative of $y$. 
		The space of controlled rough paths is denoted by $\mathscr{D}^{\gamma}_{\mathbf{X},\alpha}([0,T])$ and endowed with the norm $\|\cdot,\cdot\|_{\mathscr{D}^{\gamma}_{\mathbf{X},\alpha}([0,T])}$ given by
		\begin{align}\label{g:norm}
			\begin{split}
				\|y,y'\|_{\mathscr{D}^{\gamma}_{\mathbf{X},\alpha}([0,T])}:= \left\|y \right\|_{\infty,E_\alpha} 
				+ \|y' \|_{\infty,E^d_{\alpha-\gamma}}
				+ \left[y'\right]_{\gamma,E^d_{\alpha-2\gamma}}
				+\left[R^y \right]_{\gamma,E_{\alpha-\gamma}}    + \left[R^y \right]_{2\gamma,E_{\alpha-2\gamma}}.
			\end{split}
		\end{align}
	\end{definition}
	Here we use for $y^\prime=(y^{i,\prime})_{1\leq i\leq d}$ the notation $|y'_s|_{E_\alpha^d} \coloneqq \sup\limits_{1\leq i\leq d}|y_s^{i,\prime}|_{\alpha}$. \\%We omit the time dependence and write $\mathcal{D}^\gamma_{\mathbf{X},\alpha}([0,T])=\mathcal{D}^\gamma_{\mathbf{X},\alpha}$ if the time interval is clear from the context.\\
	
	In our case we need the following generalizations. First, in order to obtain integrable bounds we have to replace the H\"older norms of the random input by suitable control functions. Therefore we leave the framework of $\gamma$-H\"older rough paths and study $(p,\gamma)$-rough paths as given in Definition~\ref{pgamma:rp}. 
	Second, since we allow paths of regularity $\gamma\in(1/4,1/3)$ and deal with higher order iterated integrals, we have to decrease the spatial regularity of the controlled rough path in $E_\alpha$ accordingly. This is specified in Definition~\ref{controlled} for paths taking values in $E_{\alpha-\sigma}$ for $0\leq \sigma<\gamma$ and in Definition~\ref{controlled2} for paths taking values in $E_\alpha$. 
	\begin{definition}\label{pgamma:rp}
		We assume that $d\geq 1$, \( 0 \leq p < \gamma < \frac{1}{2} \) and set \( N = \left\lfloor \frac{1}{\gamma} \right\rfloor \). Since \( \frac{1}{\gamma} \) is not an integer we have that \( N\gamma < 1 \). We further let \( \mathbf{X}: \Delta_{[0,T]} \longrightarrow T^{N}(\mathbb{R}^d) \) be a multiplicative functional such that for every \( 1 \leq j \leq N \) we have 
		\begin{align}\label{AASdfd}
			\begin{split}
				&	W_{\Pi^{j}(\mathbf{X}),\gamma,p} \colon \Delta_{[0,T]} \rightarrow \mathbb{R}, \\
				&	W_{\Pi^{j}(\mathbf{X}),\gamma,p}(s,t) \coloneqq 	\sup_{\pi \in P[s,t]}  \left\{ \sum_{k}\frac{\left\vert(\Pi^{j}(\mathbf{X}))_{\tau_{k},\tau_{k+1}}\right\vert^{\frac{1}{j(\gamma-p)}}}{(\tau_{k+1}-\tau_k)^{\frac{p}{\gamma-p}}} \right\}< \infty.
			\end{split}
		\end{align}
		In this case, we say that \( \mathbf{X} \) is a \( (p, \gamma) \)-{rough path}. We also set \( W_{\mathbf{X}, \gamma, p}(s, t) := \sum_{1 \leq j \leq N} W_{\Pi^{j}(\mathbf{X}), \gamma, p}(s, t) \). From the definition, we can easily verify that $W$ is a control function, i.e., for \( s < u < t \), we have the subadditivity property
		\[
		W_{\mathbf{X}, \gamma, p}(s, u) + W_{\mathbf{X}, \gamma, p}(u, t) \leq W_{\mathbf{X}, \gamma, p}(s, t).
		\]
		We say that our \( (p, \gamma) \)-rough path is continuous if \( W_{\mathbf{X}, \gamma, p} \), is a continuous function. We also use \( \mathscr{C}^{\gamma, p}([0,T]) \) to denote the space of \( (p, \gamma) \)-rough paths. We also refer to \( \mathbf{X} \) as a weakly geometric \( (p, \gamma) \)-rough path if in addition \( \mathbf{X} \) takes values in the free nilpotent group of step \( N \) over \( \mathbb{R}^d \). For a similar definition we refer to \cite[Definition 9.15]{FV10}.
	\end{definition}
	We can prove the following simple lemma. %\todo{what happens in (3.5)? In (3.6) an $\leq$ is missing in the second line.}\todo{Thanks, the first formula was superfluous. The 3.6 (Now 3.5) is corrected}
	\begin{lemma}\label{sudj14}
		Assume that \( \mathbf{X} \) is a \( (p, \gamma) \)-rough path. Then \( \mathbf{X} \) is also a \( (0, \gamma) \)-rough path. Moreover, the following inequality for the corresponding control functions holds
		\begin{align*}
			W_{\mathbf{X}, \gamma, 0}^{\gamma}(s,t) \leq (t-s)^{p}W_{\mathbf{X}, \gamma, p}^{\gamma-p}(s,t).
		\end{align*}
	\end{lemma}
	\begin{proof}
		We recall \eqref{AASdfd} and assume that $\pi=\lbrace \tau_{k}\rbrace_{0\leq k \leq m}\in P[s,t]$. Then for \(1 \leq j \leq N\), it follows using Hölder's inequality that
		\begin{align}
			\begin{split}
				&\sum_{k}\left\vert(\Pi^{j}(\mathbf{X}))_{\tau_{k},\tau_{k+1}}\right\vert^{\frac{1}{j\gamma}}=\sum_{k}(\tau_{k+1}-\tau_k)^{\frac{p}{\gamma}}\frac{\left\vert(\Pi^{j}(\mathbf{X}))_{\tau_{k},\tau_{k+1}}\right\vert^{\frac{1}{j\gamma}}}{(\tau_{k+1}-\tau_k)^{\frac{p}{\gamma}}}\\&\quad \leq\left(\sum_{k}\left( \frac{\left\vert(\Pi^{j}(\mathbf{X}))_{\tau_{k},\tau_{k+1}}\right\vert^{\frac{1}{j\gamma}}}{(\tau_{k+1}-\tau_k)^{\frac{p}{\gamma}}}\right)^{\frac{\gamma}{\gamma-p}}\right)^{\frac{\gamma-p}{\gamma}}\left(\sum_{k}(\tau_{k+1}-\tau_k)\right)^{\frac{p}{\gamma}}\leq (t-s)^{\frac{p}{\gamma}} W_{\Pi^{j}(\mathbf{X}), \gamma, p}^{\frac{\gamma-p}{\gamma}}(s,t).
			\end{split}
		\end{align}
		From this point onward, the rest is straightforward.
	\end{proof}
	Now, given a \( (p, \gamma) \)-rough path $\mathbf{X}$, we naturally expect to define the spaces of controlled paths with respect to $\mathbf{X}$ taking values in $E_{\alpha-\sigma}$ for $0\leq \sigma<\gamma$ in the spirit of Definition~\ref{def:crp}. The next definition is a natural generalization of Definition~\ref{def:crp} where $N=2$ and $X\in \mathscr{C}^\gamma([0,T])$. Moreover, the H\"older regularities of the Gubinelli derivative and of the remainder specified in Definition~\ref{def:crp} are now quantified in terms of suitable control functions. We emphasize that the spatial regularity of the Gubinelli derivatives and remainders is successively decreased according to the time regularity $\gamma$ of the rough path $\mathbf{X}$.

	%\todo{comment about the regularity reference to definition where we have it in $E_\alpha$}
	%In the following definitions, we provide this definition.
%\todo{\textcolor{red}{Please see Remark \ref{asasR} . One could argue that at the beginning, this separation was not a good idea; unfortunately, now, unifications really cause a lot of changes. I hope you will be satisfied with this remark. I can still add more details to this remark.}}
    \begin{definition}\label{controlled}
		Let \( \mathbf{X} \in \mathscr{C}^{\gamma, p}([0,T]) \) and \( 0 \leq \sigma < \gamma \). We say that \( \boldsymbol{\xi} = (\xi^{i})_{0 \leq i < N} \) is controlled by \( \mathbf{X} \) in $\left(L\left((\R^d)^{\otimes j+1},E_{\alpha-j\gamma-\sigma}\right)\right)_{0\leq j<N}$ if we have for every \( 0 \leq i\leq j < N \) that %\todo{simplify above as in Definition~\ref{controlled2} to be consistent. \textcolor{red}{OK, Now I got the point}}
		%        \todo{ why $(\R^d)^{\oplus j+1}$ below and not $(\R^d)^{\oplus j}$ if we start in $E_{\alpha-j\gamma-\sigma}$ and not in $L(\R^d, E_{\alpha-j\gamma-\sigma})$?}
		\begin{comment}
			
			\textcolor{blue}{In def~\ref{controlled2} we have
				\begin{align*}
					&\tilde{\xi}^j: [0,T] \to L\left((\R^d)^{\otimes j},E_{\alpha-j\gamma}\right)=L\left((\R^d)^{\otimes j-i},L\left((\R^d)^{\otimes i},E_{\alpha-j\gamma}\right)\right),
				\end{align*}
				so the tensor spaces are different. we have to say why. We have here $\xi$ controlled in $L(\R^d, E_{\alpha-j\gamma-\sigma})$ and in def 2.14. there $\tilde{\xi}$ controlled in $E_{\alpha-j\gamma-\sigma}$. 
			}
		\end{comment}
		
		%	\todo{ make a remark that one could also define as in~\cite{GH19}
			%		$L(R^d, E_{\alpha-j\gamma})=E^d_{\alpha-j\gamma}$ but we stick to the tensor notation, \textcolor{red}{see the remark}}
		%\todo{ please check the range of $\xi^j$ here and in Def~\ref{controlled2}. All the indices must fit to Def~\ref{controlled2}. There we have for $\tilde{\xi}$ $(\R^d)^{\oplus j}$ instead of $(\R^d)^{\oplus j+1}$ as here. }
       % \todo{i saw the remark but still think we should unify starting with def 2.15 and everything afterwards. i know this requires lots of changes but if we don't do it now a referee will complain a lot. my suggestion is to let the separation but unify the dimension.}
		\begin{align*}
			&\xi^j: [0,T] \to L\left((\R^d)^{\otimes j+1},E_{\alpha-j\gamma-\sigma}\right)=L\left((\R^d)^{\otimes j-i},L\left((\R^d)^{\otimes i+1}, E_{\alpha-j\gamma-\sigma})\right)\right),
		\end{align*}
		where we use the identification specified in Notation~\ref{NOTa} (3). 
		%\todo{ add that we identify the spaces as stated in Notation 2.2. (3)}
		%         \todo{ i reformulated for every \( [s,t] \subseteq [0,T] \) since (1) does not depend on it }
		For every \( 0\leq i, l \leq N \) with \( l - i > 1 \) 
		\begin{align}\label{xii}
			\delta \xi^{i}_{s,t}=\sum_{i<j<l}\xi^{j}_s\circ (\Pi^{j-i}(\mathbf{X}))_{s,t}+R^{i,l}_{s,t}
		\end{align}
		such that the following properties hold: %\todo{should we drop the tensor notation below? we do it anyway later. so just write where $R$ and $\xi$ are and then drop the tensor from the norm. and have to write the summation index $i$ in (3.9)}
		\begin{enumerate}
			\item $\sup_{\tau\in [0,T]}\left\vert\xi^{i}_{\tau}\right\vert_{L\left((\R^d)^{\otimes i+1},E_{\alpha-i\gamma-\sigma}\right)}<\infty$ for every $0\leq i<N$.
			\item For every \( 0\leq i, l \leq N \) with \( l - i > 1 \) we assume for every \( [s,t] \subseteq [0,T] \) that
			\begin{align*}
				&W_{R^{i,l},\alpha-\sigma,\gamma,p,1}(s,t):=	\sup_{\pi \in P[s,t]} \left\{\sum_{k} \frac{\left|R^{i,l}_{\tau_{k},\tau_{k+1}}\right|_{L\left((\R^d)^{\otimes i+1},E_{\alpha-(l-1)\gamma-\sigma}\right)}^{\frac{1}{(l-i-1)(\gamma-p)}}}{(\tau_{k+1}-\tau_k)^{\frac{p}{\gamma-p}}} \right\}<\infty,
				\\&W_{R^{i,l},\alpha-\sigma,\gamma,p,2}(s,t):=	\sup_{\pi \in P[s,t]} \left\{\sum_{k} \frac{\left|R^{i,l}_{\tau_{k},\tau_{k+1}}\right|_{L\left((\R^d)^{\otimes i+1},E_{\alpha-l\gamma-\sigma}\right)}^{\frac{1}{(l-i)(\gamma-p)}}}{(\tau_{k+1}-\tau_k)^{\frac{p}{\gamma-p}}} \right\}<\infty.
			\end{align*}
			This means that the previous terms are assumed to be control functions reflecting the regularity of the remainders of the controlled rough path, i.e.~generalizing  the last two terms in~\eqref{g:norm} in our setting.
            Moreover, in the notation above we emphasize the dependence of the two control functions $W_{\cdot,1}$ and $W_{\cdot,2}$ of the corresponding remainders $R^{i,l}$ on the spatial regularity of the controlled rough path $\alpha-\sigma$, time regularity $\gamma$ of $\mathbf{X}$ and the parameter $p<\gamma$. 
			\item  For every \( 0 \leq i < N \) and every \( [s,t] \subseteq [0,T] \) we define \( R^{i,i+1}_{s,t} := \delta \xi^i_{s,t} \) and assume that
			\begin{align*}
				W_{R^{i,i+1},\alpha-\sigma,\gamma,p}(s,t):=	\sup_{\pi \in P[s,t]} \left\{\sum_{k} \frac{\left|  \delta \xi^i_{\tau_{k},\tau_{k+1}}\right|_{L\left((\R^d)^{\otimes i+1},E_{\alpha-(i+1)\gamma-\sigma}\right)}^{\frac{1}{\gamma-p}}}{(\tau_{k+1}-\tau_k)^{\frac{p}{\gamma-p}}} \right\}<\infty.
			\end{align*}
			For technical reasons, we adopt the convention that for  \( i = N-1 \) 
			\begin{align}\label{NNN}
				W_{R^{N-1,N},\alpha-\sigma,\gamma,p,2}(s,t):= W_{R^{N-1,N},\alpha-\sigma,\gamma,p}(s,t),
			\end{align}
			which is a control functions based on our previous assumptions.
			
			%We also sometimes use \( R^{N-1,N}_{s,t} := \delta \xi^{N-1}_{s,t} \), as a matter of convention.
			%\item We assume \( W_{R^{i,l},\alpha-\sigma, \gamma, p, 1} \), \( W_{R^{i,l},\alpha-\sigma, \gamma, p, 2} \), and \( W_{R^{i-1,i},\alpha-\sigma,\gamma,p}(s,t)\) are all control functions. 
			% \todo{ in 2) we already assume that for the first two. we should also state that one can verify this assumption. \textcolor{red}{True, the (finiteness) assumption yields that these are controller functions, and we do not need to assume that. I reformulated the statement a bit}}
		\end{enumerate}
		We denote by \( \mathscr{D}^{\gamma, p}_{\mathbf{X},\alpha-\sigma}\left([0,T]\right) \) the space of all such elements equipped with the following norm
		\begin{align}\label{YHN123}
			\begin{split}
				&\Vert\boldsymbol{\xi}\Vert_{ \mathscr{D}^{\gamma, p}_{\mathbf{X},\alpha-\sigma}\left([0,T]\right)}:=\sum W_{R^{i,i+1},\alpha-\sigma,\gamma,p}^{\gamma-p}(0,T)\\&+\quad\sum\left(\sup_{\tau\in [0,T]}\left\vert\xi^{i}_{\tau}\right\vert_{L\left((\R^d)^{\otimes i+1},E_{\alpha-i\gamma-\sigma}\right)}+W_{R^{i,l},\alpha-\sigma,\gamma,p,1}^{(l-i-1)(\gamma-p)}(0,T)+W_{R^{i,l},\alpha-\sigma,\gamma,p,2}^{(l-i)(\gamma-p)}(0,T)\right),
				%\sum W_{R^{i-1,i},\alpha-\sigma,\gamma,p}(s,t)\\&\quad
			\end{split}
		\end{align}
		where the summations are taken over all indices $i$ and $l$ in the range specified above, i.e. according to the restrictions on \( i \) and \( l \) that we assumed in (1)--(3). %\todo{ maybe we should give here more details  \textcolor{red}{enough?}}
		%\todo{ add a remark why this is a norm\ \ \textcolor{red}{See item 3}}
	\end{definition}
	\begin{remark}
		\begin{enumerate}
			\item 
			A controlled rough path  \( \boldsymbol{\xi} \in \mathscr{D}^{\gamma, p}_{\mathbf{X},\alpha-\sigma}\left([0,T]\right) \) according to $\mathbf{X}$ consists of a sequence \( (\xi^i)_{0 \leq i < N} \), where the component \( \xi^i \) is commonly referred to as the \( i \)-th Gubinelli derivative. Note that~\eqref{xii} is the generalization of~\eqref{remainder} in our case.
			In comparison to Definition~\ref{def:crp} we assume here that the path component $\xi^0$ takes values in $E_{\alpha-\sigma}$ for $0\leq \sigma<\gamma$. We consider in Definition~\ref{controlled2} a slightly modified space of controlled rough paths \( \mathscr{\tilde{D}}^{\gamma, p}_{\mathbf{X},\alpha}\left([0,T]\right) \) where the corresponding path component $\tilde{\xi}^0$ takes values in $E_\alpha$ and construct the solution of~\eqref{Main_Equation} as an element of this space.
            %\footnote{More precisely, for technical reasons, we use \( \tilde{\mathscr{D}}^{\gamma, p}_{\mathbf{X}, \alpha} \left( [0, T] \right) \) in this definition.} 
			%For clarity we decided to provide two separate definitions keeping a better track of the spatial regularity of the controlled rough paths. 
			\item The reason why deal with two different spatial regularities $\alpha-\sigma$ and $\alpha$ is that the stochastic convolution improves the spatial regularity of the controlled rough paths by a parameter $\sigma$ which is strictly less than the H\"older regularity of the random input $\gamma$. This is known for $\gamma\in(1/3,1/2)$ from~\cite[Corollary 4.6]{GHT21} and can be recovered in our setting in Proposition~\ref{UJMMA}.  
			\item The definition of the control functions together with the Minkowski inequality yield that  \eqref{YHN123} defines a norm on the space \( \mathscr{D}^{\gamma, p}_{\mathbf{X},\alpha-\sigma}\left([0,T]\right) \). In conclusion \( \mathscr{D}^{\gamma, p}_{\mathbf{X},\alpha-\sigma}\left([0,T]\right) \) endowed with the norm~\eqref{YHN123} is a Banach space.
			\item We decided to emphasize the dependence of the path, Gubinelli derivatives and remainders on the tensor product spaces in Definition~\ref{controlled} to keep track of the corresponding index of the tensor product and of the index with which we decrease the spatial regularity. Alternatively one could set for e.g.~\( E_{\alpha-j\gamma-\sigma}^d= L(\R^d, E_{\alpha-j\gamma-\sigma})  \) as in~\cite{GH19} and accordingly for the other elements in the space of controlled rough paths.
			\begin{comment}
				
				\textcolor{blue}{i reformulated a bit, i already wrote at the end at page 6 that this is a generalization of~\cite{GHT21} for $N=2$. it is not of the controlled rp of~\cite{GH19} because there they incorporate the semigroup so we have to be careful with the citations. only the notation $E^d_{\alpha-j\gamma-\sigma}$ is from~\cite{GH19}.}
			\end{comment}
			
		\end{enumerate}
		
	\end{remark}
	\subsection{Sewing lemma}
	We naturally expect to define a rough integral for every \( \boldsymbol{\xi} \in \mathscr{D}^{\gamma, p}_{\mathbf{X},\alpha-\sigma}\left([0,T]\right) \) against \( \mathbf{X} \). The main idea is to apply a variant of the sewing lemma. The main technical aspect is to incorporate the definition of the controlled rough path and the control functions in the arguments of the sewing lemma. Before stating this result, let us first begin with some algebraic identities and introduce further notations.% that will make our future expressions.% more accessible and structured.
	\begin{definition}\label{NOTTR}
		We consider an arbitrary time interval $[s,t]\subseteq [0,T]$. For $m\geq 1$ we set
		\begin{align*}
			\mathcal{D}_m[s,t] \coloneqq \left\{ s + \frac{2n + 1}{2^m}(t-s) \, :\, n = 0,\ldots, 2^{m-1} - 1 \right\}.% \quad m \geq 1.
		\end{align*}
		For $\tau \in \mathcal{D}_m[s,t]$ with $\tau = s + \frac{2n + 1}{2^m}(t-s)$ we define
		\begin{align*}
			\tau_{-} \coloneqq s + \frac{2n}{2^m}(t-s) \quad \text{and} \quad \tau_{+} \coloneqq s + \frac{2n + 2}{2^m}(t-s).
		\end{align*}
	\end{definition}
	\begin{lemma}\label{increment}
		We assume that $[\tau_1,\tau_3]\subseteq [s,t]$, let $\boldsymbol{\xi} \in \mathscr{D}^{\gamma, p}_{\mathbf{X},\alpha-\sigma}\left([0,T]\right)$ and set
		\begin{align*}
			\Xi^{\tau_1,\tau_3}_{s,t} \coloneqq \sum_{0\leq j< N}U_{t,\tau_1}\xi^{j}_{\tau_1}\circ (\Pi^{j+1}(\mathbf{X}))_{\tau_1,\tau_3}.
		\end{align*} 
		Then we have for \( \tau_1 < \tau_2 < \tau_3 \) that
		\begin{align}\label{DIFFF}
			\begin{split}
				\Xi^{\tau_1,\tau_2}_{s,t}+\Xi^{\tau_2,\tau_3}_{s,t}-\Xi^{\tau_1,\tau_3}_{s,t} &= \sum_{0\leq j<N}U_{t,\tau_1} R^{j,N}_{\tau_1,\tau_2}\circ (\Pi^{j+1}(\mathbf{X}))_{\tau_2,\tau_3} \\&+\sum_{0\leq j<N}(U_{t,\tau_2}-U_{t,\tau_1})\xi^{j}_{\tau_2}\circ(\Pi^{j+1}(\mathbf{X}))_{\tau_2,\tau_3}.
			\end{split}
		\end{align}
	\end{lemma}
	\begin{proof}
		The proof is rather algebraic. First, we note that by the definition of \( \boldsymbol{\xi} \), in particular \eqref{xii} and the multiplicity of \( \mathbf{X} \), the following identity holds true
		\begin{align}\label{TGG}
			\begin{split}
				&\sum_{0\leq j<N}\left(\xi^{j}_{\tau_1}\circ (\Pi^{j+1}(\mathbf{X}))_{\tau_1,\tau_2}+\xi^{j}_{\tau_2}\circ (\Pi^{j+1}(\mathbf{X}))_{\tau_2,\tau_3}-\xi^{j}_{\tau_1}\circ (\Pi^{j+1}(\mathbf{X}))_{\tau_1,\tau_3}\right)\\&=\sum_{0\leq j<N} R^{j,N}_{\tau_1,\tau_2}\circ (\Pi^{j+1}(\mathbf{X}))_{\tau_2,\tau_3}.
			\end{split}
		\end{align}
		By definition
		\begin{align*}
			&\Xi^{\tau_1,\tau_2}_{s,t}+\Xi^{\tau_2,\tau_3}_{s,t}-\Xi^{\tau_1,\tau_3}_{s,t}\\&=\sum_{0\leq j<N}U_{t,\tau_1}\left(\xi^{j}_{\tau_1}\circ (\Pi^{j+1}(\mathbf{X}))_{\tau_1,\tau_2}+\xi^{j}_{\tau_2}\circ (\Pi^{j+1}(\mathbf{X}))_{\tau_2,\tau_3}-\xi^{j}_{\tau_1}\circ (\Pi^{j+1}(\mathbf{X}))_{\tau_1,\tau_3}\right)
			\\&+\sum_{0\leq j<N}(U_{t,\tau_2}-U_{t,\tau_1})\xi^{j}_{\tau_2}\circ(\Pi^{j+1}(\mathbf{X}))_{\tau_2,\tau_3}.
		\end{align*}
		Therefore, it is sufficient to apply \eqref{TGG} to prove the claim.
	\end{proof}
	We are ready to define the integral, which is a first step in developing the solution theory of~\eqref{Main_Equation}. The main idea is to modify the sewing lemma accordingly in order to incorporate the control functions specified in Definition~\ref{controlled}.~This is technically challenging since we cannot work with the H\"older norms of the rough input or of the corresponding controlled rough path, since they prevent us from obtaining a final integrable bound for the norm of the solution of~\eqref{Main_Equation}. \\
	
	We first make the following assumption reflecting an interplay between the regularity of the noise $\gamma$, the loss of spatial regularity $\sigma<\gamma$, the number of iterated integrals $N$ and the parameter $p<\gamma$ arising in the definition of the control functions. In particular, the following assumption will imply that $\sigma<p<\gamma$.
	
	% This modification is necessary because, in the infinite-dimensional case, due to the presence of the semigroup, we cannot neglect the role of the Hölder norms. Otherwise, it would be impossible to define the integral.
	%	However, the Hölder norms lead to an increase in the prior bounds, which, in turn, makes it impossible to establish integrable  priori bounds. Therefore, our definitions are designed to minimize the Hölder norms to levels where the integration is well-defined. 

	\begin{assumption}\label{ADssdfg}
		In the following we always assume that  
		\begin{align}\label{AASdf}  
			\frac{\sigma + N\gamma}{N+1} < p < \gamma.  
		\end{align}  
		For \(0 \leq i \leq N\) we further define 
		\begin{align*}  
			P_i := (i+1)p - i\gamma - \sigma,  
		\end{align*}  
		which is always positive due to   \eqref{AASdf}. Moreover, condition \eqref{AASdf} obviously yields that \( \sigma <p \). %\todo{ $\sigma$ must be much smaller than $p$ for (2.15) to hold we have $\sigma +N\gamma<Np +p$ but $p<\gamma$ so $N p < N \gamma$. $\sigma<Np-N\gamma +p$ }
	\end{assumption}
	\begin{comment}
		\begin{remark}
			While the choice of \( p \) in \eqref{AASdf} is arbitrary, we can select \( p \) arbitrarily close to \( \frac{\sigma + N\gamma}{N+1} \) such that  
			\begin{align*}
				N + 1 < \frac{\gamma - \sigma}{\gamma - p} < N + 2.
			\end{align*}
		\end{remark}
	\end{comment}
	Now we can state the sewing lemma and define the rough integral for controlled rough paths $\boldsymbol{\xi}$ as given in  Definition~\ref{controlled} against rough paths $\mathbf{X}$.  We first recall that the number of iterated integrals $N$ and the regularity of the rough path $\gamma$ satisfy the following relation \( (N+1)\gamma > 1 \) and \( N\gamma < 1 \).
	\begin{theorem}\label{shdccsa}%\todo{Change the notation, remove the index $1$ and write $\cD$ and $\tilde{\cD}$ below}\todo{Has been changed accordingly.}
		We	assume that \( \mathbf{X} \in \mathscr{C}^{\gamma, p}([0,T]) \) and let \( 0 \leq \sigma < \gamma \). Let Assumption \ref{ADssdfg} hold and let
		$
		\boldsymbol{\xi} \in \mathscr{D}^{\gamma, p}_{\mathbf{X},\alpha-\sigma}\left([0,T]\right).
		$
		For every \( [s,t] \subseteq [0,T] \), \( m \geq 0 \) and \( 0 \leq n \leq 2^m \), we set  
\[
\tau^{n}_{m} := s + \frac{n}{2^m} (t-s).
\]
Then, the following limit exists
		\begin{align}\label{HAYs}
			\int_{s}^{t}U_{t,v}\xi_{v}\circ\mathrm{d}\mathbf{X}_{v}:=\lim_{m\rightarrow\infty}\left(\sum_{0\leq n< 2^m}\sum_{0\leq j< N}U_{t,\tau^{n}_{m}}\xi^{j}_{\tau^{n}_{m}}\circ (\Pi^{j+1}(\mathbf{X}))_{\tau^{n}_{m},\tau^{n+1}_{m}}\right).
          %  \sum_{\tau\in \mathcal{D}_m[s,t]}\Xi^{\tau_-,\tau}_{s,t}. \Xi^{\tau^{n}_{m},\tau^{n+1}_{m}}_{s,t}
		\end{align}
		%where the notations are as defined in Definition \ref{NOTTR}. 
		Moreover, for \( 0 \leq l \leq N \), we have %\todo{index $j$ for the projection in the lhs of (2.16), \textcolor{red}{Sorry dont understand, what you mean?} there was a typo below, $i$ instead of $j$, i put it in blue}
		\begin{align}\label{sodais} 
			\begin{split}
				&\left|\int_{s}^{t}U_{t,v}\xi_{v}\circ\mathrm{d}\mathbf{X}_{v}-\sum_{0\leq j< N}U_{t,s}\xi^{j}_{s}\circ(\Pi^{{j}+1}(\mathbf{X}))_{s,t}\right|_{E_{\alpha-l\gamma}}\lesssim (t-s)^{lp}\times\\&\sum_{0\leq j<N}W_{\Pi^{j+1}(\mathbf{X}),\gamma,p}^{(j+1)(\gamma-p)}(s,t)\left( (t-s)^{P_{N}+l(\gamma-p)}W_{R^{j,N},\alpha-\sigma,\gamma,p,2}^{(N-j)(\gamma-p)}(s,t)+(t-s)^{P_{j}+l(\gamma-p)}\sup_{\tau\in [s,t]}\left\vert\xi^{j}_{\tau}\right\vert_{\alpha-j\gamma-\sigma}\right),
			\end{split}
		\end{align}
		where $W_{\Pi^j(\mathbf{X}),\gamma,p}$ is specified in~\eqref{AASdfd}.
		%Also, by \( |\cdot|_{\beta} \), we mean \( |\cdot|_{E_\beta} \).
	\end{theorem}
	\begin{proof}
		Throughout the proof we use that $\mathbf{\xi}\in \mathscr{D}^{\gamma,p}_{\mathbf{X},\alpha-\sigma}([0,T])$ as specified in Definition \ref{controlled} and  
		the regularizing effect of parabolic evolution families given in~\eqref{regularity}. The idea of the proof is similar to the classical sewing lemma~\cite[Theorem 4.1]{GHT21} which involves H\"older norms  
		and~\cite[Lemma 2.5]{GVR25} that replaces the H\"older norms of the rough path $\mathbf{X}$ by control functions for rough paths of regularity $\gamma\in(1/3,1/2)$. Here we do not only have to deal with the control functions in~\eqref{AASdfd} but additionally have to incorporate the norm~\eqref{YHN123} defined on the space of controlled rough paths. We divide the proof into four steps for better readability. \\
		\textbf{Step one:}
		Let us fix \( [s,t] \subseteq [0,T] \). For \( m \geq 1 \), we define  
		\[
		\Gamma^{m}_{s,t} :=\sum_{0\leq n< 2^m}\sum_{0\leq j< N}U_{t,\tau^{n}_{m}}\xi^{j}_{\tau^{n}_{m}}\circ (\Pi^{j+1}(\mathbf{X}))_{\tau^{n}_{m},\tau^{n+1}_{m}}.
		\]  
	In particular
		\[
		\Gamma^{0}_{s,t} = \Xi^{s,t}_{s,t} = \sum_{0 \leq j < N} U_{t,s} \xi^{j}_{s} \circ (\Pi^{j+1}(\mathbf{X}))_{s,t}.
		\]  
	Then, from Lemma \ref{increment} and with some direct computations, we have for \( m \geq 1\)
		\begin{align*}
			\Gamma^{m}_{s,t}-\Gamma^{m-1}_{s,t}=\sum_{\tau\in D_{m}[s,t]}\sum_{0\leq j<N}\left(U_{t,\tau_-}R^{j,N}_{\tau_-,\tau}\circ(\Pi^{j+1}(\mathbf{X}))_{\tau,\tau_+}+(U_{t,\tau}-U_{t,\tau_-})\xi^{j}_{\tau}\circ(\Pi^{j+1}(\mathbf{X}))_{\tau,\tau_+}\right).
		\end{align*}
		Our aim is to estimate the following terms for every \( l \leq N \) 
		\begin{align}\label{DFFDA}
			\begin{split}
				&\sum_{m\geq 1}\left\vert\Gamma^{m}_{s,t}-\Gamma^{m-1}_{s,t}\right\vert_{E_{\alpha-l\gamma}}\leq \underbrace{\sum_{m\geq 1}\sum_{\tau\in D_{m}[s,t]}\sum_{0\leq j<N}\left\vert U_{t,\tau_-}R^{j,N}_{\tau_-,\tau}\circ(\Pi^{j+1}(\mathbf{X}))_{\tau,\tau_+}\right\vert_{E_{\alpha-l\gamma}}}_{:=E_{l}(s,t)}\\&+\underbrace{\sum_{m\geq 1}\sum_{\tau\in D_{m}[s,t]}\sum_{0\leq j<N}\left\vert(U_{t,\tau}-U_{t,\tau_-})\xi^{j}_{\tau}\circ(\Pi^{j+1}(\mathbf{X}))_{\tau,\tau_+}\right\vert_{E_{\alpha-l\gamma}}}_{:=Q_{l}(s,t)}.
			\end{split}
		\end{align}
		Our objective is to estimate \( E_{l}(s,t) \) and \( Q_{l}(s,t) \) which will be done in the following steps.
		\newline
		\textbf{Step two: } We begin with the estimate for $ E_{l}(s,t)$. 
		For \( \tau \in \mathcal{D}_{m}[s,t] \), we have \( \tau - \tau_{-} = \tau_{+} - \tau = \frac{t-s}{2^m} \). Thus for \( j < N \) we derive based on the regularizing properties of the parabolic evolution family given in~\eqref{regularity} combined with Definition~\ref{controlled} and~\eqref{AASdfd}
		\begin{align}\label{YHGBC}
			\begin{split}
				&\left\vert U_{t,\tau_-}R^{j,N}_{\tau_-,\tau}\circ(\Pi^{j+1}(\mathbf{X}))_{\tau,\tau_+}\right\vert_{E_{\alpha-l\gamma}}\leq \left\vert U_{t,\tau_-}R^{j,N}_{\tau_-,\tau}\right\vert_{L\left((\R^d)^{\otimes j+1},E_{\alpha-l\gamma}\right)}(\tau_{+}-\tau)^{p(j+1)}W_{\Pi^{j+1}(\mathbf{X}),\gamma,p}^{(j+1)(\gamma-p)}(\tau,\tau_+)\\&\lesssim (t-\tau_{-})^{\gamma(l-N)-\sigma}(\tau-\tau_{-})^{p(N-j)}(\tau_+-\tau)^{p(j+1)}W_{R^{j,N},\alpha-\sigma,\gamma,p,2}^{(N-j)(\gamma-p)}(\tau_{-},\tau)W_{\Pi^{j+1}(\mathbf{X}),\gamma,p}^{(j+1)(\gamma-p)}(\tau,\tau_+)\\
				&=(t-\tau_{-})^{\gamma(l-N)-\sigma}(\tau-\tau_{-})^{p(N-j)}(\tau_+-\tau)^{p(j+1)}\left(W_{R^{j,N},\alpha-\sigma,\gamma,p,2}^{\frac{N-j}{N+1}}(\tau_{-},\tau)W_{\Pi^{j+1}(\mathbf{X}),\gamma,p}^{\frac{j+1}{N+1}}(\tau,\tau_+)\right)^{(N+1)(\gamma-p)}\\&=(t-s)^{P_{N}+l\gamma}(\frac{1}{2^m})^{(N+1)p}\Big(\frac{t-\tau_-}{t-s}\Big)^{(l-N)\gamma-\sigma}\left(W_{R^{j,N},\alpha-\sigma,\gamma,p,2}^{\frac{N-j}{N+1}}(\tau_{-},\tau)W_{\Pi^{j+1}(\mathbf{X}),\gamma,p}^{\frac{j+1}{N+1}}(\tau,\tau_+)\right)^{(N+1)(\gamma-p)}.
			\end{split}
		\end{align}
		Since $(N+1)(\gamma-p) < 1$, it follows from the Hölder inequality that
		\begin{align}\label{SHHSBND}
			\begin{split}
				&\sum_{\tau\in \mathcal{D}_{m}[s,t]}\sum_{0\leq j<N}\Big(\frac{1}{2^m}\Big)^{(N+1)p}\Big(\frac{t-\tau_-}{t-s}\Big)^{(l-N)\gamma-\sigma}\left(W_{R^{j,N},\alpha-\sigma,\gamma,p,2}^{\frac{N-j}{N+1}}(\tau_{-},\tau)W_{\Pi^{j+1}(\mathbf{X}),\gamma,p}^{\frac{j+1}{N+1}}(\tau,\tau_+)\right)^{(N+1)(\gamma-p)}
				\\&\leq \sum_{0\leq j< N}\Bigg(
				\Big(\sum_{\tau\in \mathcal{D}_{m}[s,t]}\Big(\frac{1}{2^m}\Big)^{\frac{(N+1)p}{1-(N+1)(\gamma-p)}}\Big(\frac{t-\tau_-}{t-s}\Big)^{\frac{(l-N)\gamma-\sigma}{1-(N+1)(\gamma-p)}}\Big)^{1-(N+1)(\gamma-p)}\bigtimes\\&\quad\Big(\sum_{\tau\in \mathcal{D}_{m}[s,t]}W_{R^{j,N},\alpha-\sigma,\gamma,p,2}^{\frac{N-j}{N+1}}(\tau_{-},\tau)W_{\Pi^{j+1}(\mathbf{X}),\gamma,p}^{\frac{j+1}{N+1}}(\tau,\tau_+)\Big)^{(N+1)(\gamma-p)}\Bigg).
			\end{split}
		\end{align}
		Recall that \( W_{\Pi^{j+1}(\mathbf{X}),\gamma,p} \) and $W_{R^{j,N},\alpha-\sigma,\gamma,p,2}$ are assumed to be  controlled functions as introduced in Definition~\ref{pgamma:rp} and Definition \ref{controlled}. Therefore, by using the Hölder inequality again, we have
		\begin{align}
        \label{HNBZS}
			\begin{split}
				&\left(\sum_{\tau\in \mathcal{D}_{m}[s,t]}W_{R^{j,N},\alpha-\sigma,\gamma,p,2}^{\frac{N-j}{N+1}}(\tau_{-},\tau)W_{\Pi^{j+1}(\mathbf{X}),\gamma,p}^{\frac{j+1}{N+1}}(\tau,\tau_+)\right)^{(N+1)(\gamma-p)}\\&\leq \left( \sum_{\tau\in \mathcal{D}_{m}[s,t]}W_{R^{j,N},\alpha-\sigma,\gamma,p,2}(\tau_{-},\tau)\right)^{(N-j)(\gamma-p)}\left( \sum_{\tau\in \mathcal{D}_{m}[s,t]}W_{\Pi^{j+1}(\mathbf{X}),\gamma,p}(\tau,\tau_+)\right)^{(j+1)(\gamma-p)}
				\\& \leq W_{R^{j,N},\alpha-\sigma,\gamma,p,2}^{(N-j)(\gamma-p)}(s,t)W_{\Pi^{j+1}(\mathbf{X}),\gamma,p}^{(j+1)(\gamma-p)}(s,t).
			\end{split}
		\end{align}
		Consequently, from \eqref{DFFDA}, \eqref{YHGBC} and \eqref{HNBZS} we derive a bound for $E_l$ as %\todo{bigger parantheses $()$ might look better below for the terms with fractions,\textcolor{red}{I have no idea how to do that} \textcolor{blue}{$\Big(\Big)$ as above} }
		\begin{align}\label{A85ywere}
			\begin{split}
				&E_{l}(s,t)\lesssim (t-s)^{P_{N}+l\gamma}\left(\sum_{0\leq j<N}W_{R^{j,N},\alpha-\sigma,\gamma,p,2}^{(N-j)(\gamma-p)}(s,t)W_{\Pi^{j+1}(\mathbf{X}),\gamma,p}^{(j+1)(\gamma-p)}(s,t)\right)\times\\&\left(\sum_{m\geq 1}\left(\sum_{\tau\in \mathcal{D}_{m}[s,t]}\Big(\frac{1}{2^m}\Big)^{\frac{(N+1)p}{1-(N+1)(\gamma-p)}}\Big(\frac{t-\tau_-}{t-s}\Big)^{\frac{(l-N)\gamma-\sigma}{1-(N+1)(\gamma-p)}}\right)^{1-(N+1)(\gamma-p)}\right).
			\end{split}
		\end{align}
		We must ensure that the last term in the previous formula is finite. Note that for every \( \tau \in \mathcal{D}_{m}[s,t] \)  %\todo{should we emphasize that the choice of $\tau$ determines $\tau_-$ below just to be clear?,\textcolor{red}{Is it ok to put in the footnote, In the text might be confusion?} i usually avoid footnotes, i think it's ok} 
        we can find \( 0 \leq n < 2^{m-1} \) such that \[
		\frac{t - \tau_{-}}{t - s} = 1 - \frac{n}{2^{m-1}}.
		\]
		We further have that $\frac{t - \tau_{-}}{t - s}=1 - \frac{n}{2^{m-1}}\geq \frac{1}{2^{m}}.$ Recalling that Assumption \ref{ADssdfg} holds, we choose 
		\[
		0<\epsilon<\min\left\lbrace \frac{P_N}{1-(N+1)(\gamma-p)},\frac{(N+1)\gamma-1}{1-(N+1)(\gamma-p)}\right\rbrace.	
		\]
		Consequently,  for every $\tau \in \mathcal{D}_{m}[s,t]$ and $0\leq l<N$ we obtain 
		\begin{align*}
			&\Big(\frac{1}{2^m}\Big)^{\frac{(N+1)p}{1-(N+1)(\gamma-p)}}\Big(\frac{t-\tau_-}{t-s}\Big)^{\frac{(l-N)\gamma-\sigma}{1-(N+1)(\gamma-p)}}\\ &\leq \Big(\frac{1}{2^m}\Big)^{1+\epsilon}\Big(\frac{t-\tau_-}{t-s}\Big)^{\frac{(l-N)\gamma+(N+1)p-\sigma}{1-(N+1)(\gamma-p)}-1-\epsilon}=\Big(\frac{1}{2^m}\Big)^{1+\epsilon}\Big(\frac{t-\tau_-}{t-s}\Big)^{\frac{P_{N}+l\gamma}{1-(N+1)(\gamma-p)}-1-\epsilon}.
		\end{align*}
		Thus, keeping in mind the definitions of $P_N$ and $\tau_{-}$, we further infer that
		\begin{align}\label{HBZVFS}
			\begin{split}
				&\sum_{m\geq 1}\left(\sum_{\tau\in \mathcal{D}_{m}[s,t]}\Big(\frac{1}{2^m}\Big)^{\frac{(N+1)p}{1-(N+1)(\gamma-p)}}\Big(\frac{t-\tau_-}{t-s}\Big)^{\frac{(l-N)\gamma-\sigma}{1-(N+1)(\gamma-p)}}\right)^{1-(N+1)(\gamma-p)}\\&\myquad[1]= \sum_{m\geq1}\left(\sum_{0\leq n< 2^{m-1}}\Big(\frac{1}{2^m}\Big)^{\frac{(N+1)p}{1-(N+1)(\gamma-p)}}\Big(1-\frac{n}{2^{m-1}}\Big)^{\frac{(l-N)\gamma-\sigma}{1-(N+1)(\gamma-p)}}\right)^{1-(N+1)(\gamma-p)}\\&\myquad[2]\leq \sum_{m\geq 1}\left(\sum_{0\leq n< 2^{m-1}}\Big(\frac{1}{2^m}\Big)^{1+\epsilon}\Big(1-\frac{n}{2^{m-1}}\Big)^{\frac{P_{N}+l\gamma}{1-(N+1)(\gamma-p)}-1-\epsilon}\right)^{1-(N+1)(\gamma-p)}\\&\myquad[3]\leq \sum_{m\geq1}\left(\frac{1}{2^m}\right)^{\epsilon\left(1-(N+1)(\gamma-p)\right)}\left(\sum_{0\leq n< 2^{m-1}}\Big(\frac{1}{2^m}\Big)\Big(1-\frac{n}{2^{m-1}}\Big)^{\frac{P_{N}+l\gamma}{1-(N+1)(\gamma-p)}-1-\epsilon}\right)^{1-(N+1)(\gamma-p)}\\&\myquad[4]\lesssim \left(\int_{0}^{1}(1-x)^{\frac{P_{N}+l\gamma}{1-(N+1)(\gamma-p)}-1-\epsilon}\mathrm{d}x\right)^{1-(N+1)(\gamma-p)}<\infty.
			\end{split}
		\end{align}
		Therefore, from \eqref{A85ywere} and \eqref{HBZVFS}, we obtain
		\begin{align}\label{BAGST}
			E_{l}(s,t)\lesssim (t-s)^{lp}(t-s)^{P_{N}+l(\gamma-p)}\sum
			_{0\leq j<N}W_{R^{j,N},\alpha-\sigma,\gamma,p,2}^{(N-j)(\gamma-p)}(s,t)W_{\Pi^{j+1}(\mathbf{X}),\gamma,p}^{(j+1)(\gamma-p)}(s,t).
		\end{align}
		\newline
		\textbf{Step three:} Now, we estimate \( Q_{l}(s,t) \) given in~\eqref{DFFDA}. %\todo{label it so we can refer to it here} 
		Since most of the calculations are similar to the previous step, we do not present all details. Recalling that $\xi^{j}$ takes values in $L\left((\R^d)^{\otimes j+1},E_{\alpha-j\gamma-\sigma}\right)$, we have
		\begin{align}\label{YHNBZa}
			\begin{split}
				&\left\vert(U_{t,\tau}-U_{t,\tau_-})\xi^{j}_{\tau}\circ(\Pi^{j+1}(\mathbf{X}))_{\tau,\tau_+}\right\vert_{E_{\alpha-l\gamma}}\\ &\quad\lesssim \left\vert (U_{t,\tau}-U_{t,\tau_-})\xi^{j}_{\tau}\right\vert_{L\left((\R^d)^{\otimes j+1},E_{\alpha-l\gamma}\right)}(\tau_{+}-\tau)^{p(j+1)}W_{\Pi^{j+1}(\mathbf{X}),\gamma,p}^{(j+1)(\gamma-p)}(\tau,\tau_+).
			\end{split}
		\end{align}
		We choose a parameter \( \theta \) such that
		\begin{align}\label{AAA}
			\begin{split}
				&N\gamma < \theta < 1,\\
				& 1-\gamma+\sigma<\theta.
			\end{split}
		\end{align}
		This choice is possible since \( N\gamma < 1 \) and \( \sigma < \gamma \), and it ensures that the following conditions are satisfied for every \( 0 \leq j < N \) and every \( 0 \leq l \leq N \):
		\begin{itemize}  
			\item \( \alpha - \theta < \alpha - j\gamma - \sigma \),  
			\item \( \alpha - \theta < \alpha - l\gamma \).  
		\end{itemize}  
		Then, from \eqref{regularity}, we conclude for this choice of $\theta$ that
		\begin{align}\label{UJNXa}
			\begin{split}
				&\left\vert (U_{t,\tau}-U_{t,\tau_-})\xi^{j}_{\tau}\right\vert_{L\left((\R^d)^{\otimes j+1},E_{\alpha-l\gamma}\right)}=\left\vert U_{t,\tau}(I-U_{\tau,\tau_-})\xi^{j}_{\tau}\right\vert_{L\left((\R^d)^{\otimes j+1},E_{\alpha-l\gamma}\right)}\\&\quad \lesssim \left\vert U_{t,\tau}\right\vert_{L(E_{\alpha -\theta},E_{\alpha-l\gamma})}\left\vert (I-U_{\tau,\tau_-})\xi^{j}_{\tau}\right\vert_{L\left((\R^d)^{\otimes j+1},E_{\alpha-\theta}\right)}\\&\qquad\lesssim (t-\tau)^{l\gamma-\theta}(\tau-\tau_{-})^{\theta-j\gamma-\sigma}\sup_{\tau\in [s,t]}\left\vert\xi^{j}_{\tau}\right\vert_{L\left((\R^d)^{\otimes j+1},E_{\alpha-j\gamma-\sigma}\right)}.
			\end{split}
		\end{align}
		So, from \eqref{YHNBZa} and \eqref{UJNXa}
		\begin{align}\label{II54}
			\begin{split}
				&Q_{l}(s,t)=\sum_{m\geq 1}\sum_{\tau\in D_{m}[s,t]}\sum_{0\leq j<N}\left\vert(U_{t,\tau}-U_{t,\tau_-})\xi^{j}_{\tau}\circ(\Pi^{j+1}(\mathbf{X}))_{\tau,\tau_+}\right\vert_{E_{\alpha-l\gamma}}\\&\lesssim \sum_{0\leq j<N}(t-s)^{lp+l(\gamma-p)+P_{j}}\sup_{\tau\in [s,t]}\left\vert\xi^{j}_{\tau}\right\vert_{\alpha-j\gamma-\sigma}\times \\&\quad\left(\sum_{m\geq 1}\sum_{\tau\in D_{m}[s,t]}\Big(\frac{1}{2^m}\Big)^{P_{j}+\theta}\Big(\frac{t-\tau_-}{t-s}\Big)^{l\gamma-\theta}W_{\Pi^{j+1}(\mathbf{X}),\gamma,p}^{(j+1)(\gamma-p)}(\tau,\tau_+)\right).
			\end{split}
		\end{align}
		Since \( (j+1)(\gamma - p) < 1 \) holds for all \( 0 \leq j < N \), and by a similar argument to that used in \eqref{SHHSBND} and \eqref{HNBZS}, we conclude that
		\begin{align}\label{sidua654}
			\begin{split}
				&\sum_{m\geq 1}\sum_{\tau\in D_{m}[s,t]}\Big(\frac{1}{2^m}\Big)^{P_{j}+\theta}\Big(\frac{t-\tau_-}{t-s}\Big)^{l\gamma-\theta}W_{\Pi^{j+1}(\mathbf{X}),\gamma,p}^{(j+1)(\gamma-p)}(\tau,\tau_+)\\& \leq \sum_{m\geq 1}\left(\sum_{\tau\in D_{m}[s,t]}\Big(\frac{1}{2^m}\Big)^{\frac{P_{j}+\theta}{1-(j+1)(\gamma-p)}}
            \Big(\frac{t-\tau_-}{t-s}\Big)^{\frac{l\gamma-\theta}{1-(j+1)(\gamma-p)}}\right)^{1-(j+1)(\gamma-p)}\times\\&\qquad\left(\sum_{\tau\in D_{m}[s,t]}W_{\Pi^{j+1}(\mathbf{X}),\gamma,p}(\tau,\tau_+)\right)^{(j+1)(\gamma-p)}\\
				&\myquad[3]\leq W_{\Pi^{j+1}(\mathbf{X}),\gamma,p}^{(j+1)(\gamma-p)}(s,t)\sum_{m\geq 1}\left(\sum_{\tau\in D_{m}[s,t]}\Big(\frac{1}{2^m}\Big)^{\frac{P_{j}+\theta}{1-(j+1)(\gamma-p)}}\Big(\frac{t-\tau_-}{t-s}\Big)^{\frac{l\gamma-\theta}{1-(j+1)(\gamma-p)}}\right)^{1-(j+1)(\gamma-p)}.
			\end{split}
		\end{align}
		We need to verify that the summation term does not grow large as the size of the partition becomes smaller. The reasoning closely resembles the previous case, i.e., \eqref{HBZVFS}, though a few minor adjustments need to be considered. Our choice of \( \theta \) in \eqref{AAA} implies that 
		\begin{align}\label{hasnx}
			\frac{P_{j}+\theta}{1-(j+1)(\gamma-p)}>1.
		\end{align}
		Thus, for every \( \tau \in \mathcal{D}_{m}[s,t] \), since \( \frac{1}{2^m} \leq \frac{t-\tau_-}{t-s} \) and $p>\frac{\sigma+N\gamma}{(N+1)}$, we get 
		\begin{align}\label{UJMMZLa}
			\Big(\frac{1}{2^m}\Big)^{\frac{P_{j}+\theta}{1-(j+1)(\gamma-p)}}\Big(\frac{t-\tau_-}{t-s}\Big)^{\frac{l\gamma-\theta}{1-(j+1)(\gamma-p)}}\leq \Big(\frac{1}{2^m}\Big)^{1+\epsilon} \Big(\frac{t-\tau_-}{t-s}\Big)^{\frac{P_{j}+l\gamma}{1-(j+1)(\gamma-p)}-1-\epsilon},
		\end{align}
		where 
		\begin{align*}
			0<	\epsilon< \min\left\lbrace \frac{P_{N}}{1-(N+1)(\gamma-p)}, \min\left\lbrace\frac{P_{j}+\theta}{1-(j+1)(\gamma-p)}-1\right\rbrace_{0\leq j< N}\right\rbrace.		\end{align*}
		Note that this choice of \(\epsilon\) is valid due to \eqref{hasnx} and Assumption \ref{ADssdfg}. This implies that for every \(0 \leq j < N\) and every \(0 \leq l \leq N\) we have
		\begin{align*}
			\frac{P_{j}+l\gamma}{1-(j+1)(\gamma-p)}-1-\epsilon+1=\frac{P_{j}+l\gamma}{1-(j+1)(\gamma-p)}-\epsilon>0.
		\end{align*}
		Consequently, as in~\eqref{HBZVFS} we derive 
		\begin{align}\label{AUysaa}
			\begin{split}
				&\sum_{m\geq 1}\left(\sum_{\tau\in D_{m}[s,t]}\Big(\frac{1}{2^m}\Big)^{\frac{P_{j}+\theta}{1-(j+1)(\gamma-p)}}\Big(\frac{t-\tau_-}{t-s}\Big)^{\frac{l\gamma-\theta}{1-(j+1)(\gamma-p)}}\right)^{1-(j+1)(\gamma-p)}\\&\quad\lesssim \sum_{m\geq 1}\Big(\frac{1}{2^m}\Big)^{\epsilon(1-(j+1)(\gamma-p))}\left(\sum_{0\leq n< 2^{m-1}}\Big(\frac{1}{2^m}\Big)\Big(1-\frac{n}{2^{m-1}}\Big)^{\frac{P_{j}+l\gamma}{1-(j+1)(\gamma-p)}-1-\epsilon}\right)^{1-(j+1)(\gamma-p)}\\&\qquad\lesssim \left(\int_{0}^{1}(1-x)^{\frac{P_{j}+l\gamma}{1-(j+1)(\gamma-p)}-1-\epsilon}~\mathrm{d} x\right)^{1-(j+1)(\gamma-p)}<\infty.
			\end{split}
		\end{align}
		Note that the estimate  \eqref{AUysaa} depends on \( j \), but we take the maximal bound over all \( 0 \leq j < N \). Now combining \eqref{II54} and \eqref{sidua654} and \eqref{AUysaa} we derive
		\begin{align}\label{sdjfs0}
			Q_{l}(s,t)\lesssim (t-s)^{lp}\sum_{0\leq j<N}(t-s)^{l(\gamma-p)+P_{j}}W_{\Pi^{j+1}(\mathbf{X}),\gamma,p}^{(j+1)(\gamma-p)}(s,t)\sup_{\tau\in [s,t]}\left\vert\xi^{j}_{\tau}\right\vert_{\alpha-j\gamma-\sigma}.
		\end{align}
		\textbf{Step four:}  Now we have all the estimates required in order to prove our claim. From \eqref{DFFDA},~\eqref{BAGST} and \eqref{sdjfs0}, for every $0\leq l\leq N$, we conclude
		\begin{align}
			\sum_{m\geq 1}\left\vert\Gamma^{m}_{s,t}-\Gamma^{m-1}_{s,t}\right\vert_{\alpha-l\gamma}<\infty.
		\end{align}
		Therefore, we can infer that the following limit exists
		\begin{align*}
			\int_{s}^{t}U_{t,v}\xi_{v}\circ\mathrm{d}\mathbf{X}_{v}:=\lim_{m\rightarrow\infty}\left(\sum_{0\leq n< 2^m}\sum_{0\leq j< N}U_{t,\tau^{n}_{m}}\xi^{j}_{\tau^{n}_{m}}\circ (\Pi^{j+1}(\mathbf{X}))_{\tau^{n}_{m},\tau^{n+1}_{m}}\right)=\lim_{m\rightarrow\infty}\Gamma^{m}_{s,t}.
		\end{align*}
		Furthermore, we have for every $0\leq l\leq N$ that
		\begin{align*}
			\Bigg|\int_{s}^{t}U_{t,v}\xi_{v}\circ\mathrm{d}\mathbf{X}_{v}-\underbrace{\sum_{0\leq j< N}U_{t,s}\xi^{j}_{s}\circ(\Pi^{j+1}(\mathbf{X}))_{s,t}}_{\Gamma^{0}_{s,t}}\Bigg|_{E_{\alpha-l\gamma}}\leq \sum_{m\geq 1}\left\vert\Gamma^{m}_{s,t}-\Gamma^{m-1}_{s,t}\right\vert_{E_{\alpha-l\gamma}}.
		\end{align*}
		Now, from \eqref{DFFDA}, \eqref{BAGST}, and \eqref{sdjfs0}, we conclude that
		\begin{align*}
			&\Bigg|\int_{s}^{t}U_{t,v}\xi_{v}\circ\mathrm{d}\mathbf{X}_{v}-\underbrace{\sum_{0\leq j< N}U_{t,s}\xi^{j}_{s}\circ(\Pi^{j+1}(\mathbf{X}))_{s,t}}_{\Gamma^{0}_{s,t}}\Bigg|_{\alpha-l\gamma}\leq E_{l}(s,t)+Q_{l}(s,t) \lesssim (t-s)^{lp}\times\\
			&\sum_{0\leq j<N}W_{\Pi^{j+1}(\mathbf{X}),\gamma,p}^{(j+1)(\gamma-p)}(s,t)\left( (t-s)^{P_{N}+l(\gamma-p)}W_{R^{j,N},\alpha-\sigma,\gamma,p,2}^{(N-j)(\gamma-p)}(s,t)+(t-s)^{P_{j}+l(\gamma-p)}\sup_{\tau\in [s,t]}\left\vert\xi^{j}_{\tau}\right\vert_{\alpha-j\gamma-\sigma}\right).
		\end{align*}
		This proves the statement, i.e., \eqref{sodais}.
	\end{proof}
	As a direct consequence of Theorem \ref{shdccsa}, we can prove the following result. 
	\begin{corollary}\label{saswf65}
		We assume the same setting as in Theorem \ref{shdccsa}. Then for every \( [s,t] \subseteq [0,T] \), we have the following uniform bound of the rough integral
		\begin{align}\label{sodaiwqes}
			\begin{split}
				&\left|\int_{s}^{t}U_{t,v}\xi_{v}\circ\mathrm{d}\mathbf{X}_{v}\right|_{E_{\alpha}}\lesssim \\&\sum_{0\leq j<N}W_{\Pi^{j+1}(\mathbf{X}),\gamma,p}^{(j+1)(\gamma-p)}(s,t)\left( (t-s)^{P_{N}}W_{R^{j,N},\alpha-\sigma,\gamma,p,2}^{(N-j)(\gamma-p)}(s,t)+(t-s)^{P_{j}}\sup_{\tau\in [s,t]}\left\vert\xi^{j}_{\tau}\right\vert_{\alpha-j\gamma-\sigma}\right).
			\end{split}
		\end{align}
	\end{corollary}
	\begin{proof}
		The idea is to apply Theorem \ref{shdccsa}. First, note that
		\begin{align}\label{UAJ12}
			\begin{split}
				&\left\vert\int_{s}^{t}U_{t,v}\xi_{v}\circ\mathrm{d}\mathbf{X}_{v}\right\vert_{E_{\alpha}}\\&\leq \left|\int_{s}^{t}U_{t,v}\xi_{v}\circ\mathrm{d}\mathbf{X}_{v}-\sum_{0\leq j< N}U_{t,s}\xi^{j}_{s}\circ(\Pi^{i+1}(\mathbf{X}))_{s,t}\right|_{E_{\alpha}}+\sum_{0\leq j< N}\left\vert U_{t,s}\xi^{j}_{s}\circ(\Pi^{i+1}(\mathbf{X}))_{s,t}\right\vert_{E_{\alpha}}.
			\end{split}
		\end{align}
		Then, since $\xi^j\in L(\mathbb{R}^d, E_{\alpha - j\gamma - \sigma}) $ we obtain from \eqref{regularity} that
		\begin{align}\label{UAJ13}
			\begin{split}
				&\sum_{0\leq j< N}\left\vert U_{t,s}\xi^{j}_{s}\circ(\Pi^{i+1}(\mathbf{X}))_{s,t}\right\vert_{E_{\alpha}}\\&\lesssim\sum_{0\leq j<N}(t-s)^{(j+1)p-j\gamma-\sigma}W_{\Pi^{j+1}(\mathbf{X}),\gamma,p}^{(j+1)(\gamma-p)}(s,t)\sup_{\tau\in [s,t]}\left\vert\xi^{j}_{\tau}\right\vert_{\alpha-j\gamma-\sigma}.
			\end{split}
		\end{align}
		From \eqref{sodais} and for \( l = 0 \), we conclude that
		\begin{align}\label{UAJ14}
			\begin{split}
				&\left|\int_{s}^{t}U_{t,v}\xi_{v}\circ\mathrm{d}\mathbf{X}_{v}-\sum_{0\leq j< N}U_{t,s}\xi^{j}_{s}\circ(\Pi^{i+1}(\mathbf{X}))_{s,t}\right|_{\alpha}\\&\lesssim\sum_{0\leq j<N}W_{\Pi^{j+1}(\mathbf{X}),\gamma,p}^{(j+1)(\gamma-p)}(s,t)\left( (t-s)^{P_{N}}W_{R^{j,N},\alpha-\sigma,\gamma,p,2}^{(N-j)(\gamma-p)}(s,t)+(t-s)^{P_{j}}\sup_{\tau\in [s,t]}\left\vert\xi^{j}_{\tau}\right\vert_{\alpha-j\gamma-\sigma}\right).
			\end{split}
		\end{align}
		The claim follows combining  \eqref{UAJ12}, \eqref{UAJ13} and \eqref{UAJ14}.
	\end{proof}
	
	Based on Theorem \ref{shdccsa}, we can define for every controlled rough path \( \mathbf{\xi} \in \mathscr{D}^{\gamma, p}_{\mathbf{X},\alpha-\sigma}\left([0,T]\right) \) an integral with respect to \( \mathbf{X} \). Naturally, the rough integral together with its Gubinelli derivatives forms again a controlled rough path with respect to \( \mathbf{X} \). We establish this in Proposition~\ref{UJMMA}.
	%, but within a different space than \( (L(\mathbb{R}^d, E_{\alpha - j\gamma - \sigma}))_{0 \leq j < N} \). %This will be the focus of our discussion in the next section. We close this section with the following remark:
%	\todo{\textcolor{red}{Is this remark  now Ok?}}
	\begin{remark}
		Recall that for \( \mathbf{\xi} \in \mathscr{D}^{\gamma, p}_{\mathbf{X}, \alpha - \sigma}\left([0,T]\right) \), we defined terms such as \( W_{R^{i,l}, \alpha - \sigma, \gamma, p, 1}(s,t) \) and \( W_{R^{i,l}, \alpha - \sigma, \gamma, p, 2}(s,t) \). In Theorem \ref{shdccsa} we did not use these terms when \( l < N \). % (along with many other terms), and these items appear to be superfluous at first glance. 
		However, {the role of these terms becomes clear in Propositions \ref{ATSY} and \ref{UJMMA}. }%\todo{state where exactly it becomes clear in the manuscript.,\textcolor{red}{done}}
	\end{remark}
	
	%\section{Integrand path}

	The following definition closely resembles Definition \ref{controlled} with two main differences. One of them is that the controlled rough path \( \boldsymbol{\tilde{\xi}} \) directly takes values in \( E_\alpha \) instead of \( E_{\alpha-\sigma} \) for \( \sigma < \gamma \). Furthermore, we also modify the index of the tensor product appearing in the definition of the controlled rough path  \( \boldsymbol{\tilde{\xi}} \). 
    This is done for technical reasons which will become clear in  Proposition~\ref{ATSY} and Proposition~\ref{UJMMA}, where given a path $\boldsymbol{\xi} \in \mathscr{D}^{\gamma,p}_{\mathbf{X},\alpha-\sigma} $ we construct another path $\boldsymbol{\tilde{\xi}}$ shifting the Gubinelli derivatives of $\boldsymbol{\xi}$ accordingly. The new controlled rough path $\boldsymbol{\tilde{\xi}}$ satisfies the properties stated in the next definition. 
    
    %\textcolor{red}{with an additional slight difference regarding the dimension.}
    \begin{comment}       
    \todo{ please check the indices. we have $(\R^d)^{\oplus j}$ and $j=0$ is possible for that we get $\R$. In Definition~\ref{controlled} we have $(\R^d)^{\oplus j+1}$ and $i+1$ instead of $i$ in (1)--(3). the indices must be consistent,\textcolor{red}{I see your point, So the best solution that comes to my mind is to reformulate both with tensor spaces($L(\R^{j},E_{\alpha-..})$) then one that was earlier, agree? in this way we can sort out this inconsistency} \textcolor{blue}{agree} }
        \end{comment}
	%\todo{we still have to say why we have first have ($\R^d)^{\otimes j+1}$ and here $(\R^d)^{\otimes j}$. i would unify the definitions, either $j+1$ or $j$ but both should be the same.\textcolor{red}{I  added  Remark \ref{asasR} and explained. I hope now you are satisfied. Unification unfortunate causes a lot of changes.}  }
    \begin{definition}\label{controlled2}
		We assume that \( \mathbf{X} \in \mathscr{C}^{\gamma, p}([0,T]) \). We say that \( \tilde{\boldsymbol{\xi}} = (\tilde{\xi}^{i})_{0 \leq i < N} \) is controlled by \( \mathbf{X} \) in $\left(L\left((\R^d)^{\otimes j},E_{\alpha-j\gamma}\right)\right)_{0\leq j<N}$ if we have for every \( 0 \leq i\leq j < N \) that 
		\begin{align*}
			&\tilde{\xi}^j: [0,T] \to L\left((\R^d)^{\otimes j},E_{\alpha-j\gamma}\right)=L\left((\R^d)^{\otimes j-i},L\left((\R^d)^{\otimes i},E_{\alpha-j\gamma}\right)\right),
		\end{align*}
		Moreover, for every \( 0\leq i, l \leq N \) with \( l - i > 1 \)
		\begin{align}\label{xiii}
			\delta \tilde{\xi}^{i}_{s,t}=\sum_{i<j<l}\tilde{\xi}^{j}_s\circ (\Pi^{j-i}(\mathbf{X}))_{s,t}+\tilde{R}^{i,l}_{s,t}
		\end{align}
		such that the following properties hold:
		\begin{enumerate}
			\item $\sup_{\tau\in [0,T]}\left\vert\tilde\xi^{i}_{\tau}\right\vert_{L\left((\R^d)^{\otimes i},E_{\alpha-i\gamma}\right)}<\infty$, for every $0\leq i<N$.
			\item For every \( 1\leq i, l \leq N \) with \( l - i > 1 \) and every \( [s,t] \subseteq [0,T] \) we define
			\begin{align*}
				&\tilde{W}_{\tilde{R}^{i,l},\alpha,\gamma,p,1}(s,t):=	\sup_{\pi \in P[s,t]} \left\{\sum_{k} \frac{\left|\tilde{R}^{i,l}_{\tau_{k},\tau_{k+1}}\right|_{L\left((\R^d)^{\otimes i},E_{\alpha-(l-1)\gamma}\right)}^{\frac{1}{(l-i-1)(\gamma-p)}}}{(\tau_{k+1}-\tau_k)^{\frac{p}{\gamma-p}}} \right\}<\infty,
				\\&\tilde{W}_{\tilde{R}^{i,l},\alpha,\gamma,p,2}(s,t):=	\sup_{\pi \in P[s,t]} \left\{\sum_{k} \frac{\left|\tilde{R}^{i,l}_{\tau_{k},\tau_{k+1}}\right|_{L\left((\R^d)^{\otimes i},E_{\alpha-l\gamma}\right)}^{\frac{1}{(l-i)(\gamma-p)}}}{(\tau_{k+1}-\tau_k)^{\frac{p}{\gamma-p}}} \right\}<\infty.
			\end{align*}
			\item  For every \( 1 \leq i < N \) and every \( [s,t] \subseteq [0,T] \), we define \( \tilde{R}^{i,i+1}_{s,t} := \delta \tilde\xi^i_{s,t} \) and assume that
			\begin{align*}
				\tilde{W}_{\tilde{R}^{i,i+1},\alpha,\gamma,p}(s,t):=	\sup_{\pi \in P[s,t]} \left\{\sum_{k} \frac{\left|  \delta \tilde\xi^i_{\tau_{k},\tau_{k+1}}\right|_{L\left((\R^d)^{\otimes i},E_{\alpha-(i+1)\gamma}\right)}^{\frac{1}{\gamma-p}}}{(\tau_{k+1}-\tau_k)^{\frac{p}{\gamma-p}}} \right\}<\infty.
			\end{align*}
			We adopt the convention that for \( i = N-1 \)
			\begin{align}\label{NNAN}
				\tilde{W}_{\tilde{R}^{N-1,N},\alpha,\gamma,p,2}(s,t):= \tilde{W}_{\tilde{R}^{N-1,N},\alpha,\gamma,p}(s,t),
			\end{align}
			which is again a control function by the previous assumptions.
			%We also sometimes use \( R^{N-1,N}_{s,t} := \delta \xi^{N-1}_{s,t} \), as a matter of convention.
			%\item We assume \( \tilde{W}_{\tilde{R}^{i,l},\alpha, \gamma, p, 1} \), \( \tilde{W}_{\tilde{R}^{i,l},\alpha, \gamma, p, 2} \), and \( \tilde{W}_{\tilde{R}^{i-1,i},\alpha,\gamma,p}(s,t)\) are all control functions.
		\end{enumerate}
		We use \( \tilde{\mathscr{D}}^{\gamma, p}_{\mathbf{X},\alpha}\left([0,T]\right) \) to denote the space of all such elements, equipped with the following norm 
		\begin{align}\label{YHN457}
			\begin{split}
				&\Vert\tilde{\boldsymbol{\xi}}\Vert_{\tilde{\mathscr{D}}^{\gamma, p}_{\mathbf{X},\alpha}\left([0,T]\right)}:=\sum \tilde{W}_{\tilde{R}^{i,i+1},\alpha,\gamma,p}^{\gamma-p}(0,T)\\&+\sum\left(\sup_{\tau\in [0,T]}\left\vert\tilde\xi^{i}_{\tau}\right\vert_{L\left((\R^d)^{\otimes i},E_{\alpha-i\gamma}\right)}+\tilde{W}_{\tilde{R}^{i,l},\alpha,\gamma,p,1}^{(l-i-1)(\gamma-p)}(0,T)+\tilde{W}_{\tilde{R}^{i,l},\alpha,\gamma,p,2}^{(l-i)(\gamma-p)}(0,T)\right),
			\end{split}
		\end{align}
		where the summations are taken over all indices $i$ and $l$ satisfying the conditions specified in (1)--(3) above.
	\end{definition}
%    \todo{ \textcolor{red}{ OK, let us to talk this point on next week. I have one ideas for the unification} \textcolor{blue}{i think with the blue text in front of def 2.14 we can keep now both definitions and do not necessarily have to unify them,\textcolor{red}{Thanks, very good Now I am glad} }}
    
%\begin{remark}\label{asasR}
%As we have already stated, both Definitions \ref{controlled} and \ref{controlled2} are essentially the same (though technically they are different) and can be unified under a general definition. The reason they are presented on two different platforms is primarily for ease of understanding and to avoid overwhelming details. Indeed, roughly speaking, if we assume that \( \sigma = 0 \), then algebraically (though not mathematically precise) we can argue that  
%\[
%L\left(\R^d,\tilde{\mathscr{D}}^{\gamma, p}_{\mathbf{X},\alpha}([0,T]\right)=\left( \tilde{\mathscr{D}}^{\gamma, p}_{\mathbf{X},\alpha}([0,T]) \right)^{d} = \mathscr{D}^{\gamma, p}_{\mathbf{X},\alpha}([0,T]).
%\]
%\end{remark}
 %   }
	\begin{comment}
		\text Before proceeding to the integral's definition, we aim to show that if we shift the elements of \( \boldsymbol{\xi} \) by one step, the resulting Gubinelli derivatives (approximately) remain in the same space as described in Definition \ref{controlled}. 
		To begin, we introduce the following definition:
	\end{comment}
	Previously, we defined in Theorem~\ref{shdccsa} for every \( \boldsymbol{\xi} \in \mathscr{D}^{\gamma, p}_{\mathbf{X},\alpha-\sigma} \) the integral with respect to \( \mathbf{X} \). Our aim now is to show that the rough integral together with its Gubinelli derivatives forms a controlled rough path with respct to \( \mathbf{X} \) within the space \( (E_{\alpha-j\gamma})_{0 \leq j < N} \), as specified in Definition \ref{controlled2}.~First, we begin by defining the Gubinelli derivatives, which are naturally obtained by shifting the elements of the original path by one step.~As already stated, given a controlled rough path $\boldsymbol{\xi} \in \mathscr{D}^{\gamma,p}_{\mathbf{X},\alpha-\sigma} $ we construct another control rough path $\boldsymbol{\tilde{\xi}} \in \mathscr{D}^{\gamma,p}_{\mathbf{X},\alpha-\sigma} $ shifting the Gubinelli derivatives of the original path accordingly. \\

    Therefore we obtain the next statement as a consequence of the interpolation inequality~\eqref{interpolation:ineq}. This entails a bound for the terms defining the norm~\eqref{YHN457} on ${\tilde{\cD}}^{\gamma,p}_{\mathbf{X},\alpha}$ in terms of the controlled functions specified in Definition~\ref{controlled}.
	\begin{comment}
		\begin{definition}\label{FDSS}
			For $\boldsymbol{\xi} \in  \mathscr{D}^{\gamma, p}_{\mathbf{X}}\left([0,T],(L(\R^d,E_{\alpha-j\gamma-\sigma}))_{0\leq j<N}\right) $ and $1\leq i,l\leq N$ with $l-i>1$
			\begin{align*}
				\begin{split}
					&\tilde{W}_{\tilde{R}^{i,l},\alpha,\gamma,p,1}(s,t):=	\sup_{\pi \in P[s,t]} \left\{\sum_{k} \frac{\left|R^{i-1,l-1}_{\tau_{k},\tau_{k+1}}\right|_{L\left((\R^d)^{\otimes i},E_{\alpha-l\gamma+\gamma}\right)}^{\frac{1}{(l-i-1)(\gamma-p)}}}{(\tau_{k+1}-\tau_k)^{\frac{p}{\gamma-p}}} \right\},\\
					&\tilde{W}_{\tilde{R}^{i,l},\alpha,\gamma,p,2}(s,t):=	\sup_{\pi \in P[s,t]} \left\{\sum_{k} \frac{\left|R^{i-1,l-1}_{\tau_{k},\tau_{k+1}}\right|_{L\left((\R^d)^{\otimes i},E_{\alpha-l\gamma}\right)}^{\frac{1}{(l-i)(\gamma-p)}}}{(\tau_{k+1}-\tau_k)^{\frac{p}{\gamma-p}}} \right\}.
				\end{split}
			\end{align*}
			Also, for $1\leq i< N$
			\begin{align*}
				&\tilde{W}_{R^{i,i+1},\alpha,\gamma,p}(s,t):=	\sup_{\pi \in P[s,t]} \left\{\sum_{k} \frac{\left|\delta\xi^{i-1}_{\tau_{k},\tau_{k+1}}\right|_{L\left((\R^d)^{\otimes i},E_{\alpha-i\gamma-\gamma}\right)}^{\frac{1}{\gamma-p}}}{(\tau_{k+1}-\tau_k)^{\frac{p}{\gamma-p}}} \right\}.
			\end{align*}
			As \eqref{NNN}, we define
			\begin{align}
				\tilde{W}_{R^{N-1,N},\alpha,\gamma,p,2}(s,t):= \tilde{W}_{R^{N-1,N},\alpha,\gamma,p}(s,t).
			\end{align}
		\end{definition}
	\end{comment}
	\begin{proposition}\label{ATSY}
		Assume \( \boldsymbol{\xi} \in \mathscr{D}^{\gamma, p}_{\mathbf{X},\alpha-\sigma}\left([0,T]\right)\). For $1\leq i<N$, we define 
		\begin{align*}
			&\tilde{\xi}^{i}:= [0,T] \to L\left((\R^d)^{\otimes i},L(\R^d,E_{\alpha-i\gamma})\right)=L\left((\R^d)^{\otimes i+1},E_{\alpha-i\gamma}\right),\\
			&\tilde{\xi}^{i}_\tau:=\xi^{i-1}_{\tau}.
		\end{align*}
       % \todo{ why do we gain $\sigma$ regularity in $\tilde{\xi}$ because we only shift the elements of $\xi$ that all have $\alpha-\sigma$ regularity? i totally see this in prop 2.17 because the convolution improves exactly this regularity.\textcolor{red}{we shift by $\gamma$ and since $\sigma<\gamma$ we can obtain these inequalities} i see now, it's just interpolation, thanks}
		For $1\leq i,l\leq N$ with $l-i>1$, we set $\tilde{R}^{i,l}(s,t):=R^{i-1,l-1}$.
		Then for $[s,t] \subseteq [0,T]$, the following properties hold. %\todo{please remove item and give a label like for an equation so that we can refer to it later}
		\begin{itemize}
			\item [\textnormal{\textbf{ I)}}] \label{Item I} 
			\begin{align*}
				\begin{split}
					&\sup_{\tau\in [s,t]}\left\vert\tilde{\xi}^{i}_{\tau}\right\vert_{L\left((\R^d)^{\otimes i},E_{\alpha-i\gamma}\right)}\lesssim \vert\xi^{i-1}_s\vert_{L\left((\R^d)^{\otimes i},E_{\alpha-(i-1)\gamma-\sigma}\right)}\\&+\left(W_{\Pi^{1}(\mathbf{X}),\gamma,p}^{\gamma-\sigma}(s,t)+(t-s)^{\gamma-\sigma}+(t-s)^{\frac{p(\gamma-\sigma)}{\gamma}}\right)\Vert\boldsymbol{\xi}\Vert_{\mathscr{D}^{\gamma, p}_{\mathbf{X},\alpha-\sigma}\left([0,T]\right)}.
				\end{split}
			\end{align*}
			\item [\textnormal{\textbf{ II)}}]	\label{Item II} 
			\begin{align*}
				\begin{split}
					\tilde{W}_{\tilde{R}^{i,l},\alpha,\gamma,p,1}(s,t)\lesssim(t-s)^{\frac{p(\gamma-\sigma)}{\gamma}}\left( W_{R^{i-1,l-1},\alpha-\sigma,\gamma,p,1}^{(l-i-1)(\gamma-p)}(s,t)\right)^{\frac{\sigma}{\gamma}} \left(W_{R^{i-1,l-1},\alpha-\sigma,\gamma,p,2}^{(l-i)(\gamma-p)}(s,t)\right)^{\frac{\gamma-\sigma}{\gamma}}. 
				\end{split}
			\end{align*} 
			\item[\textnormal{\textbf{ III)}}]	\label{Item III} 
			\begin{align*}
				\begin{split}
					\tilde{W}_{\tilde{R}^{i,l},\alpha,\gamma,p,2}(s,t)&\lesssim \sup_{\tau\in [s,t]}\vert\xi^{l-1}_{\tau}\vert_{\alpha-(l-1)\gamma-\sigma}W_{\Pi^{l-i}(\mathbf{X}),\gamma,p}^{(l-i)(\gamma-p)}(s,t)\\&+ (t-s)^{\frac{p(\gamma-\sigma)}{\gamma}}\left( W_{R^{i-1,l},\alpha-\sigma,\gamma,p,1}^{(l-i)(\gamma-p)}(s,t)\right)^{\frac{\sigma}{\gamma}} \left(W_{R^{i-1,l},\alpha-\sigma,\gamma,p,2}^{(l-i+1)(\gamma-p)}(s,t)\right)^{\frac{\gamma-\sigma}{\gamma}}.
				\end{split}
			\end{align*}
			\item[\textnormal{\textbf{ IV)}}]	\label{Item IV}  For $1\leq i<N$
			\begin{align*}	
				\tilde{W}_{\tilde{R}^{i,i+1},\alpha,\gamma,p}(s,t)&\lesssim \sup_{\tau \in [s,t]}\vert \xi_{\tau}^i \vert_{\alpha - i\gamma - \sigma}W_{\Pi^{1}(\mathbf{X}),\gamma,p}^{\gamma-p}(s,t)\\&+(t-s)^{\frac{p(\gamma-\sigma)}{\gamma}}\left( W_{R^{i-1,i+1},\alpha-\sigma,\gamma,p,1}^{\gamma-p}(s,t)\right)^{\frac{\sigma}{\gamma}} \left(W_{R^{i-1,i+1},\alpha-\sigma,\gamma,p,2}^{2(\gamma-p)}(s,t)\right)^{\frac{\gamma-\sigma}{\gamma}}.
			\end{align*}
		\end{itemize}
	\end{proposition}
	\begin{proof}
		Throughout the proof, we suppress for notational simplicity the tensor component of $\mathbf{\xi}$ and only write the index of the interpolation space in order to keep track of the spatial regularity.\\ % For example, we use \( \vert \cdot \vert_{\alpha - i \gamma} \) instead of \( |\cdot|_{L((\mathbb{R}^d)^{\otimes (i+1)}, E_{\alpha - i \gamma})} \).\newline
		We begin with \hyperref[Item I]{\textbf{I)}}.
		Note that
		\begin{align}\label{BNBZV}
			\delta \xi^{i-1}_{s,t}=\xi^{i}_s\circ (\Pi^{1}(\mathbf{X}))_{s,t}+R^{i-1,i+1}_{s,t}.
		\end{align}
		Therefore
		\begin{align*}
			\left\vert\delta{\xi}^{i-1}_{s,t}\right\vert_{\alpha-i\gamma-\sigma}\lesssim (t-s)^pW_{\Pi^{1}(\mathbf{X}),\gamma,p}^{\gamma-p}(s,t)\left\vert\xi^i_{s}\right\vert_{\alpha-i\gamma-\sigma}+(t-s)^{p}W_{R^{i-1,i+1},\alpha-\sigma,\gamma,p,1}^{\gamma-p}(s,t).
		\end{align*}
		So, by the interpolation inequality~\eqref{interpolation:ineq} for $\tau\in [s,t]$ we have
		\begin{align*}
			\vert\tilde{\xi}^{i}_{\tau}\vert_{\alpha-i\gamma}&\lesssim \vert\xi^{i-1}_{s}\vert_{\alpha-i\gamma+\gamma-\sigma}+\left\vert\delta{\xi}^{i-1}_{s,\tau}\right\vert_{\alpha-i\gamma-\sigma}^{\frac{\gamma-\sigma}{\gamma}}\left\vert\delta{\xi}^{i-1}_{s,\tau}\right\vert_{\alpha-i\gamma+\gamma-\sigma}^{\frac{\sigma}{\gamma}}\\&\lesssim \vert\xi^{i-1}_s\vert_{\alpha-i\gamma+\gamma-\sigma}+\left((t-s)^{\frac{p(\gamma-\sigma)}{\gamma}}W_{\Pi^{1}(\mathbf{X}),\gamma,p}^{\frac{(\gamma-p)(\gamma-\sigma)}{\gamma}}(s,t)+(t-s)^{\frac{p(\gamma-\sigma)}{\gamma}}\right)\Vert\boldsymbol{\xi}\Vert_{\mathscr{D}^{\gamma, p}_{\mathbf{X},\alpha-\sigma}\left([0,T]\right)}\\&\lesssim \vert\xi^{i-1}_s\vert_{\alpha-i\gamma+\gamma-\sigma}+\left(W_{\Pi^{1}(\mathbf{X}),\gamma,p}^{\gamma-\sigma}(s,t)+(t-s)^{\gamma-\sigma}+(t-s)^{\frac{p(\gamma-\sigma)}{\gamma}}\right)\Vert\boldsymbol{\xi}\Vert_{\mathscr{D}^{\gamma, p}_{\mathbf{X},\alpha-\sigma}\left([0,T]\right)}.
		\end{align*}
		We now show	\hyperref[Item I]{\textbf{II)}}. First, note that for every \( \tau_1 < \tau_2 \), we have  
		\[
		\tilde{R}^{i,l}_{\tau_1,\tau_2} = R^{i-1,l-1}_{\tau_1,\tau_2}
		\]  
		as specified in Definition \ref{controlled2}. Since we want to estimate \( \tilde{W}_{\tilde{R}^{i,l},\alpha,\gamma,p,1} \), our strategy is first to estimate \( \tilde{R}^{i,l}_{\tau_1,\tau_2} \) on an arbitrary interval $[\tau_1,\tau_2]$ and then extend this result to a partition of $[s,t]$.
		Again from the interpolation inequality~\eqref{interpolation:ineq}, we have 
		\begin{align*}
			\left|R^{i-1,l-1}_{\tau_{1},\tau_{2}}\right|_{\alpha-l\gamma+\gamma}&\lesssim \left|R^{i-1,l-1}_{\tau_{1},\tau_{2}}\right|_{\alpha-l\gamma+2\gamma-\sigma}^{\frac{\sigma}{\gamma}}\left|R^{i-1,l-1}_{\tau_{1},\tau_{2}}\right|_{\alpha-l\gamma+\gamma-\sigma}^{\frac{\gamma-\sigma}{\gamma}}.
		\end{align*}
		Therefore we conclude
		\begin{align}\label{SADsdsg}
			\frac{\left|R^{i-1,l-1}_{\tau_{1},\tau_{2}}\right|_{\alpha-l\gamma+\gamma}^{\frac{1}{(l-i-1)(\gamma-p)}}}{(\tau_{2}-\tau_1)^{\frac{p}{\gamma-p}}} \lesssim \left(\frac{\left|R^{i-1,l-1}_{\tau_{1},\tau_{2}}\right|_{\alpha-l\gamma+2\gamma-\sigma}^{\frac{1}{(l-i-1)(\gamma-p)}}}{(\tau_{2}-\tau_1)^{\frac{p}{\gamma-p}}}\right)^{\frac{\sigma}{\gamma}}\left(\frac{\left|R^{i-1,l-1}_{\tau_{1},\tau_{2}}\right|_{\alpha-l\gamma+\gamma-\sigma}^{\frac{1}{(l-i-1)(\gamma-p)}}}{(\tau_{2}-\tau_1)^{\frac{p}{\gamma-p}}}\right)^{\frac{\gamma-\sigma}{\gamma}}.
		\end{align}
		Consequently, for every arbitrary partition \( \pi = \lbrace s = \tau_0 < \tau_1 < \ldots < \tau_m = t \rbrace \), we derive using the interpolation inequality~\eqref{interpolation:ineq} {and \eqref{SADsdsg}) that}%\todo{explanations for the computations on the next pages}
		\begin{align}\label{AXZ}
			\begin{split}
				\sum_{k}\frac{\left|R^{i-1,l-1}_{\tau_{k},\tau_{k+1}}\right|_{\alpha-l\gamma+\gamma}^{\frac{1}{(l-i-1)(\gamma-p)}}}{(\tau_{k+1}-\tau_k)^{\frac{p}{\gamma-p}}}&\lesssim\sum_{k} \left(\frac{\left|R^{i-1,l-1}_{\tau_{k},\tau_{k+1}}\right|_{\alpha-l\gamma+2\gamma-\sigma}^{\frac{1}{(l-i-1)(\gamma-p)}}}{(\tau_{k+1}-\tau_k)^{\frac{p}{\gamma-p}}}\right)^{\frac{\sigma}{\gamma}}\left(\frac{\left|R^{i-1,l-1}_{\tau_{k},\tau_{k+1}}\right|_{\alpha-l\gamma+\gamma-\sigma}^{\frac{1}{(l-i-1)(\gamma-p)}}}{(\tau_{k+1}-\tau_k)^{\frac{p}{\gamma-p}}}\right)^{\frac{\gamma-\sigma}{\gamma}}\\&\leq \left(\underbrace{\sum_{k} \left(\frac{\left|R^{i-1,l-1}_{\tau_{k},\tau_{k+1}}\right|_{\alpha-l\gamma+2\gamma-\sigma}^{\frac{1}{(l-i-1)(\gamma-p)}}}{(\tau_{k+1}-\tau_k)^{\frac{p}{\gamma-p}}}\right)}_{\leq W_{R^{i-1,l-1},\alpha-\sigma,\gamma,p,1}(s,t) }\right)^{\frac{\sigma}{\gamma}}\left(\sum_{k}\left(\frac{\left|R^{i-1,l-1}_{\tau_{k},\tau_{k+1}}\right|_{\alpha-l\gamma+\gamma-\sigma}^{\frac{1}{(l-i-1)(\gamma-p)}}}{(\tau_{k+1}-\tau_k)^{\frac{p}{\gamma-p}}}\right)\right)^{\frac{\gamma-\sigma}{\gamma}},
			\end{split}
		\end{align}
		where, in the last step, we used again Hölder's inequality. Now, we estimate the last term in \eqref{AXZ} as follows
		\begin{align}\label{DRA}
			\begin{split}
				&\sum_{k}\left(\frac{\left|R^{i-1,l-1}_{\tau_{k},\tau_{k+1}}\right|_{\alpha-l\gamma+\gamma-\sigma}^{\frac{1}{(l-i-1)(\gamma-p)}}}{(\tau_{k+1}-\tau_k)^{\frac{p}{\gamma-p}}}\right)= \sum_{k}(\tau_{k+1}-\tau_k)^{\frac{p}{(\gamma-p)(l-i-1)}}\left(\frac{\left|R^{i-1,l-1}_{\tau_{k},\tau_{k+1}}\right|_{\alpha-l\gamma+\gamma-\sigma}^{\frac{1}{(l-i)(\gamma-p)}}}{(\tau_{k+1}-\tau_k)^{\frac{p}{\gamma-p}}}\right)^{\frac{l-i}{l-i-1}}\\&\leq (t-s)^{\frac{p}{(\gamma-p)(l-i-1)}}\sum_{k}\left(\frac{\left|R^{i-1,l-1}_{\tau_{k},\tau_{k+1}}\right|_{\alpha-l\gamma+\gamma-\sigma}^{\frac{1}{(l-i)(\gamma-p)}}}{(\tau_{k+1}-\tau_k)^{\frac{p}{\gamma-p}}}\right)^{\frac{l-i}{l-i-1}}\\&\leq (t-s)^{\frac{p}{(\gamma-p)(l-i-1)}}\left(\sum_{k}\frac{\left|R^{i-1,l-1}_{\tau_{k},\tau_{k+1}}\right|_{\alpha-l\gamma+\gamma-\sigma}^{\frac{1}{(l-i)(\gamma-p)}}}{(\tau_{k+1}-\tau_k)^{\frac{p}{\gamma-p}}}\right)^{\frac{l-i}{l-i-1}}\leq (t-s)^{\frac{p}{(\gamma-p)(l-i-1)}}W_{R^{i-1,l-1},\alpha-\sigma,\gamma,p,2}^{\frac{l-i}{l-i-1}}(s,t),
			\end{split}
		\end{align}
		{where in the step before the last, we used Lemma \ref{INM} with $p_1=\frac{l-i}{l-i-1}$ and $p_2=1$.} Consequently, from \eqref{AXZ} and \eqref{DRA}
		\begin{align*}
			\tilde{W}_{\tilde{R}^{i,l},\alpha,\gamma,p,1}^{(l-i-1)(\gamma-p)}(s,t)\lesssim (t-s)^{\frac{p(\gamma-\sigma)}{\gamma}}\left( W_{R^{i-1,l-1},\alpha-\sigma,\gamma,p,1}^{(l-i-1)(\gamma-p)}(s,t)\right)^{\frac{\sigma}{\gamma}} \left(W_{R^{i-1,l-1},\alpha-\sigma,\gamma,p,2}^{(l-i)(\gamma-p)}(s,t)\right)^{\frac{\gamma-\sigma}{\gamma}},
		\end{align*}
		which proves \hyperref[Item II]{\textbf{II)}}.\newline
		We now focus on		\hyperref[Item III]{\textbf{III)}}.
		To estimate the  $W_{\tilde{R}^{i,l},\alpha,\gamma,p,2}^{l-i-1}(s,t)$, we first observe that
		\begin{align}\label{YBVASs}
			R^{i-1,l-1}_{\tau_1,\tau_2} = R^{i-1,l}_{\tau_1,\tau_2} + \xi_{\tau_1}^{l-1} \circ \big(\Pi^{l-i}(\mathbf{X})\big)_{\tau_1,\tau_2}.
		\end{align}
		Similar to the previous case using the interpolation inequality~\eqref{interpolation:ineq}
		\begin{align}\label{HBNMMZ}
			\left|R^{i-1,l}_{\tau_{1},\tau_{2}}\right|_{\alpha-l\gamma}&\lesssim \left|R^{i-1,l}_{\tau_1,\tau_2}\right|_{\alpha-l\gamma+\gamma-\sigma}^{\frac{\sigma}{\gamma}}\left|R^{i-1,l}_{\tau_{1},\tau_{2}}\right|_{\alpha-l\gamma-\sigma}^{\frac{\gamma-\sigma}{\gamma}}.
		\end{align}
		Therefore,
		\begin{align*}
			\frac{\left|R^{i-1,l}_{\tau_{1},\tau_{2}}\right|_{\alpha-l\gamma}^{\frac{1}{(l-i)(\gamma-p)}}}{(\tau_{2}-\tau_1)^{\frac{p}{\gamma-p}}} \lesssim \left(\frac{\left|R^{i-1,l}_{\tau_{1},\tau_{2}}\right|_{\alpha-l\gamma+\gamma-\sigma}^{\frac{1}{(l-i)(\gamma-p)}}}{(\tau_{2}-\tau_1)^{\frac{p}{\gamma-p}}}\right)^{\frac{\sigma}{\gamma}}\left(\frac{\left|R^{i-1,l}_{\tau_{1},\tau_{2}}\right|_{\alpha-l\gamma-\sigma}^{\frac{1}{(l-i)(\gamma-p)}}}{(\tau_{2}-\tau_1)^{\frac{p}{\gamma-p}}}\right)^{\frac{\gamma-\sigma}{\gamma}}.
		\end{align*}
		We then argue similarly as in \eqref{AXZ} and \eqref{DRA} to conclude
		\begin{align}\label{BNMio}
			\begin{split}
				&\underbrace{\left(	\sup_{\pi \in P[s,t]} \left\{\sum_{k} \frac{\left|R^{i-1,l}_{\tau_{k},\tau_{k+1}}\right|_{\alpha-l\gamma}^{\frac{1}{(l-i)(\gamma-p)}}}{(\tau_{k+1}-\tau_k)^{\frac{p}{\gamma-p}}} \right\}\right)^{(l-i)(\gamma-p)}}_{\tilde{\tilde{W}}_{R^{i-1,l},\alpha,\gamma,p}^{(l-i)(\gamma-p)}(s,t)}\\&\lesssim (t-s)^{\frac{p(\gamma-\sigma)}{\gamma}}\left( W_{R^{i-1,l},\alpha-\sigma,\gamma,p,1}^{(l-i)(\gamma-p)}(s,t)\right)^{\frac{\sigma}{\gamma}} \left(W_{R^{i-1,l},\alpha-\sigma,\gamma,p,2}^{(l-i+1)(\gamma-p)}(s,t)\right)^{\frac{\gamma-\sigma}{\gamma}}.
			\end{split}
		\end{align}
		Consequently, from \eqref{YBVASs}, the Minkowski inequality and \eqref{BNMio} we infer that
		\begin{align*}
			\tilde{W}_{\tilde{R}^{i,l},\alpha,\gamma,p,2}^{(l-i)(\gamma-p)}(s,t)&\leq \tilde{\tilde{W}}_{R^{i-1,l},\alpha,\gamma,p}^{(l-i)(\gamma-p)}(s,t)+\sup_{\tau\in [s,t]}\vert\xi^{l-1}_{\tau}\vert_{\alpha-l\gamma}W_{\Pi^{l-i}(\mathbf{X}),\gamma,p}^{(l-i)(\gamma-p)}(s,t)\\&\lesssim (t-s)^{\frac{p(\gamma-\sigma)}{\gamma}}\left( W_{R^{i-1,l},\alpha-\sigma,\gamma,p,1}^{(l-i)(\gamma-p)}(s,t)\right)^{\frac{\sigma}{\gamma}} \left(W_{R^{i-1,l},\gamma,\alpha-\sigma,p,2}^{(l-i+1)(\gamma-p)}(s,t)\right)^{\frac{\gamma-\sigma}{\gamma}}\\&+\sup_{\tau\in [s,t]}\vert\xi^{l-1}_{\tau}\vert_{\alpha-(l-1)\gamma-\sigma}W_{\Pi^{l-i}(\mathbf{X}),\gamma,p}^{(l-i)(\gamma-p)}(s,t).
		\end{align*}
		We finally verify \hyperref[Item IV]{\textbf{IV)}}. This item is very similar to the previous case. Recalling \eqref{BNMio} and the definition of \( \tilde{\tilde{W}}_{R^{i-1,i+1}} \) therein, we obtain  
         \begin{align*}  	
			&\left(\sum_{k} \frac{\left|R^{i-1,i+1}_{\tau_{k},\tau_{k+1}}\right|_{\alpha-i\gamma-\gamma}^{\frac{1}{\gamma-p}}}{(\tau_{k+1}-\tau_k)^{\frac{p}{\gamma-p}}}\right)^{\gamma-p}\\ &\quad\leq  (t-s)^{\frac{p(\gamma-\sigma)}{\gamma}}\left( W_{R^{i-1,i+1},\alpha-\sigma,\gamma,p,1}^{\gamma-p}(s,t)\right)^{\frac{\sigma}{\gamma}} \left(W_{R^{i-1,i+1},\alpha-\sigma,\gamma,p,2}^{2(\gamma-p)}(s,t)\right)^{\frac{\gamma-\sigma}{\gamma}}.
        \end{align*}
        Then, from \eqref{BNBZV}, the inequality  
\[
\sup_{\tau \in [s,t]} \vert \xi_{s}^i \vert_{\alpha - i\gamma - \gamma} \lesssim \sup_{\tau \in [s,t]} \vert \xi_{s}^i \vert_{\alpha - i\gamma - \sigma},
\]
and the Minkowski inequality, we deduce  
		\begin{align*}
			\tilde{W}_{R^{i,i+1},\alpha,\gamma,p}(s,t)&\lesssim \sup_{\tau \in [s,t]}\vert \xi_{s}^i \vert_{\alpha - i\gamma - \sigma}W_{\Pi^{1}(\mathbf{X}),\gamma,p}^{\gamma-p}(s,t)\\&+(t-s)^{\frac{p(\gamma-\sigma)}{\gamma}}\left( W_{R^{i-1,i+1},\alpha-\sigma,\gamma,p,1}^{\gamma-p}(s,t)\right)^{\frac{\sigma}{\gamma}} \left(W_{R^{i-1,i+1},\alpha-\sigma,\gamma,p,2}^{2(\gamma-p)}(s,t)\right)^{\frac{\gamma-\sigma}{\gamma}}.
		\end{align*}
		This proves the last statement.
	\end{proof}
	Now we focus on the rough integral. 
	\begin{proposition}\label{UJMMA}
		Let $\boldsymbol{\xi} \in \mathscr{D}^{\gamma, p}_{\mathbf{X},\alpha-\sigma}\left([0,T]\right)$ and assume the setting of Theorem \ref{shdccsa}. We further let $[s,t]\subseteq [0,T]$ and define for \( a \in [s,t] \)  
		\begin{align*}
			\tilde{\xi}^{0}_{a} := \int_{s}^{a} U_{a,v} \xi_{v} \circ \mathrm{d}\mathbf{X}_{v}.
		\end{align*}
		In addition, for \( 1 \leq j < N \),  we set  
		\begin{align*}
			\tilde{\xi}^{j}_{a} := \xi^{j-1}_{a}.
		\end{align*}
		Furthermore, for \( [a,b] \subseteq [s,t] \) and \( 1 < l \leq N \)  
		\begin{align*}
			\tilde{R}^{0,l}_{a,b} := \delta\tilde{\xi}^{0}_{a,b} - \sum_{0 < j < l} \tilde{\xi}^{j}_{a}\circ(\Pi^{j}(\mathbf{X}))_{a,b}.
		\end{align*}
		Similar to Definition \ref{controlled2}, we define
		{\small
			\begin{align*}
				\tilde{W}_{\tilde{R}^{0,l},\alpha,\gamma,p,1}(s,t):=&\sup_{\pi \in P[s,t]} \left\{\sum_{k} \frac{\left|\tilde{R}^{0,l}_{\tau_{k},\tau_{k+1}}\right|_{E_{\alpha-(l-1)\gamma}}^{\frac{1}{(l-1)(\gamma-p)}}}{(\tau_{k+1}-\tau_k)^{\frac{p}{\gamma-p}}} \right\},\ 
				\tilde{W}_{\tilde{R}^{0,l},\alpha,\gamma,p,2}(s,t):=	\sup_{\pi \in P[s,t]} \left\{\sum_{k} \frac{\left|\tilde{R}^{0,l}_{\tau_{k},\tau_{k+1}}\right|_{E_{\alpha-l\gamma}}^{\frac{1}{l(\gamma-p)}}}{(\tau_{k+1}-\tau_k)^{\frac{p}{\gamma-p}}} \right\},
				\\& \tilde{W}_{\tilde{R}^{0,1},\alpha,\gamma,p}(s,t):=	\sup_{\pi \in P[s,t]} \left\{\sum_{k} \frac{\left|  \delta \tilde\xi^0_{\tau_{k},\tau_{k+1}}\right|_{E_{\alpha-\gamma}}^{\frac{1}{\gamma-p}}}{(\tau_{k+1}-\tau_k)^{\frac{p}{\gamma-p}}} \right\}.
		\end{align*}}
		Then we have
		\begin{align}\label{52}
			\begin{split}
				&\max\left\lbrace \tilde{W}_{\tilde{R}^{0,l},\alpha,\gamma,p,1}(s,t),	\tilde{W}_{\tilde{R}^{0,l},\alpha,\gamma,p,2}(s,t),	\tilde{W}_{\tilde{R}^{0,1},\alpha,\gamma,p}(s,t)\right\rbrace\\&\lesssim \sum_{0\leq j<N}\left( W_{\Pi^{j+1}(\mathbf{X}),\gamma,p}^{(j+1)(\gamma-p)}(s,t) W_{R^{j,N},\alpha-\sigma,\gamma,p,2}^{(N-j)(\gamma-p)}(s,t)+W_{\Pi^{j+1}(\mathbf{X}),\gamma,p}^{(j+1)(\gamma-p)}(s,t)\sup_{\tau\in [s,t]}\left\vert\xi^{j}_{\tau}\right\vert_{\alpha-j\gamma-\sigma}\right).
			\end{split}
		\end{align}
	\end{proposition}
	\begin{proof}
		Throughout the proof, we follow the same convention as in Proposition \ref{ATSY} for denoting norms. Specifically, we omit the tensor component and only indicate the spatial regularity of the Banach scale. We assume that \([a, b] \subseteq [s,t]\), and let \(\pi = \{s = \tau_0 < \tau_1 < \cdots < \tau_m = t\}\) be an arbitrary partition of \([s, t]\). Moreover, recalling from Assumption \ref{ADssdfg} we have that for every \(0 \leq j \leq N\) 
		\begin{align*}  
			P_{j} = (j+1)p - j\gamma - \sigma > 0.  
		\end{align*}
		For every $1\leq l\leq N$ we have 
		\begin{align*}
			&\delta \tilde{\xi}^{0}_{a, b}=\int_{s}^{b}U_{b,v}\xi_{v}\circ\mathrm{d}\mathbf{X}_{v}-\int_{s}^{a}U_{a,v}\xi_{v}\circ\mathrm{d}\mathbf{X}_{v}\\&\quad=\int_{s}^{a}(U_{b,v}-U_{a,v})\xi_{v} \circ \mathrm{d}\mathbf{X}_{v}+\int_{a}^{b}U_{b,v}\xi_{v}\circ\mathrm{d}\mathbf{X}_{v}=(U_{b,a}-I)\int_{s}^{a}U_{a,v}\xi_{v}\circ\mathrm{d}\mathbf{X}_{v}+\int_{a}^{b}U_{b,v}\xi_{v}\circ\mathrm{d}\mathbf{X}_{v}\\&\myquad[2]=\underbrace{(U_{b,a}-I)\int_{s}^{a}U_{a,v}\xi_{v}\circ\mathrm{d}\mathbf{X}_{v}}_{C_{l}(u,b)}+\underbrace{\left(\int_{a}^{b}U_{b,v}\xi_{v}\circ\mathrm{d}\mathbf{X}_{v}-\sum_{0\leq j<N}U_{b,a}\xi^{j}_{a}\circ(\Pi^{j+1}(\mathbf{X}))_{a,b}\right)}_{D_{l}(a,b)}\\&\myquad[3]+\underbrace{\sum_{0\leq j<l-1}\left((U_{b,a}-I)\xi^{j}_{a}\circ(\Pi^{j+1}(\mathbf{X}))_{a,b}\right)}_{F_{l}(a,b)}+U_{b,a}\xi^{l-1}_{a}\circ(\Pi^{l}(\mathbf{X}))_{a,b}\\&\myquad[4]+\underbrace{\sum_{l-1< j<N}U_{b,a}\xi^{j}_{a}\circ(\Pi^{j+1}(\mathbf{X}))_{a,b}}_{H_{l}(a,b)}+\underbrace{\sum_{0\leq j<l-1}\xi^{j}_{a}\circ(\Pi^{j+1}(\mathbf{X}))_{a,b}}_{=\sum_{0<j<l}\tilde{\xi}^{j}_{a}(\Pi^{j}\circ(\mathbf{X}))_{a,b}}.
		\end{align*}
		Consequently, for $1\leq l\leq N$ and $r\in\lbrace 0,1\rbrace$
		\begin{align}\label{A11}
			\begin{split}
				\vert \tilde{R}^{0,l}_{u,b}\vert_{\alpha-(l-r)\gamma}&\leq \vert C_{l}(u,b)\vert_{\alpha-(l-r)\gamma}+\vert D_{l}(u,b)\vert_{\alpha-(l-r)\gamma}\\&+\vert U_{b,a}\xi^{l-1}_{a}\circ(\Pi^{l}(\mathbf{X}))_{a,b}\vert_{\alpha-(l-r)\gamma}+\vert F_{l}(u,b)t\vert_{\alpha-(l-r)\gamma}+\vert H_{l}(u,b)\vert_{\alpha-(l-r)\gamma}.
			\end{split}
		\end{align}
		We now estimate each of these terms separately.\newline %\todo{ it would be good to remove the bullet points,\textcolor{red}{i see you point, I used another way, A), B),...}}\newline
			\textbf{A)} To estimate \(F_{l}\) in the \(E_{\alpha-(l-r)\gamma}\)-norm, first recall that \(\xi^{j} \in L((\mathbb{R}^d)^{\otimes j+1}, E_{\alpha - j\gamma - \sigma})\). Thus, for \(0 \leq j < l - 1\) and \(r \in \{0, 1\}\) we have
			\begin{align*}
				\alpha-j\gamma-\sigma> \alpha-(l-r)\gamma.
			\end{align*}
			Thus, it follows from \eqref{regularity} that  
			\begin{align*}  
				\left\vert (U_{b,a} - I)\xi^{j}_{a} \right\vert_{\alpha - (l-r)\gamma} \lesssim (b-a)^{(l-j-r)\gamma - \sigma} \vert \xi^{j}_{a} \vert_{\alpha - j\gamma - \sigma}.  
			\end{align*}  
			Consequently, from the definition of \(F_{l}\), we obtain
			\begin{align*}
				\left\vert F_{l}(a,b) \right \vert_{\alpha-(l-r)\gamma}&\lesssim\sum_{0\leq j<l-1} (b-a)^{(l-j-r)\gamma-\sigma}(b-a)^{(j+1)p}W_{\Pi^{j+1}(\mathbf{X}),\gamma,p}^{(j+1)(\gamma-p)}(a,b)\vert\xi^{j}_{a}\vert_{\alpha-j\gamma-\sigma}\\& =(b-a)^{(l-r)p}\sum_{0\leq j<l-1}(b-a)^{(l-r)(\gamma-p)+P_{j}}W_{\Pi^{j+1}(\mathbf{X}),\gamma,p}^{(j+1)(\gamma-p)}(a,b)\vert\xi^{j}_{a}\vert_{\alpha-j\gamma-\sigma}.
			\end{align*}
			{This, together with Lemma \ref{INM}, yields}
			\begin{align*}
				\frac{\left\vert F_{l}(a,b) \right \vert_{\alpha-(l-r)\gamma}^{\frac{1}{(l-r)(\gamma-p)}}}{(b-a)^{\frac{p}{\gamma-p}}}\lesssim \sum_{0\leq j<l-1}(b-a)^{1+\frac{P_j}{(l-r)(\gamma-p)}}W_{\Pi^{j+1}(\mathbf{X}),\gamma,p}^{\frac{j+1}{l-r}}(a,b)\vert\xi^{j}_{a}\vert_{\alpha-j\gamma-\sigma}^{\frac{1}{(l-r)(\gamma-p)}}.
			\end{align*}
			{So, since the latter inequality holds for every \( a,b \in  [0, T] \), it follows from Lemma \ref{INM} that}%\todo{more explanations}
			\begin{align}\label{B11} 
				\begin{split}
					&\left(\sum_{k}\frac{\left\vert F_{l}(\tau_k,\tau_{k+1})  \right \vert_{\alpha-(l-r)\gamma}^{\frac{1}{(l-r)(\gamma-p)}}}{(\tau_{k+1}-\tau_{k})^{\frac{p}{\gamma-p}}}\right)^{(l-r)(\gamma-p)}\\&\myquad[1]\lesssim\left(\sum_{0\leq j<l-1}\sum_{k}  (\tau_{k+1}-\tau_{k})^{1+\frac{P_j}{(l-r)(\gamma-p)}}W_{\Pi^{j+1}(\mathbf{X}),\gamma,p}^{\frac{j+1}{l-r}}(\tau_{k},\tau_{k+1})\vert\xi^{j}_{\tau_{k}}\vert_{\alpha-j\gamma-\sigma}^{\frac{1}{(l-r)(\gamma-p)}}\right)^{(l-r)(\gamma-p)} \\&\myquad[2]\lesssim\sum_{0\leq j<l-1}\left(\sum_k(\tau_{k+1}-\tau_{k})^{1+\frac{P_j}{(l-r)(\gamma-p)}}W_{\Pi^{j+1}(\mathbf{X}),\gamma,p}^{\frac{j+1}{l-r}}(\tau_{k},\tau_{k+1})\vert\xi^{j}_{\tau_{k}}\vert_{\alpha-j\gamma-\sigma}^{\frac{1}{(l-r)(\gamma-p)}}\right)^{(l-r)(\gamma-p)}\\&\myquad[3]\lesssim\sum_{0\leq j<l-1}\sup_{\tau\in [s,t]}\vert\xi^{j}_\tau\vert_{\alpha-j\gamma-\sigma}W_{\Pi^{j+1}(\mathbf{X}),\gamma,p}^{(j+1)(\gamma-p)}(s,t)\left(\sum_k(\tau_{k+1}-\tau_{k})^{1+\frac{P_j}{(l-r)(\gamma-p)}}\right)^{(l-r)(\gamma-p)}\\
					&\myquad[4]\leq \sum_{0\leq j<l-1}(t-s)^{(l-r)(\gamma-p)+P_j}W_{\Pi^{j+1}(\mathbf{X}),\gamma,p}^{(j+1)(\gamma-p)}(s,t)\sup_{\tau\in [s,t]}\vert\xi^{j}_\tau\vert_{\alpha-j\gamma-\sigma}\\&\myquad[5]\lesssim  \sum_{0\leq j<l-1}W_{\Pi^{j+1}(\mathbf{X}),\gamma,p}^{(j+1)(\gamma-p)}(s,t)\sup_{\tau\in [s,t]}\vert\xi^{j}_\tau\vert_{\alpha-j\gamma-\sigma} .
				\end{split}
			\end{align}
			\textbf{B)} The estimate for \(H_l\) is quite similar to that for \(F_l\). For the reader's convenience, we provide the details. Note that for every $l-1< j<N$ and $r\in\lbrace 0,1\rbrace$ we have that
			\begin{align*}
				\alpha-j\gamma-\sigma\leq \alpha-(l-r)\gamma,
			\end{align*}
			so from \eqref{regularity}
			\begin{align*}
				\left\vert U_{b,a}\xi^{j}_{a}\right\vert_{\alpha-(l-r)\gamma}\lesssim (b-a)^{(l-r-j)\gamma-\sigma}\vert\xi^{j}_{a}\vert_{\alpha-j\gamma-\sigma}.
			\end{align*}
			This further leads to
			\begin{align*}
				\left\vert H_{l}(a,b) \right \vert_{\alpha-(l-r)\gamma}&\lesssim\sum_{l-1< j<N}(b-a)^{(l-r-j)\gamma-\sigma}(b-a)^{(j+1)p}W_{\Pi^{j+1}(\mathbf{X}),\gamma,p}^{(j+1)(\gamma-p)}(a,b)\vert\xi^{j}_{a}\vert_{\alpha-j\gamma-\sigma}\\
				&=(b-a)^{(l-r)p}\sum_{l-1< j<N}(b-a)^{P_j+(l-r)(\gamma-p)}W_{\Pi^{j+1}(\mathbf{X}),\gamma,p}^{(j+1)(\gamma-p)}(a,b)\vert\xi^{j}_{a}\vert_{\alpha-j\gamma-\sigma}.
			\end{align*}
			{Thus, from Lemma \ref{INM}, we get}
			\begin{align}\label{sydv}
				\begin{split}
					\frac{\left\vert H_{l}(a,b)  \right \vert_{\alpha-(l-r)\gamma}^{\frac{1}{(l-r)(\gamma-p)}}}{(b-a)^{\frac{p}{\gamma-p}}}\lesssim \sum_{l-1< j<N}(b-a)^{\frac{P_j+(l-r)(\gamma-p)}{(l-r)(\gamma-p)}}W_{\Pi^{j+1}(\mathbf{X}),\gamma,p}^{\frac{j+1}{l-r}}(a,b)\vert\xi^{j}_{a}\vert_{\alpha-j\gamma-\sigma}^{\frac{1}{(l-r)(\gamma-p)}}
				\end{split}
			\end{align}
			For every $l-1< j<N$ and $r\in \lbrace 0,1\rbrace$, we have 
			\begin{align*}
				\frac{j+1}{l-r}>1.
			\end{align*}
			{Note that inequality \eqref{sydv} holds for every \( [a ,b]\subseteq [0, T] \). Thus, from \eqref{INM}, and using the fact that \( W_{\Pi^{j+1}(\mathbf{X}), \gamma, p} \) is a control function, we conclude that}%\todo{explanations\textcolor{red}{i added another line to \ref{C11} }} \
				\begin{align}\label{C11}
				\begin{split}
					&\left(\sum_{k}\frac{\left\vert H_{l}(\tau_k,\tau_{k+1})  \right \vert_{\alpha-(l-r)\gamma}^{\frac{1}{(l-r)(\gamma-p)}}}{(\tau_{k+1}-\tau_{k})^{\frac{p}{\gamma-p}}}\right)^{(l-r)(\gamma-p)}\\&\myquad[1]\lesssim\left(\sum_{l-1< j<N}\sum_k(\tau_{k+1}-\tau_k)^{{\frac{(P_j+(l-r)(\gamma-p)}{(l-r)(\gamma-p)}}}W_{\Pi^{j+1}(\mathbf{X}),\gamma,p}^{\frac{j+1}{l-r}}(\tau_k,\tau_{k+1})\vert\xi^{j}_{\tau_{k}}\vert_{\alpha-j\gamma-\sigma}^{\frac{1}{(l-r)(\gamma-p)}}\right)^{(l-r)(\gamma-p)}\\&\myquad[2] \lesssim\sum_{l-1< j<N}\left(\sum_k(\tau_{k+1}-\tau_k)^{{\frac{(P_j+(l-r)(\gamma-p)}{(l-r)(\gamma-p)}}}W_{\Pi^{j+1}(\mathbf{X}),\gamma,p}^{\frac{j+1}{l-r}}(\tau_k,\tau_{k+1})\vert\xi^{j}_{\tau_{k}}\vert_{\alpha-j\gamma-\sigma}^{\frac{1}{(l-r)(\gamma-p)}}\right)^{(l-r)(\gamma-p)} \\
					&\myquad[3]\leq \sum_{l-1< j<N}(t-s)^{P_j+(l-r)(\gamma-p)}\sup_{\tau\in [s,t]}\vert\xi^{j}_\tau\vert_{\alpha-j\gamma-\sigma}\left(\sum_{k}W_{\Pi^{j+1}(\mathbf{X}),\gamma,p}^{\frac{j+1}{l-r}}(\tau_k,\tau_{k+1})\right)^{(l-r)(\gamma-p)}\\&\myquad[4]\leq \sum_{l-1< j<N}(t-s)^{P_j+(l-r)(\gamma-p)}\sup_{\tau\in [s,t]}\vert\xi^{j}_\tau\vert_{\alpha-j\gamma-\sigma}\left(\sum_{k}W_{\Pi^{j+1}(\mathbf{X}),\gamma,p}(\tau_k,\tau_{k+1})\right)^{(j+1)(\gamma-p)}\\&\myquad[5]\leq \sum_{l-1< j<N}(t-s)^{P_j+(l-r)(\gamma-p)}W_{\Pi^{j+1}(\mathbf{X}),\gamma,p}^{(j+1)(\gamma-p)}(s,t)\sup_{\tau\in [s,t]}\vert\xi^{j}_\tau\vert_{\alpha-j\gamma-\sigma}\\&\myquad[6]\lesssim \sum_{l-1< j<N}W_{\Pi^{j+1}(\mathbf{X}),\gamma,p}^{(j+1)(\gamma-p)}(s,t)\sup_{\tau\in [s,t]}\vert\xi^{j}_\tau\vert_{\alpha-j\gamma-\sigma}.
				\end{split}
			\end{align}
			\item Since $\xi^{l-1}\in L((\mathbb{R}^d)^{\otimes l}, E_{\alpha - (l-1)\gamma - \sigma})$, we have the following statements.
			\begin{itemize}
				\item Since $\alpha - (l-1)\gamma - \sigma\geq \alpha -l\gamma$, we therefore have
				\begin{align*}
					\left\vert\left(U_{b,a}\xi^{l-1}_{a}\circ(\Pi^{l}(\mathbf{X}))_{a,b}\right)\right\vert_{\alpha-l\gamma}\lesssim \sup_{\tau\in [s,t]}\vert\xi^{l-1}_{\tau}\vert_{\alpha - (l-1)\gamma - \sigma}(b-a)^{lp}W_{\Pi^{l}(\mathbf{X}),\gamma,p}^{l(\gamma-p)}(a,b).
				\end{align*} 
				As in the previous cases, we obtain 
				\begin{align}\label{D11}
					\left(\sum_{k}\frac{\left\vert\left(U_{\tau_{k+1},\tau_k}\xi^{l-1}_{\tau_k}\circ(\Pi^{l}(\mathbf{X}))_{\tau_k,\tau_{k+1}}\right)\right\vert_{\alpha-l\gamma}^{\frac{1}{l(\gamma-p)}}}{(\tau_{k+1}-\tau_k)^{\frac{p}{\gamma-p}}}\right)^{l(\gamma-p)}\lesssim \sup_{\tau\in [s,t]}\vert\xi^{l-1}_{\tau}\vert_{\alpha - (l-1)\gamma - \sigma}W_{\Pi^{l}(\mathbf{X}),\gamma,p}^{l(\gamma-p)}(s,t).
				\end{align}
				\item Since $\alpha - (l-1)\gamma - \sigma\leq \alpha -(l-1)\gamma$, we consequently derive from \eqref{regularity}
				\begin{align*}
					\left\vert\left(U_{b,a}\xi^{l-1}_{a}\circ(\Pi^{l}(\mathbf{X}))_{a,b}\right)\right\vert_{\alpha-(l-1)\gamma}\leq \sup_{\tau\in [s,t]}\vert\xi^{l-1}_{\tau}\vert_{\alpha - (l-1)\gamma - \sigma}(b-a)^{lp-\sigma}W_{\Pi^{l}(\mathbf{X}),\gamma,p}^{l(\gamma-p)}(a,b).
				\end{align*}
				Then, we argue as in \eqref{C11} to conclude that
				\begin{align}\label{E11}
					\begin{split}
						&\left(\sum_{k}\frac{\left\vert\left(U_{\tau_{k+1},\tau_k}\xi^{l-1}_{\tau_k}\circ(\Pi^{l}(\mathbf{X}))_{\tau_k,\tau_{k+1}}\right)\right\vert_{\alpha-(l-1)\gamma}^{\frac{1}{(l-1)(\gamma-p)}}}{(\tau_{k+1}-\tau_k)^{\frac{p}{\gamma-p}}}\right)^{(l-1)(\gamma-p)}\\&\lesssim (t-s)^{p-\sigma}\sup_{\tau\in [s,t]}\vert\xi^{l-1}_{\tau}\vert_{\alpha - (l-1)\gamma - \sigma}W_{\Pi^{l}(\mathbf{X}),\gamma,p}^{l(\gamma-p)}(s,t)\lesssim \sup_{\tau\in [s,t]}\vert\xi^{l-1}_{\tau}\vert_{\alpha - (l-1)\gamma - \sigma}W_{\Pi^{l}(\mathbf{X}),\gamma,p}^{l(\gamma-p)}(s,t).
					\end{split}
				\end{align}
			\end{itemize}
			\textbf{C)}  To estimate \(D_l\), we use Theorem \ref{shdccsa}. From \eqref{sodais}
			\begin{align}\label{SAASas}
            \begin{split}    
				&\vert D_{l}(a,b)\vert_{\alpha-(l-r)\gamma}\lesssim (b-a)^{(l-r)p}\times\\&\sum_{0\leq j<N}W_{\Pi^{j+1}(\mathbf{X}),\gamma,p}^{(j+1)(\gamma-p)}(a,b)\left( (b-a)^{P_{N}+(l-r)(\gamma-p)}W_{R^{j,N},\alpha-\sigma,\gamma,p,2}^{(N-j)(\gamma-p)}(a,b)+(b-a)^{P_j+(l-r)(\gamma-p)}\sup_{\tau\in [s,t]}\left\vert\xi^{j}_{\tau}\right\vert_{\alpha-j\gamma-\sigma}\right)\\& \lesssim (b-a)^{(l-r)p}\times\sum_{0\leq j<N}W_{\Pi^{j+1}(\mathbf{X}),\gamma,p}^{(j+1)(\gamma-p)}(a,b)\left( W_{R^{j,N},\alpha-\sigma,\gamma,p,2}^{(N-j)(\gamma-p)}(a,b)+(b-a)^{P_j+(l-r)(\gamma-p)}\sup_{\tau\in [s,t]}\left\vert\xi^{j}_{\tau}\right\vert_{\alpha-j\gamma-\sigma}\right).
                            \end{split}
			\end{align}
Thus, from Lemma \ref{INM} and \eqref{SAASas}, we obtain			\begin{align*}
				&\frac{\vert D_{l}(a,b)\vert_{\alpha-(l-r)\gamma}^{\frac{1}{(l-r)(\gamma-p)}}}{(b-a)^{\frac{p}{\gamma-p}}}\lesssim\sum_{0\leq j<N}\left(\underbrace{W_{\Pi^{j+1}(\mathbf{X}),\gamma,p}^{\frac{j+1}{N+1}}(a,b) W_{R^{j,N},\alpha-\sigma,\gamma,p,2}^{\frac{N-j}{N+1}}(a,b)}_{W_{1}(a,b)}\right)^{\frac{N+1}{l-r}}\\&+\sum_{0\leq j<N}\left(\underbrace{(b-a)^{\frac{P_j+(l-r)(\gamma-p)}{P_j+(l+j+1-r)(\gamma-p)}}W_{\Pi^{j+1}(\mathbf{X}),\gamma,p}^{\frac{(j+1)(\gamma-p)}{P_j+(l-r+j+1)(\gamma-p)}}(a,b)}_{W_{2}(a,b)}\right)^{\frac{P_j+(l-r+j+1)(\gamma-p)}{(l-r)(\gamma-p)}}\sup_{\tau\in [s,t]}\left\vert\xi^{j}_{\tau}\right\vert_{\alpha-j\gamma-\sigma}^{\frac{1}{(l-r)(\gamma-p)}}.
			\end{align*}
            %\todo{recall who \( W_1 \) is}
			{Since \( W_1 \) and \( W_2 \) (which are defined above) are clearly controlled functions, we can apply Lemma \ref{INM}, in particular~\eqref{85sd} and \eqref{785a}, and arguing similarly to \eqref{C11}, to conclude}
			\begin{align}\label{F11}
				\begin{split}
					&\left(\sum_k\frac{\vert D_{l}(\tau_k,\tau_{k+1})\vert_{\alpha-(l-r)\gamma}^{\frac{1}{(l-r)(\gamma-p)}}}{(\tau_{k+1}-\tau_{k})^{\frac{p}{\gamma-p}}}\right)^{(l-r)(\gamma-p)}\lesssim \sum_{0\leq j<N}W_{\Pi^{j+1}(\mathbf{X}),\gamma,p}^{(j+1)(\gamma-p)}(s,t) W_{R^{j,N},\alpha-\sigma,\gamma,p,2}^{(N-j)(\gamma-p)}(s,t)\\&\quad+\sum_{0\leq j<N}(t-s)^{P_j+(l-r)(\gamma-p)}W_{\Pi^{j+1}(\mathbf{X}),\gamma,p}^{(j+1)(\gamma-p)}(s,t)\sup_{\tau\in [s,t]}\left\vert\xi^{j}_{\tau}\right\vert_{\alpha-j\gamma-\sigma}
					\\&\myquad[2]\lesssim \sum_{0\leq j<N}\left( W_{\Pi^{j+1}(\mathbf{X}),\gamma,p}^{(j+1)(\gamma-p)}(s,t) W_{R^{j,N},\alpha-\sigma,\gamma,p,2}^{(N-j)(\gamma-p)}(s,t)+W_{\Pi^{j+1}(\mathbf{X}),\gamma,p}^{(j+1)(\gamma-p)}(s,t)\sup_{\tau\in [s,t]}\left\vert\xi^{j}_{\tau}\right\vert_{\alpha-j\gamma-\sigma}\right).
				\end{split}
			\end{align}
			\textbf{D)} It remains to estimate $C_{l}$. Note that from Corollary \ref{saswf65} and the regularizing properties of parabolic evolution families~\eqref{regularity} we further infer that
			\begin{align*}
				&\left\vert(U_{b,a}-I)\int_{s}^{a}U_{a,v}\xi_{v}\circ\mathrm{d}\mathbf{X}_v\right\vert_{\alpha-(l-r)\gamma}\leq (b-a)^{(l-r)\gamma}\left\vert\int_{s}^{a}U_{a,v}\xi_{v}\circ\mathrm{d}\mathbf{X}_v\right\vert_\alpha\\
				&\lesssim (b-a)^{(l-r)\gamma}\sum_{0\leq j<N}W_{\Pi^{j+1}(\mathbf{X}),\gamma,p}^{(j+1)(\gamma-p)}(a,b)\left( (b-a)^{P_{N}}W_{R^{j,N},\alpha-\sigma,\gamma,p,2}^{(N-j)(\gamma-p)}(a,b)+(b-a)^{P_{j}}\sup_{\tau\in [a,b]}\left\vert\xi^{j}_{\tau}\right\vert_{\alpha-j\gamma-\sigma}\right).
			\end{align*}
			Therefore, we argue as in the previous case to conclude that			\begin{align}\label{G11}
				\begin{split}
					&\left(\sum_k\frac{\left\vert C_{l}(\tau_k,\tau_{k+1})\right\vert_{\alpha-(l-r)\gamma}^{\frac{1}{(l-r)(\gamma-p)}}}{(\tau_{k+1}-\tau_{k})^{\frac{p}{\gamma-p}}}\right)^{(l-r)(\gamma-p)}\\&\quad\lesssim \sum_{0\leq j<N}W_{\Pi^{j+1}(\mathbf{X}),\gamma,p}^{(j+1)(\gamma-p)}(s,t)\left( (t-s)^{P_{N}}W_{R^{j,N},\alpha-\sigma,\gamma,p,2}^{(N-j)(\gamma-p)}(s,t)+(t-s)^{P_{j}}\sup_{\tau\in [s,t]}\left\vert\xi^{j}_{\tau}\right\vert_{\alpha-j\gamma-\sigma}\right)\\&\myquad[2]\lesssim \sum_{0\leq j<N}\left( W_{\Pi^{j+1}(\mathbf{X}),\gamma,p}^{(j+1)(\gamma-p)}(s,t) W_{R^{j,N},\alpha-\sigma,\gamma,p,2}^{(N-j)(\gamma-p)}(s,t)+W_{\Pi^{j+1}(\mathbf{X}),\gamma,p}^{(j+1)(\gamma-p)}(s,t)\sup_{\tau\in [s,t]}\left\vert\xi^{j}_{\tau}\right\vert_{\alpha-j\gamma-\sigma}\right). 
				\end{split}
			\end{align}
		
		Recall that \(\pi = \{s = \tau_0 < \tau_1 < \cdots < \tau_m = t\}\) is an arbitrary partition of \([s, t]\). Therefore, from \eqref{A11}–\eqref{G11}, we derive the bound claimed in \eqref{52}.
	\end{proof}
	Now we can state the next result which states that the rough integral is a controlled rough path together with suitble estimates. This is similar to~\cite[Corollary 4.6]{GHT21} which work with H\"older norms and paths of regularity $\gamma\in(1/3,1/2)$. However, as already mentioned, we have to replace the H\"older norms by the controls previously defined in order to ensure the integrability of the bounds obtained. 
	\begin{theorem}\label{STTAS}
		Assume \( \mathbf{X} \in \mathscr{C}^{\gamma,p}([0,T]) \) and \(
		\boldsymbol{\xi} \in \mathscr{D}^{\gamma, p}_{\mathbf{X},\alpha-\sigma}\left([0,T]\right)
		\). Additionally, assume that all the conditions of Theorem \ref{shdccsa} hold. Fix an arbitrary interval \( [s,t] \subseteq [0,T] \) such that \( t-s \leq 1 \). For \( a \in [s,t] \) and \( 0 \leq j \leq N \), define:  
		\begin{align*}
			\tilde{\xi}^{0}_{a} & := \int_{s}^{a} U_{a,v} \xi_v \circ \mathrm{d}\mathbf{X}_v, \quad j = 0, \\
			\tilde{\xi}^{j}_{a} & := \xi^{j-1}_{a}, \quad 1 \leq j \leq N.
		\end{align*}
		Then 
		\begin{align*}
			\tilde{\boldsymbol{\xi}} = (\tilde{\xi}^{i})_{0 \leq i < N}\in \tilde{\mathscr{D}}^{\gamma, p}_{\mathbf{X},\alpha}\left([0,T]\right).
		\end{align*}
		In addition, we can find \( 0 < \beta < 1 \), which explicitly depends on \( p, \gamma \), and \( \sigma \), such that
		\begin{align*}
			&\Vert\tilde{\boldsymbol{\xi}}\Vert_{\tilde{\mathscr{D}}^{\gamma, p}_{\mathbf{X},\alpha}\left([s,t]\right)}\lesssim \sup_{0\leq j\leq N}\vert\xi^{j}_{s}\vert_{L\left((\mathbb{R}^d)^{\otimes j+1}, E_{\alpha-j\gamma-\sigma}\right)} \\&+\left((t-s)^{\beta}+W_{\Pi^{1}(\mathbf{X}),\gamma,p}^{(\gamma-\sigma)}(s,t)+\sum_{0<j\leq N }W_{\Pi^{j}(\mathbf{X}),\gamma,p}^{j(\gamma-p)}(s,t)\right)\Vert\boldsymbol{\xi}\Vert_{\mathscr{D}^{\gamma, p}_{\mathbf{X},\alpha-\sigma}\left([s,t]\right)}.
		\end{align*}
	\end{theorem}
	\begin{proof}
		The proof is based on Propositions \ref{ATSY} and \ref{UJMMA}. Additionally, recall the definitions in \eqref{YHN123} and \eqref{YHN457}. Indeed, in Proposition \ref{ATSY}, we estimated the higher elements of the norm \( \tilde{\boldsymbol{\xi}} \), i.e., for \( i \geq 1 \), in terms of the norm of the original path \( \boldsymbol{\xi} \). The element corresponding to the lowest level, i.e., the integrand path, is estimated in Proposition \ref{UJMMA}. Combining these results proves the statement.
	\end{proof}
	\begin{remark}
		The dependency of \( \beta \) can be explicitly obtained, but we do not mention it here since it is not important for our purposes. We also assume that \( t-s \leq 1 \). However, this is not a restriction. Indeed, when \( t-s > 1 \), the dependency of \( \beta \) can be adjusted accordingly.
        The reason we assume \( t - s \leq 1 \) is that it is natural to estimate the solution on small intervals and then combine these estimates to derive a bound over the entire interval.
	\end{remark}
	
	\section{Solution theory}\label{sec:sol}
	After defining a suitable integral in our setting, we point out that the existence of a solution of~\eqref{Main_Equation} can be obtained by a standard fixed-point argument. We omit the proof, since it is similar to the case \( \frac{1}{3}<\gamma \leq \frac{1}{2} \) treated in~\cite{GH19, GHT21, HN22} for a diffusion coefficient $G\in C^3_b$. Since $G$ is linear in our case, we only provide the necessary estimates required in order to perform this fixed-point argument. 
	%Our aim is to derive an a-priori bound in terms of our new definition for the controlled rough path, which is integrable in comparison to the references mentioned above.
	We emphasize that it is possible to incorporate the case where $G$ is nonlinear but bounded, with bounded derivatives up to $N+1$. We only focus here on the linear case and state the following assumptions on the drift and diffusion coefficients, mentioning that the conditions on the drift term are similar to the ones imposed in~\cite{BGS25}. 
	
	\begin{assumption}\label{ASSDS}
		We assume
		\begin{enumerate}
			\item There exists $\delta\in[0,1)$ such that $F:[0,T]\times E_{\alpha}\to E_{\alpha-\delta}$ is Lipschitz continuous in $E_\alpha$, uniformly in $[0,T]$. That means for every $t\in [0,T]$, $F(t,\cdot)=:F_t(\cdot)$ is Lipschitz continuous with constant $L_{t,F}$ such that $L_F:=\sup_{t\in [0,t]}L_{t,F}<\infty$. In particular, we have the following Lipschitz condition:  
           % \todo{I think we should also revise the condition for \( F \) in our paper with Tim. Currently, it is stated as  
%$|F_t(z) - F_s(y)|_{E_{\alpha-\delta}} \leq L_F |z - y|_{E_\alpha}$
%for all \( s, t \). However, this seems somewhat unusual, overly strong, and even unnecessary. The Lipschitz condition is only required for each fixed time. Am I correct? \textcolor{blue}{the Lipschitz condition in the paper with tim is uniform in time and we can assume it. Now i see, s is a typo in the paper with Tim. i'll change it} }
\begin{align*}
    |F_t(z) - F_t(y)|_{E_{\alpha-\delta}} &\leq L_F |z - y|_{E_\alpha}, 
\end{align*}
for all \( y, z \in E_\alpha \). In particular, this implies that  
\[
 |F_t(y)|_{E_{\alpha-\delta}} \leq C_F (1 + |y|_{E_{\alpha}}),
\]
where  
\[
C_F := \max \left\{ L_F, \sup_{t \in [0,T]} |F_t(0)|_{E_{\alpha-\delta}} \right\} < \infty.
\]
			\item  For $0 \leq j< N$ we assume that  $G:=[0,T]\times E_{\alpha-j\gamma}\to L\left(\mathbb{R}^d, E_{\alpha-j\gamma-\sigma}\right)$. For every $t\in [0,T]$, $G(t,\cdot)=:G_t(\cdot)$ is a continuous linear map such that %\todo{ we need in~\cite[Lemma 3.2]{Tim} $2\gamma$-H\"older continuity in time to make sure that the time dependency is not a problem.\textcolor{red}{Thanks, I reformulated}}
			\begin{align*}
				C_{G}:=\max_{0\leq j< N}\sup_{t\in[0,T]}\Vert G_{t}(.)\Vert_{L\left(E_{\alpha-j\gamma},L\left(\mathbb{R}^d, E_{\alpha-j\gamma-\sigma}\right)\right)}<\infty .
			\end{align*}
			For every \( z \in E_{\alpha} \), we assume that the map
			\begin{align*}
				G(.,z) : E_{\alpha} \to E_{\alpha-\sigma}, \quad t \mapsto G(t, z)
			\end{align*}
			is H\"older continuous with parameter $N\gamma$. In particular,  there exists a constant $M_G>0$ such that for every $s,t\in [0,T]$ and $z\in E_{\alpha}$
			\begin{align*}
				 &\left| G_{t}(z)-G_{s}(z) \right|_{E_{\alpha-\sigma}} \leq M_G  (t-s)^{N\gamma}| z |_{E_{\alpha}},\\
                 &\left| G_{t}(z)\right|_{E_{\alpha-\sigma}} \leq M_G | z |_{E_{\alpha}}.
			\end{align*}
			%where \( M_G \) is a constant.
		\end{enumerate}
	\end{assumption}
	\begin{definition}
		Under Assumption \ref{ASSDS} for every \( 1 \leq k \leq N \), \( t \in [0,T] \) and \( 0 \leq j < N+1-k \), we define
		\begin{align*}
			G^{\circ k}_t : E_{\alpha-j\gamma} \to L\left((\mathbb{R}^d)^{\otimes k}, E_{\alpha-(k+j-1)\gamma-\sigma}\right),
		\end{align*}
		where \( G^{\circ 1} = G \).  These maps are defined recursively for $z\in E_{\alpha-j\gamma}$ and $\bigotimes_{i=1}^{N+1} v_i\in (\mathbb{R}^d)^{\otimes N+1}$  as follows		\begin{align*}
			G^{\circ k}_t(z)\left(\bigotimes_{i=1}^k v_i\right) 
			= G^{\circ (k-1)}_t\left(G^{\circ 1}_t(z)(v_1)\right)\left(\bigotimes_{i=2}^k v_i\right)\in E_{\alpha-j\gamma-\sigma}.
		\end{align*}
		Since \( G_t: E_{\alpha - j\gamma} \to L(\mathbb{R}^d, E_{\alpha - j\gamma - \sigma}) \), we use the following general algebraic convention without specifying \( k \):
		\begin{align}\label{CIRC}
			\begin{split}
				&G_t^{\circ 1} : L\left((\mathbb{R}^d)^{\otimes k}, E_{\alpha - j\gamma}\right) 
				\to L\left((\mathbb{R}^d)^{\otimes (k+1)}, E_{\alpha - j\gamma - \sigma}\right), \\
				&G_t^{\circ 1}\left(J\right)\left(\bigotimes_{i=1}^{k+1} v_i\right) 
				:= G_t\left(J\left(\bigotimes_{i=1}^k v_i\right)\right)(v_{k+1}) \in E_{\alpha - j\gamma - \sigma},
			\end{split}
		\end{align}
		for every  $J\in L\left((\mathbb{R}^d)^{\otimes k}, E_{\alpha - j\gamma}\right) $.
		This enables us to define for \( \tilde{\boldsymbol{\xi}} \in \tilde{\mathscr{D}}^{\gamma, p}_{\mathbf{X},\alpha}\left([0,T]\right) \)
		\begin{align*}
			G(\tilde{\boldsymbol{\xi}})(t) := \big(G^{\circ 1}_t(\tilde{\xi}^j_t)\big)_{0 \leq j < N},
		\end{align*}
		where \( G^{\circ 1}_t(\tilde{\xi}^j_t) \in L\left((\mathbb{R}^d)^{\otimes (j+1)}, E_{\alpha - j\gamma - \sigma}\right) \) is given by \eqref{CIRC}, i.e.,
		\begin{align}\label{asfdge}
			G^{\circ 1}_t(\tilde{\xi}^j_t)\left(\bigotimes_{i=1}^{j+1} v_i\right) 
			:= G_t\big(\tilde{\xi}^j_t\big(\bigotimes_{i=1}^j v_i\big)\big)(v_{j+1}) \in E_{\alpha - j\gamma - \sigma}.
		\end{align}
	\end{definition}
	We can now prove the following  lemma, which provides an estimate for the composition of $G$ with the controlled rough path $\tilde{\boldsymbol{\xi}}$. As already stated we remove for notational simplicity the Gubinelli derivatives from the following expressions. %\todo{ add a remark that based on the next bound one can do the fixed point argument for the solution. also mention that it's important to have the result since $G$ is unbounded but linear\textcolor{red}{ see remark \ref{SAJd2}}}
	%\eqref{Main_Equation}.
	\begin{lemma}\label{compoo}
		Under the Assumption \ref{ASSDS}, we let $\tilde{\boldsymbol{\xi}} \in \tilde{\mathscr{D}}^{\gamma, p}_{\mathbf{X},\alpha}\left([0,T]\right)$. Then 
		\begin{align}\label{HAGt1}
			G(\tilde{\boldsymbol{\xi}})\in \mathscr{D}^{\gamma, p}_{\mathbf{X},\alpha-\sigma}\left([0,T]\right).
		\end{align}
		{Moreover, for a constant \( C_G \) which only depends on \( G \), and for every \( [s,t] \subseteq [0,T] \), it holds that}%\todo{ specify on which the constants below in $\lesssim$ depend}
		\begin{align*}
			\Vert G(\tilde{\boldsymbol{\xi}})\Vert_{\mathscr{D}^{\gamma, p}_{\mathbf{X},\alpha-\sigma}\left([s,t]\right)}\leq  C_G\Vert \tilde{\boldsymbol{\xi}}\Vert_{\tilde{\mathscr{D}}^{\gamma, p}_{\mathbf{X},\alpha}\left([s,t]\right)}.
		\end{align*}
	\end{lemma}
	\begin{proof}
		Recalling Definition \ref{controlled2} and the fact that $G$ is linear, we have for \( [u, v] \subseteq [0, T] \) and \( 0 \leq i, l \leq N \) with \( l - i > 1 \) that
		\begin{align}\label{REEEAA}
			\begin{split}
				G_{v}^{\circ1}(\tilde{\xi}^i_v)-G_{u}^{\circ1}(\tilde{\xi}^i_u)&=G_{u}^{\circ 1}(\delta \tilde{\xi}^i_{u,v})+\left(G_{v}^{\circ1}(\tilde{\xi}^i_v)-G_{u}^{\circ1}(\tilde{\xi}^i_v)\right)\\&=\sum_{i<j<l}G_{u}^{\circ 1}\left(\tilde{\xi}^{j}_u\circ (\Pi^{j-i}(\mathbf{X}))_{u,v}\right)+G_{u}^{\circ1}\left(\tilde{R}^{i,l}_{u,v}\right)+\left(G_{v}^{\circ1}(\tilde{\xi}^i_v)-G_{u}^{\circ1}(\tilde{\xi}^i_v)\right).
			\end{split}
		\end{align}
		We further	set
		\begin{align*}
			&R^{G,i,l}_{u,v}=G_{u}^{\circ1}\left(\tilde{R}^{i,l}_{u,v}\right)+\left(G_{v}^{\circ1}(\tilde{\xi}^i_v)-G_{u}^{\circ1}(\tilde{\xi}^i_v)\right),\\
			&\xi^{j}_{u}:=G_{u}^{\circ 1}\left(\tilde{\xi}^{j}_u\right).
		\end{align*}
		Now, it is enough to verify the condition of Definition \ref{controlled} entailing that $G(\boldsymbol{\tilde{\xi}})$ is a controlled rough path. Since the steps are rather straightforward due to the linearity of $G$, we omit the details.
	\end{proof}% \todo{reformulate the second sentence of the next lemma. Let $\boldsymbol{\xi}$ be the rough integral of $G(\tilde{\xi})$ against $\mathbf{X}$ constructed in Theorem 2.17.}
        From Lemma \ref{HAGt1}, we can find a bound for the rough integral of $G(\boldsymbol{\tilde{\xi}})$ with respect to $\mathbf{X}$. 
	\begin{lemma}\label{YBAS54}
		Let Assumption \ref{ASSDS} and all the conditions of Theorem \ref{shdccsa} hold, and let \( [s,t] \subseteq [0,T] \) such $t-s\leq 1$. Assume that \( \tilde{\boldsymbol{\xi}} \in \tilde{\mathscr{D}}^{\gamma, p}_{\mathbf{X},\alpha}([0,T]) \) and   \( G(\tilde{\boldsymbol{\xi}}) \in \mathscr{D}^{\gamma, p}_{\mathbf{X},\alpha-\sigma}([0,T]) \) as in Lemma \ref{compoo}. Let \( \boldsymbol{\xi} = (\xi^{j})_{0 \leq j \leq N} \) {be the rough integral of $G(\tilde{\xi})$ against $\mathbf{X}$ constructed in Theorem \ref{STTAS}}. In particular
		\begin{align*}
			{\xi}^{0}_{a} & := \int_{s}^{a} U_{a,v} G(v,\tilde{\xi}_{v}^{0}) \circ \mathrm{d}\mathbf{X}_v, \quad j = 0, \\
			{\xi}^{j}_{a} & := G^{\circ 1}_a(\tilde{\xi}^{j-1}_a), \quad 1 \leq j \leq N.
		\end{align*}
		{Then, there exists \( 0 < \beta < 1 \)  as in Theorem \ref{STTAS}} such that for all $0\leq j <N$ we have %\todo{ there exists $\beta\in(0,1)$ as in the previous theorem and $1\leq j \leq N$}
		\begin{align*}
			&\Vert{\boldsymbol{\xi}}\Vert_{\tilde{\mathscr{D}}^{\gamma, p}_{\mathbf{X},\alpha}\left([s,t]\right)}\lesssim \sup_{0\leq j\leq N}\vert\tilde{\xi}^{j}_{s}\vert_{L\left((\mathbb{R}^d)^{\otimes j}, E_{\alpha-j\gamma}\right)} \\&+\left((t-s)^{\beta}+W_{\Pi^{1}(\mathbf{X}),\gamma,p}^{(\gamma-\sigma)}(s,t)+\sum_{0<j\leq N }W_{\Pi^{j}(\mathbf{X}),\gamma,p}^{j(\gamma-p)}(s,t)\right)\Vert\tilde{\boldsymbol{\xi}}\Vert_{\tilde{\mathscr{D}}^{\gamma, p}_{\mathbf{X},\alpha}\left([s,t]\right)}.
		\end{align*}
		\begin{proof}
			The proof follows directly from Theorem \ref{STTAS} and Lemma \ref{compoo}.
		\end{proof}
	\end{lemma}
    %\todo{mention for completeness why we need  weakly geometric for nonlinear G.\textcolor{red}{See the remark}}
    
    \begin{remark}\label{SAJd2}
    \begin{itemize}
        \item [1)] Based on the bounds obtained in the previous Lemmas, one can perform a fixed-point argument to prove well-posedness of~\eqref{Main_Equation} for a linear diffusion coefficient $G$.
        \item [2)] Provided that $\mathbf{X}$ is weakly geometric, one can prove analogue results to Lemma \ref{compoo} and Lemma \ref{YBAS54} for a nonlinear term \( G \) which is $N$ times Fr\'echet differentiable. Moreover, if \( G \) is assumed to be $N+1$ times Fréchet differentiable, then one can prove the local well-posedness of~\eqref{Main_Equation} by a fixed-point argument similar to~\cite[Theorem 5.1]{GHT21}. Under additional boundedness assumptions on the $N+1$ Fr\'echet derivatives, one can infer that the solution exists globally by similar techniques to \cite{HN22}.
        \item [3)] {For nonlinear \( G \), even in the finite-dimensional case, the assumption that \( \mathbf{X} \) is weakly geometric is important in guaranteeing the existence and uniqueness of the solution, see~\cite[Theorem 10.14 and Theorem 10.26]{FV10}. The main reason is given by the definition of the composition of the rough path with nonlinear function $G$. In this case, the Taylor expansion of \( G \) allows us to canonically define the new Gubinelli derivatives in terms of the original Gubinelli derivatives of the path. In this context, the weak geometricity of \( \mathbf{X} \) is important for ensuring certain symmetries in the expansion.
}
    \end{itemize}
 \end{remark}

	We need the following definition, which we refer to as greedy points. This definition was first introduced in the seminal work \cite{CLL13} to obtain integrable a priori bounds for rough differential equations.% in the finite-dimensional case.
	\begin{definition}\label{GRRED}
		Let \( I = [a, b] \) and \( \chi >0\) be a given treshold. The {sequence of greedy points}, denoted by \( \lbrace \tau^{I}_{m, \mathbf{X}}(\chi) \rbrace_{m \geq 0} \), is constructed as follows: we set \( \tau^{I}_{0,\mathbf{X}}(\chi) = a \), and define subsequent points recursively by 
		\begin{align}\label{s}
			\tau^{I}_{m+1,\mathbf{X}}(\chi) \coloneqq \sup \big\lbrace \tau : \tau^{I}_{m,\mathbf{X}}(\chi) \leq \tau \leq b, \ \ 
			W_{\mathbf{X}, \gamma, p}^{\gamma-p}(\tau^{I}_{m,\mathbf{X}}(\chi), \tau) \leq \chi \big\rbrace.
		\end{align}
		 Finally, we define the quantity %\todo{ even it's clear i would prefer another notation for the number of greedy points. N is the number of iterated integrals, \textcolor{red}{Now is $\tilde{N}$}}
		\begin{align}\label{DSAQW}
			\tilde{N}(I, \chi, \mathbf{X}) \coloneqq \inf \lbrace m > 0 : \tau_{m, \mathbf{X}}^{I}(\chi) = b \rbrace,
		\end{align}
		which represents the smallest number \( m \) such that the sequence reaches the endpoint \( b \). Note that, since \( W_{\mathbf{X}, \gamma, p} \) is assumed to be continuous, for every \( m < \tilde{N}(I, \chi, \mathbf{X}) \), the definition yields\newline \( W_{\mathbf{X}, \gamma, p}^{\gamma-p}(\tau^{I}_{m-1,\mathbf{X}}(\chi), \tau^{I}_{m,\mathbf{X}}(\chi)) = \chi \). Moreover, for \( m = \tilde{N}(I, \chi, \mathbf{X}) \), we have  
		\[
		W_{\mathbf{X}, \gamma, p}^{\gamma-p}(\tau^{I}_{m-1,\mathbf{X}}(\chi), \tau^{I}_{m,\mathbf{X}}(\chi)) \leq \chi.
		\]
	\end{definition}
	Now, we turn the solution of equation \eqref{Main_Equation}. First, we begin with a definition.
	\begin{definition}\label{SOLLK}
		Let Assumption \ref{ASSDS} holds. %\todo{ Let Assumption hold or be satisfied instead of assume Assumption}\ref{ASSDS} holds and \([s,t] \subseteq [0,T]\).
        We say that \(\tilde{\boldsymbol{\xi}} \in \tilde{\mathscr{D}}^{\gamma, p}_{\mathbf{X},\alpha}\left([s,t]\right)\) is the solution for equation \eqref{Main_Equation} in \([s,t]\) with initial value \(y \in E_\alpha\) if, for every \(0 \leq i < N\), \(\tau \in [s,t]\) and \(\tilde{\xi}^{0}_{s} = y\), we have
		\begin{align*}
			\tilde{\xi}^{i}_{\tau} = G^{\circ i}_\tau(\tilde{\xi}^{0}_{\tau}).
		\end{align*}
		Moreover, the path component of the mild solution of~\eqref{Main_Equation} is given by 
		\begin{align}\label{sddwidiw}
			\tilde{\xi}^{0}_{\tau}=U_{\tau,s}y+\int_{s}^{\tau}U_{\tau,u}F(u,\tilde{\xi}_{u}^0)\mathrm{d}u+\int_{s}^{\tau}U_{\tau,u}G(u,\tilde{\xi}^{0}_{u})\circ\mathrm{d}\mathbf{X}_u,
		\end{align}
		where the rough integral of \( G(\tilde{\boldsymbol{\xi}}) \in \mathscr{D}^{\gamma, p}_{\mathbf{X},\alpha-\sigma}\left([s,t]\right) \) against $\mathbf{X}$ is defined in the sense of Theorem \ref{shdccsa}. 
	\end{definition}
	\begin{remark}
		The second term of the left-hand side of \eqref{sddwidiw}
		can be interpreted as a controlled path with zero Gubinelli derivatives. We will use this fact in the next sequel without explicitly mentioning it. Since for an arbitrary time interval $[\tau_1,\tau_2]\subseteq [s,t]$
		\begin{align}\label{A3.5}
			Z_{\tau_2}-Z_{\tau_1}=\underbrace{\int_{\tau_1}^{\tau_2}U_{\tau_2,u}F(u,\tilde{\xi}_{u}^0)\mathrm{d}u}_{\text{I}(\tau_1,\tau_2)}+\underbrace{(U_{\tau_2,\tau_1}-I)\int_{s}^{\tau_1}U_{\tau_1,u}F(u,\tilde{\xi}_{u}^0)\mathrm{d}u}_{\text{II}(\tau_1,\tau_2)},
		\end{align} 
		we further derive based on Assumption \ref{ASSDS}% \todo{ i would drop the bullet points below,\textcolor{red}{i removed then}}
			\begin{align}\label{A4}
				\begin{split}
					\sup_{\tau\in[s,t]}\left\vert\int_{s}^{\tau}U_{\tau,u}F(u,\tilde{\xi}_{u}^0)\mathrm{d}u\right\vert_{E_\alpha}&\leq \sup_{\tau\in [s,t]}\int_{s}^{\tau}\left\vert U_{\tau,u}F(u,\tilde{\xi}_{u}^0)\right\vert_{E_\alpha}\mathrm{d}\tau\\&\lesssim (t-s)^{1-\delta} \left(1+\sup_{\tau\in [s,t]}\vert \tilde{\xi}_{\tau}^0\vert_{E_{\alpha}}\right).
				\end{split}
			\end{align}
			For \( 1 \leq l \leq N \), using  \eqref{regularity} we have
			\[
			\vert U_{\tau_2,u}x \vert_{E_{\alpha-l\gamma}} \lesssim (\tau_2 - u)^{-\max\lbrace 0, \delta - l\gamma \rbrace}\vert x\vert_{E_{\alpha-\delta}}.
			\]
Using again Assumption~\ref{ASSDS} on \( F \), we can deduce			\begin{align}\label{AUAKs}
				\begin{split}
					&\vert\text{I}(\tau_1,\tau_2)\vert_{E_{\alpha-l\gamma}}\leq \int_{\tau_1}^{\tau_2}\left\vert U_{\tau_2,u}F(u,\tilde{\xi}_{u}^0)\right\vert_{E_{\alpha-l\gamma}}\mathrm{d}u\lesssim \left(1+\sup_{\tau\in [s,t]}\vert \tilde{\xi}_{\tau}^0\vert_{E_{\alpha}}\right)\int_{\tau_1}^{\tau_2}(\tau_2-u)^{-\max\lbrace 0,\delta-l\gamma\rbrace}\mathrm{d}u\\&\lesssim (\tau_2-\tau_1)^{1-\max\lbrace 0,\delta-l\gamma\rbrace}\left(1+\sup_{\tau\in [s,t]}\vert \tilde{\xi}_{\tau}^0\vert_{E_{\alpha}}\right)=(\tau_2-\tau_1)^{l\gamma+\min\lbrace 1-l\gamma,1-\delta \rbrace}\left(1+\sup_{\tau\in [s,t]}\vert \tilde{\xi}_{\tau}^0\vert_{E_{\alpha}}\right).
				\end{split}
			\end{align}
			This yields that
			\begin{align}\label{A5}
				\frac{\vert\text{I}(\tau_1,\tau_2)\vert_{E_{\alpha-l\gamma}}^{\frac{1}{l(\gamma-p)}}}{(\tau_2-\tau_1)^{\frac{p}{\gamma-p}}}\lesssim (\tau_2-\tau_1)(t-s)^{\frac{\min\lbrace 1-l\gamma,1-\delta \rbrace}{{l(\gamma-p)}}}\left(1+\sup_{\tau\in [s,t]}\vert \tilde{\xi}_{\tau}^0\vert_{E_{\alpha}}\right)^{\frac{1}{l(\gamma-p)}}.
			\end{align}
			Again for \( 1 \leq l \leq N \), it follows from \eqref{regularity} and \eqref{AUAKs} that
			\begin{align*}
				\begin{split}
					\vert\text{II}(\tau_1,\tau_2)\vert_{E_{\alpha-l\gamma}}&\lesssim (\tau_2-\tau_1)^{l\gamma}\sup_{\tau\in[s,t]}\left\vert\int_{s}^{\tau}U_{\tau_1,u}F(u,\tilde{\xi}_{u}^0)\mathrm{d}u\right\vert_{E_{\alpha-l\gamma}}\\&\lesssim (\tau_2-\tau_1)^{l\gamma}(t-s)^{\min\lbrace 1,1-\delta+l\gamma \rbrace}\left(1+\sup_{\tau\in [s,t]}\vert \tilde{\xi}_{\tau}^0\vert_{E_{\alpha}}\right),
				\end{split}
			\end{align*}
			which gives 
			\begin{align}\label{A6}
				\frac{\vert\text{II}(\tau_1,\tau_2)\vert_{E_{\alpha-l\gamma}}^{\frac{1}{l(\gamma-p)}}}{(\tau_2-\tau_1)^{\frac{p}{\gamma-p}}}\lesssim (\tau_2-\tau_1)(t-s)^{\frac{\min\lbrace 1,1-\delta+l\gamma \rbrace}{{l(\gamma-p)}}}\left(1+\sup_{\tau\in [s,t]}\vert \tilde{\xi}_{\tau}^0\vert_{E_{\alpha}}\right)^{\frac{1}{l(\gamma-p)}}.
			\end{align}
		We can represent the path $Z$ as $\mathbf{Z}\in \tilde{\mathscr{D}}^{\gamma, p}_{\mathbf{X},\alpha}([s,t])$ in the sense of Definition \ref{controlled2}. Therefore, from \eqref{A3.5}, \eqref{A4}, \eqref{A5}, and \eqref{A6}, the following estimate holds for every interval $[s,t] \subseteq [0,T]$ with $t - s \leq 1$ and \( {\delta_1} := \min\left\{ 1 - \delta, 1 - N\gamma \right\} \)
		\begin{align}\label{GHy}
			\Vert\mathbf{Z}\Vert_{\tilde{\mathscr{D}}^{\gamma, p}_{\mathbf{X},\alpha}([s,t])}\lesssim (t-s)^{\delta_1}\left(1+\Vert \tilde{\boldsymbol{\xi}}\Vert_{\tilde{\mathscr{D}}^{\gamma, p}_{\mathbf{X},\alpha}\left([s,t]\right)}\right).
		\end{align}
		By a simpler argument for the first term in \eqref{sddwidiw}, namely the path  
		\begin{align*}
			\tau \mapsto Y_{\tau} := U_{\tau,s} y,
		\end{align*}  
		we obtain
		\begin{align}\label{GHsy}
			\Vert\mathbf{Y}\Vert_{\tilde{\mathscr{D}}^{\gamma, p}_{\mathbf{X},\alpha}([s,t])}\lesssim \vert y\vert_{E_\alpha}.
		\end{align}
		%In the sequel, for both paths $Y$ and $Z$, when we consider them as members of $\tilde{\mathscr{D}}^{\gamma, p}_{\mathbf{X},\alpha}([s,t])$, we denote them in the usual way, i.e., without using the bold symbol for distinction.
	\end{remark}
As mentioned in Remark \ref{SAJd2}, using a standard fixed-point argument similar to~\cite[Theorem 5.1]{GHT21} in the space of controlled rough paths \( \tilde{\boldsymbol{\xi}} \in \tilde{\mathscr{D}}^{\gamma, p}_{\mathbf{X},\alpha}\left([0,T]\right) \), one can show that the solution exists and is unique. We refrain from providing the details of this argument, since based on the previous deliberations, this follows as in~\cite[Theorem 5.1]{GHT21}.
	\begin{definition}\label{SOOLL}
		Assume that \(\tilde{\boldsymbol{\xi}} \in \tilde{\mathscr{D}}^{\gamma, p}_{\mathbf{X},\alpha}\left([s,t]\right)\) is the unique solution of \eqref{Main_Equation} in the sense of Definition \ref{SOLLK} with initial datum \(y\in E_\alpha\). For \(\tau \in [s,t]\), we define
		\begin{align*}
			\phi_{\mathbf{X}}(s,\tau,y) := \tilde{\xi}^{0}_{\tau}.
		\end{align*}
		Note that the uniqueness of solutions yields in particular the flow property. Namely  \( \phi_{\mathbf{X}}(s, s, y) = y \) and for \(\tau_1, \tau_2 \in [s,t]\) with \(\tau_1 \leq \tau_2\), we have
		\begin{align}\label{FLOWW}
			\phi_{\mathbf{X}}(s,\tau_2,y) = \phi_{\mathbf{X}}\left(\tau_1, \tau_2, \phi_{\mathbf{X}}(s,\tau_1,y)\right).
		\end{align}
		Using the mild formulation of~\eqref{Main_Equation} we have
		\begin{align}\label{AKaks}
			\begin{split}
				\phi_{\mathbf{X}}(s,\tau_2,y)&= U_{\tau_2,\tau_1}\left(\phi_{\mathbf{X}}(s,\tau_1,y)\right)+\int_{\tau_1}^{\tau_2}U_{\tau_2,u}F\left(u,\phi_{\mathbf{X}}(s,u,y)\right)\mathrm{d}u\\&+\int_{\tau_1}^{\tau_2}U_{\tau_2,u}G\left(u,\phi_{\mathbf{X}}(s,u,y)\right)\circ\mathrm{d}\mathbf{X}_u .
			\end{split}
		\end{align}
	By definition, the Gubinelli derivatives are obtained through the composition of the solution and \( G \). Therefore, we adopt the convention of using \( (\phi_{\mathbf{X}}(s,\tau,y))_{\tau \in [s,t]} \) to represent the controlled path, including all its derivatives. Accordingly, we write  
\begin{align*}
    \Vert\phi_{\mathbf{X}}(s, \cdot, y)\Vert_{\tilde{\mathscr{D}}^{\gamma, p}_{\mathbf{X},\alpha}\left([s,t]\right)}
\end{align*}  
as part of this convention. We use the same convention for the path component of the rough integral 
\begin{align*}
    \tau \mapsto \int_{a}^{\tau} U_{\tau,u} G\left(u, \phi_{\mathbf{X}}(s,u,y)\right) \circ \mathrm{d}\mathbf{X}_u.
\end{align*}  
\end{definition}  
Now, we can present the main result of this section.  
	%To enhance readability, we first state several intermdelate lemmas.
	\begin{theorem}\label{PRIORI}
		Let Assumption \ref{ASSDS} hold, and assume further the setting of Theorem \ref{shdccsa}. Additionally, \( (\phi_{\mathbf{X}}(0, \tau, y))_{0 \leq \tau \leq T} \) is the solution of the rough PDE \eqref{Main_Equation} on the interval $[0,\tau]$ in the sense of Definition \ref{SOOLL}. Then there exist \( \chi \) and \( L_2 \) with \( 0 < \chi \leq L_2 < 1 \), both independent of \( \mathbf{X} \) (but dependent on \( F \) and \( G \)) such that for all intervals \( [s, t] \subseteq [0,T]\)  with \( t - s \leq L_2 \), the following bound holds
		\begin{align}\label{PRRTTioi}
			\sup_{\tau\in [s,t]}\left\vert \phi_{\mathbf{X}}(0,\tau,y)\right\vert_{E_\alpha}\lesssim \left\vert \phi_{\mathbf{X}}(0,s,y)\right\vert_{E_\alpha}P_1(\mathbf{X},[s,t],\chi)+P_{2}(\mathbf{X},[s,t],\chi),
		\end{align}
		where \( P_1 \) and \( P_2 \) are greater than one and
		\begin{align}\label{PRRTTioiio}
			\max\bigg\lbrace P_1(\mathbf{X},[s,t],\chi),P_2(\mathbf{X},[s,t],\chi)\bigg\rbrace\lesssim \exp\left(\tilde{N}([s,t], \chi, \mathbf{X})\log(3L)\right).
		\end{align}
	\end{theorem}
	\begin{proof}
		The main idea is to estimate the solution over small intervals, which also depend on \( \mathbf{X} \), and then apply a concatenation argument, particularly using the flow property, to derive the desired bound over an interval with positive length and independent of \( \mathbf{X} \). %\todo{ we must specify how we remove the dependence on $\mathbf{X}$. change the notation $R$, this is usually a remainder, \textcolor{red}{i changed $R$ to $K$. At the end of proof we already specified this point, Or I did not get your point?}}
		Let \( [s,t] \subseteq [0,T] \) with \( t-s\leq 1 \). This interval will be further refined in the sequel. Let \( [a,b] \subseteq [s,t] \) be arbitrary and consider the controlled path defined on $[a,b]$ given by
		\begin{align}
			\tau\to K_{\tau}:=\int_{a}^{\tau}U_{\tau,u}G\left(u,\phi_{\mathbf{X}}(0,u,y)\right)\circ\mathrm{d}\mathbf{X}_u. 
		\end{align}
Recalling that \( G \) is linear and \( \phi_{\mathbf{X}}(0,\cdot, y) \) is the solution of~\eqref{Main_Equation}, we conclude from Lemma \ref{YBAS54} that
		\begin{align}\label{AUASMss}
			\begin{split}
				&\Vert K\Vert_{\tilde{\mathscr{D}}^{\gamma, p}_{\mathbf{X},\alpha}\left([a,b]\right)}\lesssim  \vert \phi_{\mathbf{X}}(0,a,y)\vert_{E_{\alpha}}\\
				&\quad+ \left[(b-a)^{\beta}+W_{\Pi^{1}(\mathbf{X}),\gamma,p}^{(\gamma-\sigma)}(a,b)+\sum_{0<j\leq N }W_{\Pi^{j}(\mathbf{X}),\gamma,p}^{j(\gamma-p)}(a,b)\right]\Vert\phi_{\mathbf{X}}(0, \cdot, y)\Vert_{\tilde{\mathscr{D}}^{\gamma, p}_{\mathbf{X},\alpha}\left([a,b]\right)}.
			\end{split}
		\end{align}
		Then, \eqref{GHy}, \eqref{GHsy}, and \eqref{AKaks} entail that %\footnote{Note that when we apply this equation, we must plug in \( s = 0 \), and the interval \( [s,t] \) used in Definition \ref{SOOLL} must not be conflated with the interval \( [s,t] \) used during this proof.}
%        \todo{ i would drop the first footnote, \textcolor{red}{removed}} yield 
		\begin{align}\label{AIsui}
			\begin{split}
				&\Vert \phi_{\mathbf{X}}(0,\cdot,y)\Vert_{\tilde{\mathscr{D}}^{\gamma, p}_{\mathbf{X},\alpha}\left([a,b]\right)}\\&\quad\lesssim \vert \phi_{\mathbf{X}}(0,a,y)\vert_{E_\alpha}+(b-a)^{{\delta_1}}\left(1+\Vert \phi_{\mathbf{X}}(0, \cdot, y)\Vert_{\tilde{\mathscr{D}}^{\gamma, p}_{\mathbf{X},\alpha}\left([a,b]\right)}\right)+\left\Vert K \right\Vert_{\tilde{\mathscr{D}}^{\gamma, p}_{\mathbf{X},\alpha}\left([a,b]\right)},
			\end{split}
		\end{align} 
		{By setting \( \delta_2 := \min\lbrace\delta_1, \beta\rbrace \) and applying \eqref{AUASMss} and \eqref{AIsui}, we obtain a deterministic constant \( L \geq 1 \) (i.e., independent of \( \mathbf{X} \) but dependent on \( F \) and \( G \))}, 
          such that the following inequality holds
		\begin{align}\label{UNN8825as}
			\begin{split}
				&\Vert \phi_{\mathbf{X}}(0,\cdot,y)\Vert_{\tilde{\mathscr{D}}^{\gamma, p}_{\mathbf{X},\alpha}\left([a,b]\right)}\leq L \Bigg[1+\vert \phi_{\mathbf{X}}(0,a,y)\vert_{E_\alpha}\\&+\Big[(b-a)^{\delta_2}+W_{\Pi^{1}(\mathbf{X}),\gamma,p}^{(\gamma-\sigma)}(a,b)+\sum_{0<j\leq N }W_{\Pi^{j}(\mathbf{X}),\gamma,p}^{j(\gamma-p)}(a,b)\Big]\Vert \phi_{\mathbf{X}}(0,\cdot,y)\Vert_{\tilde{\mathscr{D}}^{\gamma, p}_{\mathbf{X},\alpha}\left([a,b]\right)}\Bigg].
			\end{split}
		\end{align}
		We recall that 
		\begin{align}\label{UUJJJas}
			W_{\mathbf{X}, \gamma, p}(a,b) := \sum_{0< j \leq N} W_{\Pi^{j}(\mathbf{X}), \gamma, p}(a,b).
		\end{align}
		We refine the interval \( [s,t] \) and choose \( 0<\chi \leq t-s <1\) such that
		\begin{align}\label{8sudjlad}
			\begin{split}
				&L(t-s)^{\delta_2} \leq \frac{1}{3}, \\
				&(N+1)L\chi\leq\frac{1}{3}.
			\end{split}
		\end{align}
		By choosing this value of \( \chi \), we generate the greedy points \( \lbrace \tau^{[s,t]}_{m, \mathbf{X}}(\chi) \rbrace_{m \geq 0} \) as defined in Definition \ref{GRRED}. From Definition \ref{ADssdfg} we know that \( \gamma - \sigma > \gamma - p \). Moreover, the definition of greedy points, \eqref{UUJJJas}, and the conditions imposed on \( \chi \) obviously yield for every \( 1 \leq m \leq \tilde{N}([s,t], \chi, \mathbf{X}) \) and \( 0 < j \leq N \) that
		\begin{itemize}
			\item $W_{\Pi^{1}(\mathbf{X}),\gamma,p}^{(\gamma-\sigma)}\left(\tau^{[s,t]}_{m-1, \mathbf{X}}(\chi),\tau^{[s,t]}_{m, \mathbf{X}}(\chi)\right)\leq \chi$,
			\item $W_{\Pi^{j}(\mathbf{X}),\gamma,p}^{j(\gamma-p)}\left(\tau^{[s,t]}_{m-1, \mathbf{X}}(\chi),\tau^{[s,t]}_{m, \mathbf{X}}(\chi)\right)\leq\chi$.
		\end{itemize}
		Therefore
		\begin{align*}
			L\bigg[\left(\tau^{[s,t]}_{m, \mathbf{X}}(\chi)-\tau^{[s,t]}_{m-1, \mathbf{X}}(\chi)\right)^{\delta_2}&+W_{\Pi^{1}(\mathbf{X}),\gamma,p}^{(\gamma-\sigma)}\left(\tau^{[s,t]}_{m-1, \mathbf{X}}(\chi),\tau^{[s,t]}_{m, \mathbf{X}}(\chi)\right)\\&+\sum_{0<j\leq N }W_{\Pi^{j}(\mathbf{X}),\gamma,p}^{j(\gamma-p)}\left(\tau^{[s,t]}_{m-1, \mathbf{X}}(\chi),\tau^{[s,t]}_{m, \mathbf{X}}(\chi)\right)\bigg]\leq \frac{2}{3}.
		\end{align*}
        {
        From this observation, and by plugging in the interval~\( [\tau^{[s,t]}_{m-1, \mathbf{X}}(\chi), \tau^{[s,t]}_{m, \mathbf{X}}(\chi)] \) for \( [a,b] \) into \eqref{UNN8825as} and using \eqref{8sudjlad}, we deduce that for every \( 1 \leq m \leq \tilde{N}([s,t], \chi, \mathbf{X}) \)}
%\todo{ footnote should be in the text,\textcolor{red}{done}}
		\begin{align}\label{kosa;d85df}
			\Vert \phi_{\mathbf{X}}(0,\cdot,y)\Vert_{\tilde{\mathscr{D}}^{\gamma, p}_{\mathbf{X},\alpha}\left(\left[\tau^{[s,t]}_{m-1, \mathbf{X}}(\chi),\tau^{[s,t]}_{m, \mathbf{X}}(\chi)\right]\right)}\leq 3L+3L\left\vert \phi_{\mathbf{X}}(0,\tau^{[s,t]}_{m-1, \mathbf{X}}(\chi),y)\right\vert_{E_\alpha}.
		\end{align}
		This in particular yields that
		\begin{align*}
			\left\vert \phi_{\mathbf{X}}(0,\tau^{[s,t]}_{m, \mathbf{X}}(\chi),y)\right\vert_{E_\alpha}\leq 3L+3L\left\vert \phi_{\mathbf{X}}(0,\tau^{[s,t]}_{m-1, \mathbf{X}}(\chi),y)\right\vert_{E_\alpha}.
		\end{align*}
		It is now sufficient to a discrete Gronwall lemma to conclude that
		\begin{align}\label{Uasd54a7}
			\begin{split}
				&\max_{0\leq m\leq \tilde{N}([s,t], \chi, \mathbf{X})}\left\vert \phi_{\mathbf{X}}(0,\tau^{[s,t]}_{m, \mathbf{X}}(\chi),y)\right\vert_{E_\alpha}\\&\quad\leq\underbrace{\frac{\exp\left(\tilde{N}([s,t], \chi, \mathbf{X})\log(3L)\right)-1}{3L-1}}_{P_0(\mathbf{X},[s,t],\chi)}+\left(1+\left\vert \phi_{\mathbf{X}}(0,s,y)\right\vert_{E_\alpha}\right)\underbrace{\exp\left(\tilde{N}([s,t], \chi, \mathbf{X})\log(3L)\right)}_{P_1(\mathbf{X},[s,t],\chi)}.
			\end{split}
		\end{align}
		From \eqref{kosa;d85df} we obtain
		\begin{align*}
			\max_{1 \leq m \leq \tilde{N}([s,t], \chi, \mathbf{X})}&\Vert \phi_{\mathbf{X}}(0,\cdot,y)\Vert_{\tilde{\mathscr{D}}^{\gamma, p}_{\mathbf{X},\alpha}\left(\left[\tau^{[s,t]}_{m-1, \mathbf{X}}(\chi),\tau^{[s,t]}_{m, \mathbf{X}}(\chi)\right]\right)}\\&\quad\leq 3L\left(1+\max_{0\leq m\leq \tilde{N}([s,t], \chi, \mathbf{X})}\left\vert \phi_{\mathbf{X}}(0,\tau^{[s,t]}_{m, \mathbf{X}}(\chi),y)\right\vert_{E_\alpha}\right).
		\end{align*} 
		This combined with equation \eqref{Uasd54a7} implies that
		\begin{align}\label{85df55}
			\sup_{\tau\in [s,t]}\left\vert \phi_{\mathbf{X}}(0,\tau,y)\right\vert_{E_\alpha}\lesssim \left\vert \phi_{\mathbf{X}}(0,s,y)\right\vert_{E_\alpha}P_1(\mathbf{X},[s,t],\chi)+\underbrace{1+P_0(\mathbf{X},[s,t],\chi)+P_1(\mathbf{X},[s,t],\chi)}_{P_{2}(\mathbf{X},[s,t],\chi)}.
		\end{align}
		The conclusion %\todo{ what do you mean by upshot?\textcolor{red}{changed}} 
        of our argument is inequality \eqref{85df55}, which holds for every subinterval \( [s, t] \) with an interval length upper bounded by
		\begin{align*}  
			L_2 := \left( \frac{1}{3L} \right)^{\frac{1}{\sigma_2}},  
		\end{align*}  
		due to \eqref{8sudjlad}. This completes the proof.
		
	\end{proof}
	\section{Gaussian Rough Paths}\label{sec:ibound}
	All our previous results are entirely deterministic and do not require any probabilistic tools. As standard in the rough path theory, we provide now a probabilistic framework and  focus
	%incorporate  However, the central idea of the theory of rough paths is to view it as a substitute for the probabilistic approach when the equation is perturbed by a Gaussian signal. In this section, the discussion becomes predominantly probabilistic, and our goal is to integrate our results into this framework, focusing 
	on Gaussian processes that can be enhanced to a rough path $\mathbf{X}$ of regularity \(\gamma\in (\frac{1}{4}, \frac{1}{2}) \). This level of regularity is of particular interest in the context of fractional Brownian motion with Hurst index $H\in(\frac{1}{4},\frac{1}{2})$, as it allows for the canonical enhancement of this process to be a rough path.\\
	
	%However, we believe that considering even lower levels of regularity, i.e., \( 0<\gamma < \frac{1}{4} \), would be valuable for future research, especially once canonical procedures for defining iterated integrals are established. 
	Let us first introduce an assumption, which is used throughout this section. For more background on Gaussian measures, we refer to \cite{Led96}.
	\begin{assumption}\label{Cameron-Martin}
		We assume that \( X = (X^k)_{1 \leq k \leq d} \) is a centered Gaussian process with independent components and continuous trajectories, taking values in \( \mathbb{R}^d \).~Accordingly, we assume that \( (\mathcal{W}, \mathcal{H}, \mu) \) is the abstract Wiener space %\todo{Wiener space not measure,\textcolor{red}{done}} 
        associated with this process.~Here \( \mathcal{W} := C([0,T], \mathbb{R}^d) \) and \( \mathcal{H} \subset \mathcal{W}  \) %\todo{ E?, corretd} 
        is the Cameron–Martin space which is a Hilbert space.~Furthermore, we assume that the following statements hold.
		\begin{enumerate}
			\item 	 {For \(\gamma\in (\frac{1}{4}, \frac{1}{2})\setminus\lbrace\frac{1}{3}\rbrace \)} %\todo{please avoid footnotes, \textcolor{red}{done}}, 
            this process can be enhanced to a weakly geometric \( (p, \gamma) \)-rough path \( \mathbf{X} \). We emphasize the source of  randomness writing \( \mathbf{X}(\omega) \), where with a slight abuse of notation \( \Pi^{1}\mathbf{X}(\omega) = X(\omega)\).
			\item We assume that on a set of full measure (still denoted by \( \mathcal{W} \)), the control function %\todo{recall that this is the control function in Def~\ref{pgamma:rp}}
			\begin{align*}
				W_{\mathbf{X}(\omega),\gamma,p} \colon \Delta_{[0,T]} \rightarrow \mathbb{R},
			\end{align*}
			{which is defend in Definition \ref{pgamma:rp},} is continuous for every $\omega\in \mathcal{W}$.
			\item  We assume that there exists \( \gamma' > {\gamma} \) such that \( \gamma + \gamma' > 1 \). Moreover, for every \( h \in \mathcal{H} \)
			\begin{align}\label{sdkf44sdf}
				W_{\mathbf{h},\gamma^\prime,0}(0,T)<\infty.
			\end{align}
			Here, by \( \mathbf{h} \) we mean the natural \( (p, \gamma) \)-rough path, which can be defined as the Young integral of \( h \) with respect to itself.
			\item The following uniform inequality holds for every \( h \in \mathcal{H} \)
			\begin{align}\label{HASH825}
				W_{\mathbf{h},\gamma^\prime,p}(0,T)\lesssim \vert h\vert_{\mathcal{H}}^{\frac{1}{\gamma^\prime -p}},
			\end{align}
			where \( \vert \cdot \vert_{\mathcal{H}} \) denotes the norm on the Hilbert space \( \mathcal{H} \).
		\end{enumerate}
		%We use the convention $\mathbf{X}$ 
	\end{assumption}
	In the following we first revisit Borell’s inequality which is a well-known result in the theory of Gaussian measures.
	\begin{theorem}\label{BORELL}
		Let \((\mathcal{W}, \mathcal{H}, \mu)\) be an abstract Wiener space, and let \(A \subset \mathcal{W}\) be a measurable Borel set with \(\mu(A) > 0\). Let \(a \in (-\infty, +\infty]\) be chosen such that
		\[
		\mu(A) = \frac{1}{\sqrt{2\pi}} \int_{-\infty}^{a} \exp\left( -\frac{x^2}{2} \right) \, \mathrm{d}x =: \Phi(a).
		\]
		Let \(\mathcal{K}\) represent the unit ball in \(\mathcal{H}\) and let \(\mu_{\star}\) be the inner measure corresponding to \(\mu\). Then, for all \(r \geq 0\) we have the inequality
		\[
		\mu_{\star}(A + r\mathcal{K}) \coloneqq \mu_\star \left\{ a + rk : a \in A, k \in \mathcal{K} \right\} \geq \Phi(a + r).
		\]
	\end{theorem}
	\begin{proof}
		\cite[Theorem 4.3.]{Led96}
	\end{proof}
	A famous example of Gaussian process that satisfy Assumption~\eqref{Cameron-Martin} is given by fractional Brownian motion. For further examples we refer to \cite[Section 3.3]{BGS25}. %....\todo{Maybe bring the red parts here and elaborate more?}\newline
	Furthermore, we point out that Assumption \ref{Cameron-Martin} enables us to translate any element of $\mathbf{X}$ in the direction of $\mathcal{H}$ canonically via Young integration. This concept of translation was first introduced in \cite{CLL13} to study integrable bounds for the solution of rough differential equations. Our main strategy in the infinite-dimensional setting is inspired by~\cite{CLL13}.
	\begin{definition}\label{TRAAS}%\todo{ make this footnote a remark and refer to the third item}
	%	\todo{ generally avoid footnotes that contain important information and put them in the text}
		Let Assumption \ref{Cameron-Martin} hold. For $h\in\mathcal{H}$, the translated path 
		\begin{align}\label{AUSjaksevc}
			T_{{h}}(\mathbf{X}): \Delta_{[0,T]} \longrightarrow T^{3}(\mathbb{R}^d)
		\end{align}
		is defined as follows. %\todo{The first term in (5.4) is X in the top}\todo{Answer: corrected}
		\begin{itemize}
			\item [1)] Level one:
			\begin{align*}
				\left(\Pi^{1}(T_{{h}}(\mathbf{X}))\right)_{s,t}=\left(\Pi^{1}(\mathbf{X})\right)_{s,t}+\left(\Pi^{1}(\mathbf{h})\right)_{s,t}.
			\end{align*}
			\item [2)] Level two:
			\begin{align}\label{ITET2}
				\begin{split}
					&\left(\Pi^{2}(T_{{h}}(\mathbf{X}))\right)_{s,t}=\\ &\underbrace{\left(\Pi^{2}(\mathbf{X})\right)_{s,t}}_{\left(\Forest{[X[X]]}\right)_{s,t}}+\underbrace{\int_{s}^{t}\left(\Pi^{1}(\mathbf{X})\right)_{s,u}\otimes \mathrm{d}\left(\Pi^{1}(\mathbf{h})\right)_{u}}_{\left(\Forest{[h[X]]}\right)_{s,t}}+\underbrace{\int_{s}^{t}\left(\Pi^{1}(\mathbf{h})\right)_{s,u}\otimes \mathrm{d}\left(\Pi^{1}(\mathbf{X})\right)_{u}}_{\left(\Forest{[X[h]]}\right)_{s,t}}+\underbrace{\left(\Pi^{2}(\mathbf{h})\right)_{s,t}}_{\left(\Forest{[h[h]]}\right)_{s,t}}.
				\end{split}
			\end{align}
			\item [3)] Level three:
			\begin{align}\label{ITET3}
				\begin{split}
					&\left(\Pi^{3}(T_{{h}}(\mathbf{X}))\right)_{s,t}=\\&\left(\Forest{[X[X[X]]]}\right)_{s,t}+\left(\Forest{[X[X[h]]]}\right)_{s,t}+\left(\Forest{[X[h[X]]]}\right)_{s,t}+\left(\Forest{[h[X[X]]]}\right)_{s,t}+\left(\Forest{[X[h[h]]]}\right)_{s,t}+\left(\Forest{[h[X[h]]]}\right)_{s,t}+\left(\Forest{[h[h[X]]]}\right)_{s,t}+\left(\Forest{[h[h[h]]]}\right)_{s,t}.
				\end{split}
			\end{align}
			The mixture of $\mathbf{h}$ and $\mathbf{X}$ in the second item is defined using Young integration. In the last item, by the 3-ladder trees, we refer to the integration of the above 2-ladder trees against the root. For example
			\begin{align*}
				\left(\Forest{[X[X[h]]]}\right)_{s,t} := \int_{s}^{t} \left(\Forest{[X[h]]}\right)_{s,u} \otimes \mathrm{d}\left(\Pi^{1}(\mathbf{X})\right)_{u}, \quad \left(\Forest{[h[h[X]]]}\right)_{s,t} = \int_{s}^{t} \left(\Forest{[h[X]]}\right)_{s,u} \otimes \mathrm{d}\left(\Pi^{1}(\mathbf{h})\right)_{u}.
			\end{align*}
			{Except for the rough terms, i.e., those involving only \( \mathbf{X} \) and assumed to be given, the remaining terms are defined using Young integration and the classical Sewing lemma. We provide more details on the construction of these integrals below.
} %\todo{ state which term is defined with Young and which with rough, \textcolor{red}{done}}
		\end{itemize} 
	\end{definition}
   
    \begin{remark}
For $\gamma\in(\frac{1}{3},\frac{1}{2})$, the third level in Definition \ref{TRAAS} can be omitted.
    \end{remark}
    
	Now, we estimate the translated path in terms of our controlled functions. The results presented below generalize \cite[Lemma 3.1]{CLL13} and \cite[Lemma 2.8, Lemma 2.9]{GVR25} by incorporating the control functions given in Definition~\ref{pgamma:rp}.  
	\begin{lemma}\label{ABC1}
		The following estimate is valid for all \( 0 \leq p < \gamma<\gamma' \) and for every interval \( [s, t]\subseteq [0,T] \) 		\begin{align}\label{KM5a}
			W_{\Pi^{1}(\mathbf{X+h}),\gamma,p}(s,t)\lesssim W_{\Pi^{1}(\mathbf{X}),\gamma,p}(s,t)+W_{\Pi^{1}(\mathbf{h}),\gamma^\prime,p}^{\frac{\gamma^\prime -p}{\gamma-p}}(s,t).
		\end{align}
	\end{lemma}
	\begin{proof}
		First, note that from the Minkowski inequality, for \( [s,t] \subseteq [0,T] \) and \( 0 \leq p < \gamma \), we have
		\begin{align}\label{52sdpp}
			W_{\Pi^{1}(\mathbf{X+h}),\gamma,p}^{\gamma-p}(s,t)\leq W_{\Pi^{1}(\mathbf{X}),\gamma,p}^{\gamma-p}(s,t)+W_{\Pi^{1}(\mathbf{h}),\gamma^\prime,p}^{\gamma-p}(s,t).
		\end{align}
		Since \( \gamma' > \gamma \), by a standard argument and using Lemma \ref{INM}, we conclude that
		\begin{align*}
			W_{\Pi^{1}(\mathbf{h}),\gamma^\prime,p}^{\gamma-p}(s,t)\leq W_{\Pi^{1}(\mathbf{h}),\gamma^\prime,p}^{\gamma^\prime-p}(s,t).
		\end{align*}
		This inequality together with \eqref{52sdpp} entails \eqref{KM5a}.
	\end{proof}
	Now, we obtain the same result for the iterated integrals.
	\begin{lemma}\label{AYASs}
		The following estimate is valid for all \( 0 \leq p < \gamma \) and every interval \( [s, t]\subseteq [0,T] \) 
		\begin{align*}
			W_{\Pi^{2}(\mathbf{X+h}),\gamma,p}(s,t)\lesssim W_{\Pi^{2}(\mathbf{X}),\gamma,p}(s,t)+W_{\Pi^{2}(\mathbf{h}),\gamma^\prime,p}^{\frac{\gamma^\prime -p}{\gamma-p}}(s,t).
		\end{align*}
	\end{lemma}
	\begin{proof}
		We assume that $[\tau_1, \tau_2] \subseteq [s, t]$. Note that, by the definition of the translated rough path, we must estimate the terms $\Forest{[X[h]]}$, $\Forest{[h[X]]}$, and $\Forest{[X[X]]}$ separately. We only treat the term $\Forest{[h[X]]}$, since the rest follow by similar computations.~The assumption $\gamma + \gamma^\prime > 1$ allows us to define the second iterated integral, which is a mixture of $\mathbf{h}$ and $\mathbf{X}$.  To be more precise, first note that from Lemma \ref{sudj14}, we have
		\begin{align}
			W_{\Pi^{2}(\mathbf{X}),\gamma,0}(s,t) < \infty.
		\end{align}
		Now, since $\gamma + \gamma^\prime > 1$, from classical Young integration, we obtain that
		\begin{align}\label{4as}
			\left|\left(\Forest{[h[X]]}\right)_{\tau_1,\tau_2}\right|\lesssim W_{\Pi^{1}(\mathbf{X}),\gamma,0}^{\gamma}(\tau_1,\tau_2) W_{\Pi^{1}(\mathbf{h}),\gamma^\prime,0}^{\gamma^\prime}(\tau_1,\tau_2).
		\end{align}
		Now from Lemma \ref{sudj14}, we conclude that
		\begin{align}\label{sdsa85d}
			\begin{split}
				\left|\left(\Forest{[h[X]]}\right)_{\tau_1,\tau_2}\right|&\lesssim (\tau_2-\tau_1)^{2p}W_{\Pi^{1}(\mathbf{X}),\gamma,p}^{\gamma-p}(\tau_1,\tau_2) W_{\Pi^{1}(\mathbf{h}),\gamma^\prime,p}^{\gamma^\prime-p}(\tau_1,\tau_2)\\&=(\tau_2-\tau_1)^{2p}\left(W_{\Pi^{1}(\mathbf{X}),\gamma,p}^{\frac{\gamma-p}{\gamma+\gamma^\prime-2p}}(\tau_1,\tau_2)W_{\Pi^{1}(\mathbf{h}),\gamma^\prime,p}^{\frac{\gamma^\prime-p}{\gamma+\gamma^\prime-2p}}(\tau_1,\tau_2)\right)^{\gamma+\gamma^\prime-2p}.
			\end{split}
		\end{align}
		This in turn yields 
		\begin{align*}
			\frac{\left|\left(\Forest{[h[X]]}\right)_{\tau_1,\tau_2}\right|^{\frac{1}{2(\gamma-p)}}}{(\tau_2-\tau_1)^{\frac{p}{\gamma-p}}}\lesssim \left(\underbrace{W_{\Pi^{1}(\mathbf{X}),\gamma,p}^{\frac{\gamma-p}{\gamma+\gamma^\prime-2p}}(\tau_1,\tau_2)W_{\Pi^{1}(\mathbf{h}),\gamma^\prime,p}^{\frac{\gamma^\prime-p}{\gamma+\gamma^\prime-2p}}(\tau_1,\tau_2)}_{\tilde{W}(\tau_1,\tau_2)}\right)^{\frac{\gamma+\gamma^\prime-2p}{2(\gamma-p)}}.
		\end{align*}
		Now, since \( \tilde{W} \) is a control function, \( \tau_1 \) and \( \tau_2 \) are chosen arbitrarily, and \( \gamma + \gamma' - 2p > 2(\gamma - p) \), it follows from Lemma \ref{INM} that
		\begin{align*}
			W_{\Forest{[h[X]]},\gamma,p}(s,t)\lesssim W_{\Pi^{1}(\mathbf{X}),\gamma,p}^{\frac{\gamma-p}{2(\gamma-p)}}(s,t)W_{\Pi^{1}(\mathbf{h}),\gamma^\prime,p}^{\frac{\gamma^\prime-p}{2(\gamma-p)}}(s,t)\leq W_{\Pi^{1}(\mathbf{X}),\gamma,p}(s,t)+W_{\Pi^{1}(\mathbf{h}),\gamma^\prime,p}^{\frac{\gamma^\prime-p}{\gamma-p}}(s,t).
		\end{align*}
		Note that here we also used the obvious inequality \( x^{\theta}y^{1-\theta} \leq x + y \), which holds for every \( \theta \in [0,1] \) and positive values of \( x \) and \( y \). This finishes the proof.
	\end{proof}
       \begin{remark}\label{lzds}
        The same inequality for the corresponding control holds if we replace \( \Forest{[x[h]]} \) in \eqref{4as}.
        \end{remark}
	\begin{remark}\label{adar852as}
		From \eqref{sdsa85d}, we conclude that for all $0\leq p< \gamma$ the following estimate holds 
		\begin{align*}
			\frac{\left|\left(\Forest{[h[X]]}\right)_{\tau_1,\tau_2}\right|^{\frac{1}{\gamma^\prime+\gamma-2p}}}{(\tau_2-\tau_1)^{\frac{2p}{\gamma^\prime+\gamma-2p}}}\lesssim W_{\Pi^{1}(\mathbf{X}),\gamma,p}^{\frac{\gamma-p}{\gamma+\gamma^\prime-2p}}(\tau_1,\tau_2)W_{\Pi^{1}(\mathbf{h}),\gamma^\prime,p}^{\frac{\gamma^\prime-p}{\gamma+\gamma^\prime-2p}}(\tau_1,\tau_2).
		\end{align*}
Then, using the same argument as in the proof of Lemma \ref{AYASs}, we get
\begin{align*}
			W_{\Forest{[h[X]]},\frac{\gamma+\gamma^\prime}{2},p}^{\gamma+\gamma^\prime-2p}(s,t)\lesssim W_{\Pi^{1}(\mathbf{X}),\gamma,p}^{\gamma-p}(s,t)W_{\Pi^{1}(\mathbf{h}),\gamma^\prime,p}^{\gamma^\prime-p}(s,t).
		\end{align*}
	\end{remark}
   
    \begin{remark}\label{5asora}
With minor modifications to the proof of Lemma \ref{sudj14}, we obtain the following inequalities for the control function for every \( [s,t] \subseteq [0,T] \)
 \begin{align*}
&W_{\Forest{[h[X]]},\frac{\gamma+\gamma^\prime}{2},0}^{\gamma^\prime+\gamma}(s,t)\leq (t-s)^{2p} W_{\Forest{[h[X]]},\frac{\gamma+\gamma^\prime}{2},p}^{\gamma^\prime+\gamma-2p}(s,t),\\
&W_{\Forest{[X[h]]},\frac{\gamma+\gamma^\prime}{2},0}^{\gamma^\prime+\gamma}(s,t)\leq (t-s)^{2p} W_{\Forest{[X[h]]},\frac{\gamma+\gamma^\prime}{2},p}^{\gamma^\prime+\gamma-2p}(s,t)  .
 \end{align*}
\end{remark}
	The method to estimate the third iterated integrals arising in~\eqref{ITET3} is similar to the previous deliberations. For the reader's convenience we outline the main ideas and first clarify how these iterated integrals can be defined. In \eqref{ITET3}, there are eight terms and each must be estimated separately. We focus on \( \Forest{[X[h[X]]]} \), as it requires more involved considerations. The approach for the remaining terms follows similar arguments.~For an arbitrary interval \( [a,b] \subseteq [0,T] \)
	\begin{align*}
		\left(\Forest{[X[h[X]]]}\right)_{a,b} := \lim_{\substack{\vert \pi \vert \to 0 \\ \pi \in P[a,b]}} \sum_{k} \Xi_{\tau_{k},\tau_{k+1}},
	\end{align*}
	where 
	\[
	\Xi_{\tau_{k},\tau_{k+1}} = (\Forest{[h[X]]})_{a,\tau_k} \otimes (\Pi^{1}(\mathbf{X}))_{\tau_k,\tau_{k+1}} 
	+ (\Pi^{1}(\mathbf{X}))_{a,\tau_k} \otimes (\Forest{[X[h]]})_{\tau_k,\tau_{k+1}},
	\]
	according to \cite[Theorem 4.8]{GVR23B}. By definition
	\begin{align*}
		\Xi_{\tau_{k},\tau_{k+2}}-\Xi_{\tau_{k},\tau_{k+1}} -\Xi_{\tau_{k+1},\tau_{k+2}}=- (\Forest{[h[X]]})_{\tau_k,\tau_{k+1}}\otimes (\Pi^{1}(\mathbf{X}))_{\tau_{k+1},\tau_{k+2}}-(\Pi^{1}(\mathbf{X}))_{\tau_{k},\tau_{k+1}}\otimes (\Forest{[X[h]]})_{\tau_{k+1},\tau_{k+2}}.
	\end{align*}
	Recalling \eqref{4as} and Remark \ref{lzds}, %\todo{ we should incorporate the footnotes in the next, more precisely and write for e.g. exactly which type of inequality holds, \textcolor{red}{I added a remark and incorporated}}
     one can conclude that \( \Forest{[X[h[X]]]} \) is well-defined since \( \gamma' + 2\gamma > 1 \) based on Lemma \ref{AYASs} and the classical Sewing lemma. Moreover,
	\begin{align}\label{AUsia}
		\left|\left(\Forest{[X[h[X]]]}\right)_{a,b}\right|\lesssim  W_{\Forest{[h[X]]},\frac{\gamma+\gamma^\prime}{2},0}^{\gamma^\prime+\gamma}(a,b)W_{\Pi^{1}(\mathbf{X}),\gamma,0}^{\gamma}(a,b)+W_{\Pi^{1}(\mathbf{X}),\gamma,0}^{\gamma}(a,b)W_{\Forest{[X[h]]},\frac{\gamma+\gamma^\prime}{2},0}^{\gamma^\prime+\gamma}(a,b).
	\end{align}
	Furthermore, we obtain:
	\begin{lemma}\label{ABC2}
		 For every $0 \leq p < \gamma$ and for every $[s, t]\subseteq [0, T]$, we have the following estimate 
		\begin{align}\label{ASjSAD}
			W_{\Pi^{3}(\mathbf{X+h}),\gamma,p}(s,t)\lesssim W_{\Pi^{3}(\mathbf{X}),\gamma,p}(s,t)+W_{\Pi^{3}(\mathbf{h}),\gamma^\prime,p}^{\frac{\gamma^\prime -p}{\gamma-p}}(s,t).
		\end{align}
	\end{lemma}
	\begin{proof}
		We focus on \( \Forest{[X[h[X]]]} \). We assume that \( [\tau_1, \tau_2] \subseteq [s, t] \). Then from \ref{AUsia}, Lemma \ref{sudj14} and Remark \ref{5asora}%\todo{ no footnote,\textcolor{red}{See remark \ref{5asora}}}, we conclude that
		\begin{align*}
			&\left|\left(\Forest{[X[h[X]]]}\right)_{\tau_1, \tau_2}\right|\lesssim  W_{\Forest{[h[X]]},\frac{\gamma+\gamma^\prime}{2},0}^{\gamma^\prime+\gamma}(\tau_1, \tau_2)W_{\Pi^{1}(\mathbf{X}),\gamma,0}^{\gamma}(\tau_1, \tau_2)+W_{\Pi^{1}(\mathbf{X}),\gamma,0}^{\gamma}(\tau_1, \tau_2)W_{\Forest{[X[h]]},\frac{\gamma+\gamma^\prime}{2},0}^{\gamma^\prime+\gamma}(\tau_1, \tau_2)\\&\myquad[2]\lesssim (\tau_2-\tau_1)^{3p}\left( W_{\Forest{[h[X]]},\frac{\gamma+\gamma^\prime}{2},p}^{\gamma^\prime+\gamma-2p}(\tau_1, \tau_2)W_{\Pi^{1}(\mathbf{X}),\gamma,p}^{\gamma-p}(\tau_1, \tau_2)+W_{\Pi^{1}(\mathbf{X}),\gamma,p}^{\gamma-p}(\tau_1, \tau_2)W_{\Forest{[X[h]]},\frac{\gamma+\gamma^\prime}{2},p}^{\gamma^\prime+\gamma-2p}(\tau_1, \tau_2)\right).
		\end{align*}
		This, together with Remark \ref{adar852as}, implies that
		\begin{align*}
			\frac{\left|\left(\Forest{[X[h[X]]]}\right)_{\tau_1,\tau_2}\right|^{\frac{1}{3(\gamma-p)}}}{(\tau_2-\tau_1)^{\frac{p}{\gamma-p}}}\lesssim\left( W_{\Pi^{1}(\mathbf{X}),\gamma,p}^{\frac{2(\gamma-p)}{2\gamma+\gamma^\prime-3p}}(\tau_1,\tau_2)W_{\Pi^{1}(\mathbf{h}),\gamma^\prime,p}^{\frac{\gamma^\prime-p}{2\gamma+\gamma^\prime-3p}}(\tau_1,\tau_2)\right)^{\frac{2\gamma+\gamma^\prime-3p}{3(\gamma-p)}}.
		\end{align*}
		The rest of the proof is almost the same as the proof of Lemma \ref{AYASs}. Since \( 2\gamma + \gamma' - 3p > 3(\gamma - p) \) and \( \tau_1, \tau_2 \) are arbitrary, it follows from Lemma \ref{INM} that
		\begin{align*}
			W_{\Forest{[X[h[X]]]},\gamma,p}(s,t)&\lesssim \left( W_{\Pi^{1}(\mathbf{X}),\gamma,p}^{\frac{2(\gamma-p)}{2\gamma+\gamma^\prime-3p}}(s,t)W_{\Pi^{1}(\mathbf{h}),\gamma^\prime,p}^{\frac{\gamma^\prime-p}{2\gamma+\gamma^\prime-3p}}(s,t)\right)^{\frac{2\gamma+\gamma^\prime-3p}{3(\gamma-p)}}=W_{\Pi^{1}(\mathbf{X}),\gamma,p}^{\frac{2}{3}}(s,t)W_{\Pi^{1}(\mathbf{h}),\gamma^\prime,p}^{\frac{\gamma^\prime-p}{3(\gamma-p)}}(s,t)\\&\leq W_{\Pi^{1}(\mathbf{X}),\gamma,p}(s,t)+W_{\Pi^{1}(\mathbf{h}),\gamma^\prime,p}^{\frac{\gamma^\prime-p}{\gamma-p}}(s,t).
		\end{align*}
		As we already mentioned, the other terms in \eqref{ITET3} can be treated similarly. Thus, our claim is proven.
	\end{proof}
	The following statement summarizes the previous auxiliary results.
	\begin{proposition}\label{PRPOO}
		On a set \( \tilde{\mathcal{W}} \) of full measure in \( \mathcal{W} \), the following statements hold true:
		\begin{enumerate}
			\item For every $\omega\in \tilde{\mathcal{W}} $ , and for every \( h \in \mathcal{H} \), we have
			\begin{align}\label{sdsd84}
				T_{{h}}(\mathbf{X})(\omega)\equiv \mathbf{X}(\omega+h).
			\end{align}
			In particular, this implies that
			\begin{align}\label{sdsd85}
				T_{{h}}(\mathbf{X})(\omega-h)\equiv \mathbf{X}(\omega).
			\end{align}
			\item For every subinterval \( [s, t] \) of \( [0, T] \) and \( 0 \leq p < \gamma \), there exists a constant \( C=C(p,\gamma,\gamma^\prime) \) such that the following estimate holds 
			\begin{align}\label{sdsd86}
				W_{T_{{h}}(\mathbf{X}),\gamma,p}^{\gamma-p}(s,t)\leq C\left( W_{\mathbf{X},\gamma,p}^{\gamma-p}(s,t)+W_{\mathbf{h},\gamma^\prime,p}^{\gamma^\prime-p}(s,t)\right).
			\end{align}
			\begin{comment}
				This, in particular, yields that for another constant \( \tilde{M}=\tilde{M}(p,\gamma,\gamma^\prime) \), we obtain that
				\begin{align*}
					W_{T_{{h}}(\mathbf{X}),\gamma,p}^{\gamma-p}(s,t)\leq \tilde{M}\left( W_{\mathbf{X},\gamma,p}^{\gamma-p}(s,t)+\vert h\vert_{\mathcal{H}}\right).
				\end{align*}
			\end{comment}
		\end{enumerate}
	\end{proposition}
	\begin{proof}
		The first assertion was established in \cite[Lemma 5.4]{CLL13}.
        {Keeping in mind the inequality \eqref{785a}, the second statement is simply obtained from Lemma \ref{ABC1}, Lemma \ref{AYASs}, and Lemma \ref{ABC2}.}%\todo{more details}
	\end{proof}
	\begin{remark}
		As before, we use \( \mathcal{W} \) to denote the set of full measure \( \tilde{\mathcal{W}} \), which satisfies the conditions of the previous proposition.
	\end{remark}
	Proposition \ref{PRPOO} provides an estimate for the translated path. Additionally, recall that in Theorem \ref{PRIORI}, we proved a priori bound (roughly speaking) in terms of the number of greedy points for any interval whose length is small enough but independent of \( \mathbf{X} \). The main idea is to investigate the tail of this random variable by using Proposition \ref{PRPOO} and Borell’s inequality, similar to the finite-dimensional case treated in \cite{CLL13}. % in the finite-dimensional case, %which itself also uses the idea of the Grönwall inequality (in a probabilistic environment) in the background. this is not so relevent right now
	We first state an auxiliary result.
	\begin{lemma}\label{suppoert}
		Assume that $\chi>0$ and $[s,t]\subseteq [0,T]$. Then 
		\begin{align*}
			\mu\left\lbrace \omega : \frac{\chi}{2}\leq W_{\mathbf{X}(\omega),\gamma,p}(s,t)\leq \chi\right\rbrace>0.
		\end{align*}
	\end{lemma}
	\begin{proof}
		The proof follows from the support theorem for the Gaussian rough paths. {For more details, we refer to \cite[Section 15.8]{FV10}.}
       % \todo{Add reference from \cite{FV10}}
	\end{proof}
	 The following result is a  generalized version of \cite[Proposition 2.16]{GVR25}. Compared to that result, we obtain here that \( p \) can be taken arbitrarily close to \( \gamma \).~Additionally, this proposition allows us to include the case \( \gamma\in(\frac{1}{4},\frac{1}{3}) \).
	\begin{theorem}\label{INTEFGA}
		Let Assumption \ref{Cameron-Martin} hold, \( \chi > 0 \) be given, and \( [s, t] \subseteq [0, T] \). Then, there exist two deterministic constants \( M_{1}(\chi, p) \) and \( M_{2}(\chi, p) \) such that for every \( n \in \mathbb{N} \)
		\begin{align*}
			\mu\left\lbrace\omega : \tilde{N}([s,t], \chi, \mathbf{X}(\omega))>n\right\rbrace\leq M_{1}(\chi,p)\exp\left(-M_{2}(\chi,p)n^{2(\gamma^\prime-p)}\right).
		\end{align*}
	\end{theorem}
	\begin{proof}
		Throughout the proof, we assume that we are working on a set of full measure \( \tilde{\mathcal{W}} \), as specified in Definition \ref{PRPOO}. For \( \tilde{\chi} > 0 \), which will be determined below, we consider the greedy points \( \lbrace \tau_{m,\mathbf{X}(\omega)}^{[s,t]}(\tilde{\chi}) \rbrace_{0 \leq m \leq \tilde{N}([s,t], \tilde{\chi}, \mathbf{X}(\omega))} \), as introduced in Definition \ref{s}. Recalling that \( W_{\mathbf{X}(\omega),\gamma,p}\) is continuous, from \eqref{sdsd85} and \eqref{sdsd86} and for $1\leq m< \tilde{N}([s,t], \tilde{\chi}, \mathbf{X}(\omega))$, we have %\todo{why $\tilde{\chi}$ on the LHS of (4.16)? \textcolor{red}{since the continuity holds, it is mentioned  in the sentience above} i see}
		\begin{align}\label{UJNMLsd}
			\begin{split}
				\tilde{\chi}&=W_{\mathbf{X}(\omega),\gamma,p}^{\gamma-p}\left(\tau_{m-1,\mathbf{X}(\omega)}^{[s,t]}(\tilde{\chi}),\tau_{m,\mathbf{X}(\omega)}^{[s,t]}(\tilde{\chi})\right)=W_{T_{{h}}\left(\mathbf{X}(\omega-h)\right),\gamma,p}^{\gamma-p}\left(\tau_{m-1,\mathbf{X}(\omega)}^{[s,t]}(\tilde{\chi}),\tau_{m,\mathbf{X}(\omega)}^{[s,t]}(\tilde{\chi})\right)\\&\leq C\left[ W_{\mathbf{X}(\omega-h),\gamma,p}^{\gamma-p}\left(\tau_{m-1,\mathbf{X}(\omega)}^{[s,t]}(\tilde{\chi}),\tau_{m,\mathbf{X}(\omega)}^{[s,t]}(\tilde{\chi})\right)+W_{\mathbf{h},\gamma^\prime,p}^{\gamma^\prime-p}\left(\tau_{m-1,\mathbf{X}(\omega)}^{[s,t]}(\tilde{\chi}),\tau_{m,\mathbf{X}(\omega)}^{[s,t]}(\tilde{\chi})\right)\right].
			\end{split}
		\end{align}
		We set %\todo{who is $\chi_1$?\textcolor{red}{is defined. since I wanted to emphasis dependency on $\omega-h$, i refrained using $\tilde{\chi}$. At the end it is matter of taste} i know, i would drop $\xi_1$ because i find it confusing, the dependency on $\omega-h$ is in W, so i think it's clear,\textcolor{red}{done}}
		\begin{align*}
			\tilde{\chi}:=2C W_{\mathbf{X}(\omega-h),\gamma,p}^{\gamma-p}(s,t).
		\end{align*}
		Then, from the obvious inequality
		\[
		W_{\mathbf{X}(\omega-h),\gamma,p}^{\gamma-p}\left( \tau_{m-1, \mathbf{X}(\omega)}^{[s,t]}(\tilde{\chi}), \tau_{m, \mathbf{X}(\omega)}^{[s,t]}(\tilde{\chi}) \right) \leq \frac{\tilde{\chi}}{2C}
		\]
		and \eqref{UJNMLsd}, we conclude that
		\begin{align*}
			W_{\mathbf{X}(\omega),\gamma,p}^{\gamma-p}(s,t)\leq W_{\mathbf{h},\gamma^\prime,p}^{\gamma^\prime-p}\left(\tau_{m-1,\mathbf{X}(\omega)}^{[s,t]}(\tilde{\chi}),\tau_{m,\mathbf{X}(\omega)}^{[s,t]}(\tilde{\chi})\right).
		\end{align*}
		Moreover, we have $W_{\mathbf{X}(\omega),\gamma,p}^{\frac{\gamma-p}{\gamma^\prime-p}}(s,t) \leq W_{\mathbf{h},\gamma^\prime,p}\left( \tau_{m-1,\mathbf{X}(\omega)}^{[s,t]}(\tilde{\chi}), \tau_{m,\mathbf{X}(\omega)}^{[s,t]}(\tilde{\chi}) \right).$
		Thus, by summing over \( m \) with {$1 \leq m < \tilde{N}([s,t], \tilde{\chi}, \mathbf{X}(\omega))$} %\todo{ but the footnote in the text,\textcolor{red}{done}}, 
         and using the fact that \( W_{\mathbf{h},\gamma^\prime,p} \) is a control function, together with \eqref{HASH825}, we obtain the following result for a constant \( \tilde{C} \), by substituting the value of \( \tilde{\chi} \)
		\begin{align}\label{SAId}
			\left[\tilde{N}\left([s,t],2C W_{\mathbf{X}(\omega-h),\gamma,p}^{\gamma-p}(s,t),\mathbf{X}(\omega)\right)-1\right] W_{\mathbf{X}(\omega-h),\gamma,p}^{\frac{\gamma-p}{\gamma^\prime-p}}(s,t)\leq W_{\mathbf{h},\gamma^\prime,p}(s,t)\leq \tilde{C}\vert h\vert_{\mathcal{H}}^{\frac{1}{\gamma^\prime -p}}.
		\end{align}
		Now {recall} that $\chi>0$ is given and that $\mathcal{K}$ denotes the unit ball in $\mathcal{H}$. %\todo{this was already assumed in the statement of Theorem 4.14,\textcolor{red}{modified}}
        Then 
		\begin{align}\label{ajd85}
			\bigg{\lbrace}\omega : \tilde{N}([s,t], \chi, \mathbf{X}(\omega))>n\bigg{\rbrace}\bigsubseteq \mathcal{W}\bigsetminus\left\lbrace\underbrace{\left\lbrace \omega:\frac{\chi}{2} \leq 2C W_{\mathbf{X}(\omega),\gamma,p}^{\gamma-p}(s,t)\leq \chi\right\rbrace}_{A} +r_{n}\mathcal{K}\right\rbrace,
		\end{align}
		where 
		\begin{align}\label{UJMAs}
			r_{n}:=(\frac{n-1}{\tilde{C}})^{\tilde{\gamma}-p}(\frac{\chi}{4C}).
		\end{align}
		Indeed, to prove \eqref{ajd85}, first note that if \( \chi_1 \leq \chi_2 \), then $\tilde{N}([s,t], \chi_2, \mathbf{X}(\omega)) \leq \tilde{N}([s,t], \chi_1, \mathbf{X}(\omega)).$ %\todo{ what do you mean by a standard condition of membership argument?\textcolor{red}{Take an element for the set in the left and check that also belongs in the set on the right}\textcolor{blue}{ i would write this explicitly, i did use the formulation membership argument so far},\textcolor{red}{I chaged it}}
		{Then, by using \eqref{SAId}, we can verify that \eqref{ajd85} holds.} From Lemma \ref{suppoert}, we know that $\mu(A)>0$. Thus, we can apply Borell’s inequality. For \( \mu(A) = \Phi(a) \) we conclude according to Theorem \ref{BORELL} that
		\begin{align*}
			\mu\left\lbrace\omega : \tilde{N}([s,t], \chi, \mathbf{X}(\omega))>n\right\rbrace&\leq 1-\Phi(a+r_n)=\frac{1}{\sqrt{2\pi}}\int_{a+r_n}^{\infty} \exp\left( -\frac{x^2}{2} \right) \, \mathrm{d}x \\& \lesssim \exp\left(\frac{-(a+r_{n})^2}{4}\right).
		\end{align*}
		Plugging in the value of \( r_n \) established in \eqref{UJMAs} proves the statement. 
	\end{proof}
	We impose an assumption which can easily be verified for a large family of Gaussian processes, see~\cite[Section 3.3]{BGS25} for examples. In particular, this enables us to work with paths of regularity $\gamma\in(\frac{1}{4},\frac{1}{3})$.
	\begin{assumption}\label{YAHSs}
		We assume that
		\begin{align}
			2\left(\gamma^\prime - \frac{\sigma + 3\gamma}{4}\right) > 1.
		\end{align}
		Consequently, in Assumption %\todo{can we simply choose $p$ instead of $p_1$? Ass 2.10 is written for $p$ and later we choose $p=p_1$,\textcolor{red}{I did and add one sentence to the Theorem below to address this change}} 
        \ref{ADssdfg} and for \( N = 3 \)\ we can choose \( \frac{\sigma + 3\gamma}{4} < p < \gamma \)  such that
		\begin{align}\label{HNNABS}
			2\left(\gamma^\prime - p\right) > 1.
		\end{align}
	\end{assumption}

\begin{remark}
   \begin{itemize}
       \item  Note that $p$ can be taken arbitrarily close to \( \frac{\sigma + 3\gamma}{4} \).
       \item Considering $\gamma\in(\frac{1}{3},\frac{1}{2})$ in Assumption \ref{YAHSs}, we can take $N=2$, as naturally expected.
   \end{itemize}

\end{remark}

Now, we can state the main result of this section.% \todo{below replace $R$, this notation should strictly refer to the remainder.,\textcolor{red}{done}}
%\todo{ we still have $R$ below in the statement of the theorem, for the radius, i.e. $|y|_\alpha\leq R$ etc, \textcolor{red}{Sorry, I changed it to $\rho$}}
	\begin{theorem}\label{thm:ibound}
		Let all the assumptions of Theorem \ref{PRIORI} hold. In addition, assume that \( \mathbf{X} \) is a Gaussian process satisfying Assumption \ref{Cameron-Martin} . Then for every \( \frac{\sigma + 3\gamma}{4} < p < \gamma \) and \( \rho > 0 \), we can find positive deterministic constants \( M_{1,p} \) and \( M_{2,p} \), which also depend on \( \rho \), \( F \), and \( G \), such that the following tail estimate holds for \( n \in \mathbb{N} \)
		\begin{align}\label{UJJM854}
			\mu\left\lbrace \omega: \sup_{\vert y\vert_{E_\alpha}\leq \rho}\sup_{\tau\in [0,T]}\left\vert \phi_{\mathbf{X}(\omega)}(0,\tau,y)\right\vert_{E_\alpha}> n\ \right\rbrace\leq M_{1,p}\exp(-M_{2,p}n^{2(\gamma^\prime -p)}).
		\end{align}
		{In particular, under Assumption \ref{YAHSs}, we obtain that
}
		\begin{align}\label{UJJM855}
			\sup_{\vert y\vert_{E_\alpha}\leq \rho}\sup_{\tau\in [0,T]}\left\vert \phi_{\mathbf{X}(\omega)}(0,\tau,y)\right\vert_{E_\alpha}\in \bigcap_{q\geq 1}L^{q}(\mathcal{W}).
		\end{align}
	\end{theorem}
	\begin{proof}
		The proof is based on Theorem \ref{PRIORI} and Theorem \ref{INTEFGA}. Throughout the proof, we use the setting and notations of Theorem \ref{PRIORI}. First, we assume 
		that \( 0 < t_1 =\min\lbrace L_2 ,T\rbrace\).
		Now from Theorem \ref{PRIORI}, in particular \eqref{PRRTTioi} and \eqref{PRRTTioiio}, it follows for a constant $M_1$ that
		\begin{align*}
			\left\lbrace \omega: \sup_{\vert y\vert_{E_\alpha}\leq \rho}\sup_{\tau\in [0,t_0]}\left\vert \phi_{\mathbf{X}(\omega)}(0,\tau,y)\right\vert_{E_\alpha}> n\ \right\rbrace\bigsubseteq  \left\lbrace \omega: \tilde{N}([0,t_0], \chi, \mathbf{X}(\omega))> M_1n\ \right\rbrace.
		\end{align*}
		Now, we apply Theorem \ref{INTEFGA} to prove \eqref{UJJM854} on the interval \([0, t_0]\).~Furthermore, we extend this result to the entire time interval $[0,T]$ by a standard concatenation argument. First, we assume that \( L_2 < T \), otherwise the proof is complete. Then we assume that \( N_1 = \left\lfloor \frac{T}{L_2} \right\rfloor + 1 \) and define a sequence \( (t_m)_{0 \leq m \leq N_1} \) such that \( t_m - t_{m-1} = L_2 \) for \( 1 \leq m < N_1\), and \( t_{N_1} - t_{N_1-1} \leq L_2 \). Moreover, \( t_0 = 0 \) and \( t_{N_1} = T \). Then, applying \eqref{PRRTTioi}, we get the following inequality for every \( 1 \leq m \leq N_1 \)
		\begin{align}\label{54osdp}
			\begin{split}
				&\sup_{\tau\in [t_{m-1},t_m]}\left\vert \phi_{\mathbf{X}(\omega)}(0,\tau,y)\right\vert_{E_\alpha}\\&\quad\lesssim \left\vert \phi_{\mathbf{X}(\omega)}(0,t_{m-1},y)\right\vert_{E_\alpha}P_1\left(\mathbf{X}(\omega),[t_{m-1},t_m],\chi\right)+ P_{2}\left(\mathbf{X}(\omega),[t_{m-1},t_m],\chi\right).
			\end{split}
		\end{align}
		By setting \( t_{-1} = t_0 \) , we conclude from \eqref{54osdp} that
		\begin{align}\label{54osdp0}
			\begin{split}
				&\sup_{\tau\in [t_{m-1},t_m]}\left\vert \phi_{\mathbf{X}(\omega)}(0,\tau,y)\right\vert_{E_\alpha}\\&\quad\lesssim \sup_{\tau\in [t_{m-2},t_{m-1}]}\left\vert \phi_{\mathbf{X}(\omega)}(0,\tau,y)\right\vert_{E_\alpha}P_1\left(\mathbf{X}(\omega),[t_{m-1},t_m],\chi\right)+ P_{2}\left(\mathbf{X}(\omega),[t_{m-1},t_m],\chi\right).
			\end{split}
		\end{align}
		Now, since \( P_1, P_2 > 1 \),   we can recursively infer from \eqref{54osdp0} that
		\begin{align}\label{YN541a}
			\begin{split}
				&\sup_{\tau\in [0,T]}\left\vert \phi_{\mathbf{X}(\omega)}(0,\tau,y)\right\vert_{E_\alpha}=\max_{1\leq m\leq N_1}\sup_{\tau\in [t_{m-1},t_m]}\left\vert \phi_{\mathbf{X}(\omega)}(0,\tau,y)\right\vert_{E_\alpha}\\&\quad\lesssim \vert y\vert_{E_\alpha}\prod_{1\leq m\leq N_1}P_{1}(\mathbf{X}(\omega),[t_{m-1},t_m],\chi)+P_{2}(\mathbf{X}(\omega),[t_{N_1-1},t_{N_1}],\chi)\\&\qquad+\sum_{1\leq m<N_1}P_{2}(\mathbf{X}(\omega),[t_{m-1},t_m],\chi)\left(\prod_{m< k\leq N_1}P_{1}(\mathbf{X}(\omega),[t_{k-1},t_k],\chi)\right).
			\end{split}
		\end{align}
		Recall that \eqref{PRRTTioiio} holds on every interval $[s,t]$ with $t-s \leq L_2$. Thus, by substituting the bounds from \eqref{PRRTTioiio}, we can find another constant \( M_2 \), (which linearly depends on \( \frac{1}{N_1} \)), such that
		\begin{align}\label{YHAs52}
			\left\lbrace \omega: \sup_{\vert y\vert_{E_\alpha}\leq \rho}\sup_{\tau\in [0,T]}\left\vert \phi_{\mathbf{X}(\omega)}(0,\tau,y)\right\vert_{E_\alpha}> n\ \right\rbrace\bigsubseteq\bigcup_{1\leq m\leq N_1}\left\lbrace \omega: \tilde{N}([t_{m-1},t_m], \chi, \mathbf{X}(\omega))> M_2n\right\rbrace.
		\end{align}
		Now, we can apply Theorem \ref{INTEFGA} to each term on the right-hand side of \eqref{YHAs52}  to obtain \eqref{UJJM854} on the interval \([0, T]\). This proves the statement.
	\end{proof}

	\bibliographystyle{alpha}
	\bibliography{LINEAR_V2}
\end{document}